\documentclass[11pt,a4paper]{article}
\usepackage[top=2.5cm, bottom=2.5cm, left=2.5cm, right=2.5cm]{geometry}
\usepackage[latin1]{inputenc}
\usepackage{csquotes}
\usepackage{amsmath}
\usepackage{amsthm}
\usepackage{amsfonts}
\usepackage{amssymb}
\usepackage{graphics}
\usepackage{float}
\usepackage[hang,flushmargin]{footmisc} %nonindent footnote

\usepackage{amsmath,amsthm,amssymb,graphicx, multicol, array}
\usepackage{enumerate}
\usepackage{enumitem}
\usepackage{xcolor}
\usepackage[pagebackref,bookmarks,colorlinks,breaklinks]{hyperref}
\hypersetup{linkcolor=blue,citecolor=red,filecolor=blue,urlcolor=blue} 

\usepackage{tabto}

\numberwithin{equation}{section}

\usepackage{tikz}
\usepackage{pgfplots}

\usetikzlibrary{matrix}
\usepackage{mathrsfs}

\DeclareMathOperator{\ord}{ord}
\DeclareMathOperator{\Ord}{Ord}

\DeclareMathOperator{\lcm}{lcm}
\DeclareMathOperator{\tr}{Tr}
\DeclareMathOperator{\tchar}{char}

\newtheorem{thm}{Theorem}[section]
\newtheorem{lem}{Lemma}[section]
\newtheorem{prop}{Proposition}[section]
\newtheorem{conj}{Conjecture}[section]
\newtheorem{exa}{Example}[section]

\newtheorem{cor}{Corollary}[section]
\newtheorem{dfn}{Definition}[section]
\newtheorem{exe}{Exercise}[section]

\newcommand{\N}{\mathbb{N}}
\newcommand{\Z}{\mathbb{Z}}

\newcommand{\R}{\mathbb{R}}
\newcommand{\C}{\mathbb{C}}
\newcommand{\F}{\mathbb{F}}
\newcommand{\tP}{\mathbb{P}}

\usepackage{fancyhdr}
\usepackage{titletoc}
\let\LaTeXStandardTableOfContents\tableofcontents

\renewcommand{\tableofcontents}{%
	\begingroup%
	\renewcommand{\bfseries}{\relax}%
	\LaTeXStandardTableOfContents%
	\endgroup%
}%

\title{Small Primitive Normal Elements in Finite Fields}
\date{}
\author{N. A. Carella}

\begin{document}

	%\doublespacing
\thispagestyle{empty}
\date{}
	\maketitle

\vskip .25 in 
\begin{abstract}
Let $q=p^k$ be a prime power, let $\F_q$ be a finite field and let $n\geq2$ be an integer. This note investigates the existence small primitive normal elements in finite field extensions $\F_{q^n}$. It is shown that a small nonstructured subset $\mathcal{A}\subset \F_{q^n}$ of cardinality $\#\mathcal{A}\gg (\log q^n)  (\log\log q^n)^{1+\varepsilon}) $, where $\varepsilon>0$ is a small number, contains a primitive normal element. \let\thefootnote\relax\footnote{ \today \date{} 
		}
\end{abstract}
\setcounter{tocdepth}{1}
\tableofcontents
%SSSSSSSSSSSSSSSSSSSSSSSSSSSSSSSSSSSSSSSSSSSSSSSSSSSSSSSSSSSSSSSSSSSSSSSSSSSSSSSSSSS
%SSSSSSSSSSSSSSSSSSSSSSSSSSSSSSSSSSSSSSSSSSSSSSSSSSSSSSSSSSSSSSSSSSSSSSSSSSSSSSSSSSS
%SSSSSSSSSSSSSSSSSSSSSSSSSSSSSSSSSSSSSSSSSSSSSSSSSSSSSSSSSSSSSSSSSSSSSSSSSSSSSSSSSSS
%SSSSSSSSSSSSSSSSSSSSSSSSSSSSSSSSSSSSSSSSSSSSSSSSSSSSSSSSSSSSSSSSSSSSSSSSSSSSSSSSSSS
%SSSSSSSSSSSSSSSSSSSSSSSSSSSSSSSSSSSSSSSSSSSSSSSSSSSSSSSSSSSSSSSSSSSSSSSSSSSSSSSSSSS
%SSSSSSSSSSSSSSSSSSSSSSSSSSSSSSSSSSSSSSSSSSSSSSSSSSSSSSSSSSSSSSSSSSSSSSSSSSSSSSSSSSS
%SSSSSSSSSSSSSSSSSSSSSSSSSSSSSSSSSSSSSSSSSSSSSSSSSSSSSSSSSSSSSSSSSSSSSSSSSSSSSSSSSSS
\section{Introduction} \label{S4343PNEFF-I}\hypertarget{S4343PNEFF-I}
Let $q$ be a prime power and let $\F_{q^n}$ be a finite field extension of $\F_q$ of degree $[\F_{q^n}:\F_q]=n$. Two important representations of a finite field are the additive group and the multiplicative group. A normal element $\eta\in \F_{q^n}$ generates the finite field as an additive group 
\begin{equation} \label{eq4343PNEFF.050b}
\F_{q^n}=\{ \alpha=a_0\eta+ a_1\eta^q+a_2\eta^{q^2}+\cdots+a_{n-1}\eta^{q^{n-1}}:a_i\in \F_q\}
\end{equation}	
and a primitive element $\eta\in \F_{q^n}$ generates the finite field as a multiplicative group 
\begin{equation} \label{eq4343PNEFF.050d}
	\F_{q^n}^{\times}=\{ \alpha=\eta^k:k\in [0,q^n-2]\}.
\end{equation}	
While, a primitive normal element $\eta\in \F_{q^n}$ generates both groups. Various partial results on the existence of primitive element normal elements in finite fields were achieved in \cite{CL1952} and \cite{DH1968}. The first complete result was proved in \cite{LS1987}. Currently, there is a large literature on this topic and related concepts, see \cite{LS1987}, \cite{CN2005}, \cite{FR2024}, et alia. However, there is no literature on the existence of primitive normal elements in small subsets $\mathcal{A}\subset \F_{q^n}$ of cardinality $\#\mathcal{A}=O((\log q^n) (\log\log q^n)^{c}) $, where $c>0$ is a constant. The definition of small elements and small subsets are stated in \hyperlink{S7171CFNE.100B1} {Definition} \ref{dfn7171CFNE.100B1} \hyperlink{S7171CFNE.100B2} {Definition} \ref{dfn7171CFNE.100B2}. The corresponding counting function for the number of primitive normal elements is defined by
\begin{equation} \label{eq4343PNEFF.050j}
N_q(\mathcal{A})	=\#\{ \alpha \in \mathcal{A}: \alpha \text{ is primitive normal element}\}.
\end{equation}	
This note provides a new result in this direction.
\begin{thm} \label{thm4343PNEFF.050}\hypertarget{thm4343PNEFF.050}  Let $q=p^k$ be a prime power and $\varepsilon>0$ be a small number. Let 
$\mathcal{A}\subset \F_{q^n}$ be a nonstructured subset of cardinality 
\begin{equation}\label{eq4343PNEFF.050j1}
	\#\mathcal{A}\gg (\log q^n) (\log\log q^n)^{1+\varepsilon}.
\end{equation}	
Then, as $q^n\to \infty$, the subset $\mathcal{A}$ contains primitive normal elements, uniformly for all $q\geq2$, $k\geq1$ and $n\geq2$. In particular, for large $q^n> q^2$, the weighted counting function for the number of primitive normal elements has the asymptotic lower bound   
\begin{equation} \label{eq4343PNEFF.050i3}
N_q(\mathcal{A})\gg(\log\log q^n)^{\varepsilon}.
	\end{equation}	
\end{thm}	

Equivalently, this result provides the sharper benchmark on the deterministic complexity of determining a primitive normal elements in finite fields. The algorithms for primitive normal elements investigated in {\color{red}\cite[Section 6]{VG1990}} have at best  deterministic exponential time complexity of $O(n^cq^{1/2})$ arithmetic operations, where $c>0$ is a constant.

%%%%%%%%%%%%%%%%%%%%%%%%%%%%%%%%%%%%%%%%%%%%%%%%%%%%%%%%%%%%%%%%%%%%%%
%%%%%%%%%%%%%%%%%%%%%%%%%%%%%%%%%%%%%%%%%%%%%%%%%%%%%%%%%%%%%%%%%%%%%%
%%%%%%%%%%%%%%%%%%%%%%%%%%%%%%%%%%%%%%%%%%%%%%%%%%%%%%%%%%%%%%%%%%%%%%
%%%%%%%%%%%%%%%%%%%%%%%%%%%%%%%%%%%%%%%%%%%%%%%%%%%%%%%%%%%%%%%%%%%%%%

%CCCCCCCCCCCCCCCCCCCCCCCCCCCCCCCCCCCCCCCCCCCCCCCCCCCCCCCCCCCCCCCCCCCCCCCCCCCCCCC
%CCCCCCCCCCCCCCCCCCCCCCCCCCCCCCCCCCCCCCCCCCCCCCCCCCCCCCCCCCCCCCCCCCCCCCCCCCCCCCC
\begin{cor} \label{cor4343PNEFF.100}\hypertarget{cor4343PNEFF.100} Let $q=p^k$ be a prime power and let $n\geq2$. The search algorithm for a primitive normal element in a finite field $\F_{q^n}$ has a deterministic time complexity of $O\left( (log q^n)^{c} \right) $ arithmetic operations, where $c>2$ is a constant. 
\end{cor}

%%%%%%%%%%%%%%%%%%%%%%%%%%%%%%%%%%%%%%%%%%%%%%%%%%%%%%%%%%%%%%%%%%%%%%%%%%%%%%%%%%%
The counting function for the number of primitive normal elements in a random subset $\mathcal{R}\subset\F_{q^n}$, which is not a subfield nor a linear subspace, is expected to have an asymptotic formula of the form  
\begin{equation} \label{eq4343PNEFF.050i7}
	N_q(\mathcal{R})	=\delta_q(\mathcal{R})\frac{\varphi (q^n-1)}{q^{n}}\cdot \frac{\Phi(x^n-1)}{q^{n}}\#\mathcal{R}\left (1+o(1)\right ),
\end{equation}	where $\delta_q(\mathcal{R})>0$ the density of primitive normal elements depends on the subset $\mathcal{R}$. In addition, a small nonstructured subset 
$\mathcal{S}\subset \F_{q^n}$ of cardinality 
\begin{equation}\label{eq4343PNEFF.050i11}
\#\mathcal{S}\gg (\log q^n)^{c}
\end{equation}
is expected to contain primitive normal elements. In the meantime, these remain as open problems. Merging the ideas of Artin primitive root conjecture and the Lenstra-Schoof theorem for primitive normal elements leads to the followings expectation.  

\begin{conj}\label{conj7171PNEI.300}\hypertarget{conj7171PNEI.300} Let $q$ be a prime power and let $\F_q$ be a finite field. Then, there exists infinitely many finite field extensions $\F_{q^n}$ such that a fixed element $\alpha\in \F_{q^n}-\bigcup_{d\mid n}\F_{q^d}$ such that $\alpha\ne0, \pm1,\beta^2$ and trace $\tr(\alpha)\ne0$ is primitive normal as $n\to\infty$. 
\end{conj}
The asymptotic formula for the counting function of the number of finite field extensions should be of the form  
\begin{align} \label{eq4343PNEFF.050i9}
	N_q(\alpha,x)	&=\#\{n\leq x:\alpha \in \F_{q^n} \text{ is primitive normal}\}\\[.3cm]
	&=\delta_q(\alpha)\frac{\varphi (q^n-1)}{q^{n}}\cdot \frac{\Phi(x^n-1)}{q^{n}}\left (1+o(1)\right )x,\nonumber
\end{align}	
where $\delta_q(\alpha)>0$ the density of finite field extensions depends on the fixed element $\alpha\in \F_{q^n}$, this similar to the result in {\color{red}\cite[Theorem 2]{CL1952A}}.	\\

The proof of \hyperlink{thm4343PNEFF.050}{Theorem} \ref{thm4343PNEFF.050} appears in \hyperlink{S4343PNEFF-T}{Section} \ref{S4343PNEFF-T}. The background supporting materials appears in \hyperlink{S7171CFNE-DNE}{Section} \ref{S7171CFNE-DNE} to \hyperlink{S47171SNEFR-C}{Section} \ref{S47171SNEFR-C}.
%SSSSSSSSSSSSSSSSSSSSSSSSSSeSSSSSSSSSSSSSSSSSSSSSSSSSSSSSSSSSSSSSSSSSSSSSSSSSSSSSSSSS
%SSSSSSSSSSSSSSSSSSSSSSSSSSSSSSSSSSSSSSSSSSSSSSSSSSSSSSSSSSSSSSSSSSSSSSSSSSSSSSSSSSS
%SSSSSSSSSSSSSSSSSSSSSSSSSSSSSSSSSSSSSSSSSSSSSSSSSSSSSSSSSSSSSSSSSSSSSSSSSSSSSSSSSSS
%SSSSSSSSSSSSSSSSSSSSSSSSSSSSSSSSSSSSSSSSSSSSSSSSSSSSSSSSSSSSSSSSSSSSSSSSSSSSSSSSSSS
%SSSSSSSSSSSSSSSSSSSSSSSSSSSSSSSSSSSSSSSSSSSSSSSSSSSSSSSSSSSSSSSSSSSSSSSSSSSSSSSSSSS
%SSSSSSSSSSSSSSSSSSSSSSSSSSSSSSSSSSSSSSSSSSSSSSSSSSSSSSSSSSSSSSSSSSSSSSSSSSSSSSSSSSS
%SSSSSSSSSSSSSSSSSSSSSSSSSSSSSSSSSSSSSSSSSSSSSSSSSSSSSSSSSSSSSSSSSSSSSSSSSSSSSSSSSSS
\section{Notation, Definitions and Basic Concepts} \label{S7171CFNE-DNE}\hypertarget{S7171CFNE-DNE}
The sets $\N=\{0,1,2,3,\ldots\}$, $\Z=\{\ldots, -3,-2,-1, 0,1,2,3,\ldots\}$ and $\tP=\{2,3,5,\ldots\}$ denote the sets of natural numbers, the sets of integers and the set of prime numbers respectively. The letters $p,q,r\in\mathbb{P}$ usually denote arbitrary prime numbers, and the letters $a,b,c,k,m,n\in \mathbb{N}$ usually denote arbitrary integers.

\begin{enumerate}
	\item The symbol $\log x=\ln x$ denotes the natural logarithm.
	\item Let $f, g: [x_0,\infty]\longrightarrow \R$ be a pair of functions, and assume $g(x)>0$. The big O notation is defined by
	\begin{equation}\label{eq2266N.120}
		f(x)=O(g(x))  \quad  \Longleftrightarrow  \quad|f(x)|\leq c g(x)
	\end{equation}
	for some constant $c>0$ as $x\to\infty$.
	\item The symbol $\ll$ is defined by
	\begin{equation}\label{eq2266N.100}
		f(x)\ll g(x)  \quad  \Longleftrightarrow  \quad|f(x)|\leq c g(x)
	\end{equation}
	for some constant $c>0$ as $x\to\infty$. 
	\item The symbol $\gg$ is defined by
	\begin{equation}\label{eq2266N.100}
		f(x)\gg g(x)  \quad  \Longleftrightarrow  \quad|f(x)|\geq c g(x)
	\end{equation}
	for some constant $c>0$ as $x\to\infty$. 
		\item The symbol $f\asymp g$ is defined by
	\begin{equation}\label{eq2266N.100}
		c_0g(x)\leq f(x)\leq c_1g(x)  \qquad  \text{ and }\qquad c_2g(x)\leq f(x)\leq c_3g(x)
	\end{equation}
	for some constants $c_0,c_1c_2,c_3>0$ as $x\to\infty$. 
\end{enumerate}

%SSSSSSSSSSSSSSSSSSSSSSSSSSSSSSSSSSSSSSSSSSSSSSSSSSSSSSSSSSSSSSSSSSSSSSSSSSSSSSSSSSSS
%SSSSSSSSSSSSSSSSSSSSSSSSSSSSSSSSSSSSSSSSSSSSSSSSSSSSSSSSSSSSSSSSSSSSSSSSSSSSSSSSSSSS
\subsection{Small Elements and Subsets}	
There are various methods for defining orders and topologies on the rings of polynomials $\F_q[x]$. Some of the possible metrics, orders and topologies are induced by  
\begin{itemize}
\item the Hamming weight of polynomials,
\item the height of polynomials,
\item the nonArchimedean valuation of polynomials.
\end{itemize}
The Hamming weight function is widely used in complexity theory and cryptography, the height function is widely used in the number theory and the last one is widely used in nonArchimedean analysis. The Hamming weight function and the height function induce equivalent metrics and equivalent discrete topologies. While the nonArchimedean valuation induces a nonArchimedean metric and a nondiscrete topology on the ring of polynomials $\F_q[x]$, more generally on the quotient ring $\F_q((x))$. \\

Let $\F_{q^n}\cong \F_q[x]/(f(x))$ be a representation of the finite field, where $ f(x)\in \F_q[x]$ is an irreducible polynomial of degree $\deg f=n$, and let $N:\F_q\longrightarrow \F_p$ be a norm function.\\

\begin{dfn}\label{dfn7171CFNE.100B1}\hypertarget{S7171CFNE.100B1} {\normalfont The Hamming weight of a polynomial $r(x)\in \F_q[x]$ of degree $d=\deg r$ is defined by 
\begin{align}
	w(r)=\#\{i\ne0:r(x)=a_0+a_1x+a_2x^2+\cdots+a_{d}x^{d}\}.
\end{align} 
 An element $\alpha\in\F_{q^n}$ is called \textit{small} if it has a polynomial representation of Hamming weight $w(r)\ll (\log\log q^n)^{c}$, where $c>2$ is constant. A \textit{small Hamming subset} of elements in a finite field is a subset of small elements.
	}
\end{dfn}	
\begin{prop} \label{prop7171CFNE.200C}\hypertarget{prop7171CFNE.200C}  The Hamming weight induces a metric function on the ring $\F_q[x]$. If $r,s,u\in \F_q[x]$ then
	\begin{enumerate}[font=\normalfont, label=(\roman*)]
		\item $\displaystyle w(r,s)\geq0$\tabto{9cm}nonnegative,
		\item $\displaystyle w(r,s)=0$ if and only if $r=s$\tabto{9cm}unique identity,
		\item $\displaystyle w(r,s)=w(s,r)$\tabto{9cm}symmetry,
		\item $\displaystyle w(r,s)\leq w(s,u)+w(u,s)$\tabto{9cm}triangle inequality.
	\end{enumerate}
\end{prop}
The order defined by the Hamming weight of polynomials seems to be independent on the characteristic $p\geq2$ of the prime finite field $\F_p\subseteq \F_q$. The order of set inclusion with respect to the Hamming metric is a nonsymmetric equivalent relation of ring $\F_q[x]$. A small Hamming subset is an open set in the discrete topology of the polynomials ring $\F_q[x]$ corresponding to the Hamming weight $w(r)$. \\

\begin{dfn}\label{dfn7171CFNE.100B2}\hypertarget{S7171CFNE.100B2} {\normalfont The height of a polynomial $r(x)\in \F_q[x]$ of degree $d=\deg r$ is defined by 
		\begin{align}
			h(r)=\max\{N(a_i):r(x)=a_0+a_1x+a_2x^2+\cdots+a_{d}x^{d}\}.
		\end{align} 
		An element $\alpha\in\F_{q^n}$ is called \textit{small} if it has a polynomial representation of height $h(r)\ll (\log\log q^n)^{c}$, where $c>2$ is constant. A \textit{small subset} of elements in a finite field is a subset of small elements.
	}
\end{dfn}	
\begin{prop} \label{prop7171CFNE.200D}\hypertarget{prop7171CFNE.200D}  The height induces a metric function on the ring $\F_q[x]$. If $r,s,u\in \F_q[x]$ then
	\begin{enumerate}[font=\normalfont, label=(\roman*)]
		\item $\displaystyle h(r,s)\geq0$\tabto{9cm}nonnegative,
		\item $\displaystyle h(r,s)=0$ if and only if $r=s$\tabto{9cm}unique identity,
		\item $\displaystyle h(r,s)=h(s,r)$\tabto{9cm}symmetry,
		\item $\displaystyle h(r,s)\leq h(s,u)+h(u,s)$\tabto{9cm}triangle inequality.
	\end{enumerate}
\end{prop}
The order defined by the height of polynomials seems to be primarily dependent on the characteristic $p\geq2$ of the prime finite field $\F_p\subset \F_q$. The simplest nontrivial definition of an order on $\F_q[x]$ by the height metric appears to be for $p\geq p_0$, where $p_0\geq5$.\\

The order of set inclusion with respect to the height metric is a nonsymmetric equivalent relation of ring $\F_q[x]$. A small subset is an open set in the discrete topology of the polynomials ring $\F_q[x]$ corresponding to the height metric $h(r)$.

\begin{dfn}\label{dfn7171CFNE.100S}\hypertarget{S7171CFNE.100S} {\normalfont  A \textit{small nonstructure subset} of elements in a finite field is a subset of small elements of height $H$ which do not have the algebraic properties of as subgroups nor linear subspaces of $\F_{q^n}$.
	}
\end{dfn}	
There many ways of constructing nonstructured subsets in finite fields, using either the Hamming metric or the height metric.
\begin{dfn}\label{dfn7171CFNE.100K}\hypertarget{S7171CFNE.100K} {\normalfont Let $H>0$ be a real number. A nonstructured subset $\mathcal{A}=\mathcal{A}_d(H)\subset \F_{q^n}$ is defined by
		\begin{equation} \label{eq4343PNEFF.050K1}
			\mathcal{A}_d(H)	=\{ \alpha =a_0+a_1x+a_2x^2+\cdots+a_{d}x^{d}:(a_0,\ldots,a_d)\in \mathcal{B}_d(H)\},
		\end{equation}	
		where the set of relatively prime coefficients of height $H$ is defined by
		\begin{equation} \label{eq4343PNEFF.050K2}
			\mathcal{B}_d(H)	=\{ (a_0,\ldots,a_d)\in \Z^{d+1}:\gcd(a_0,\ldots,a_d)=1 \text{ and }|a_i|\leq H\}.
		\end{equation}	
		where $2\leq d<n$ is a small integer. 
	}
\end{dfn}	
For large $q^n$, the uniform distribution of the primitive normal elements in the finite field $\F_{q^n}$ guaranteed that a small nonstructured subset as the simple ball in \hyperlink{exa7171SBS.100Q1}{Example} \ref{exa7171SBS.100Q1} contains a primitive normal element.  
\begin{exa}\label{exa7171SBS.100Q1}\hypertarget{exa7171SBS.100Q1} {\normalfont Let $q=23$ and let $r(x)\in \F_q[x]$ be a fixed polynomial such that $r(0)=0$. A simple ball of radius $H=3$ centered around $r(x)$ with respect to the Hamming metric consists of the subset of polynomials  
		\begin{align}\label{eq4343PNEFF.050D7}
			\mathbb{B}_{3}(r(x))&=\{s(x)=r(x)+c:H(c)\leq 3\}\\[.2cm]
			&=\{r(x)+1,\;r(x)+ 2,\; r(x)+3,\; r(x)+4,\; r(x)+5,\; r(x)+6,\; r(x)+7, \nonumber\\[.2cm]
			&\hskip .25 in \; r(x)+8, \; r(x)+9,\;r(x)+10,\; r(x)+11,\;r(x)+12,\; r(x)+11,\;r(x)+14\}\nonumber,
		\end{align}
where $c\in\F_q$ and its the Hamming metric is 
\begin{align}\label{eq4343PNEFF.050D8}
H(c)&=H(s(x)-r(x))\\[.2cm]
	&=\#\{i\geq 0:c=c_4c_3c_2c_1c_0 \text{ and }c_i=0,1\}\leq 3\nonumber,
\end{align} where $c=c_4c_3c_2c_1c_0$ is the binary expansion of $c\in \F_p$. More generally, a simple ball of radius $H>0$ centered at the polynomial $r(x)$ is defined by
\begin{align}\label{eq4343PNEFF.050D9}
	\mathbb{B}_{H}(r(x))&=\{s(x)\in \F_q[x]:H(s(x)-r(x))\leq H\}\\[.2cm]
	&=\{s(x)=r(x)+c:H(c)\leq H\}\nonumber.
\end{align}
For any $H\leq n$, this subset does not has a one-to-one correspondence with any subfield $\F_{23^d}\subseteq \F_{23^n}$. Thus, the ball $\mathbb{B}_{H}(r(x))$ of radius $H>0$ is a nonstructured subset of $ \F_{23^n}$. 
	}
\end{exa}
\begin{exa}\label{exa7171SBS.100Q2}\hypertarget{exa7171SBS.100Q2}{\normalfont Let $q=2,3$.  
		A unit sphere in $\F_q[x]$ consisting of polynomials of degree $d=3$ is the subset
		\begin{align}\label{eq4343PNEFF.050D3}
			\mathbb{S}_{1,3}(\F_q)&=\{r(x)=a_0+a_1x+\cdots+a_3x^{3}:h(r)=\max\{|N(a_i)|\}=1\}\\[.2cm]
			&=\{1,\;x,\; x+1,\; x^2,\; x^2+1,\; x^2+x,\; x^2+x+1,\; x^3,\; x^3+1,\;x^3+x,\;  \nonumber\\[.2cm]
			&\hskip .25 in  x^3+x^2,\;x^3+x+1,\; x^3+x^2+1,\;x^3+x^2+x,\;  x^3+x^2+x+1\}\nonumber,
		\end{align}
		where the norm $N(a)=a^{(q^d-1)/q-1}=\pm1$. This subset has a one-to-one correspondence with the subfield $\F_{2^d}$. Thus, the unit sphere $\mathbb{S}_{1,d}(\F_2)= \F_{2^d}$ is a structured subset of $ \F_{2^n}$ for any $d\mid n$. The same phenomenon is true for $\F_{3^d}\subset \F_{3^n}$.
	}
\end{exa}
\begin{exa}\label{exa7171SBS.100Q3}\hypertarget{exa7171SBS.100Q3}{\normalfont Let $q=19$ and $H=1$. 
		The subset of polynomials of degree 2 and height $h(r)=1$ in $\F_q[x]$ is  
		\begin{align}\label{eq4343PNEFF.050D2}
			\mathcal{A}_d(H)&=\{r(x)=a_0+a_1x+a_2x^{2}:h(r)=\max\{|N(a_i)|\}=1\}\\[.2cm]
			&=\{\pm1,\;\pm x,\; \pm x\pm 1,\; \pm x^2,\; \pm x^2\pm1,\; \pm x^2\pm x,\; \pm x^2\pm x\pm1\}\nonumber.
		\end{align}
		Clearly, this subset does not has a one-to-one correspondence with the finite field $\F_{19^d}$. Thus, the small subset $\mathcal{A}_d(H)\ne \F_{19^d}$ is a nonstructured subset of $ \F_{19^n}$ for any $d\mid n$.
	}
\end{exa}
%SSSSSSSSSSSSSSSSSSSSSSSSSSSSSSSSSSSSSSSSSSSSSSSSSSSSSSSSSSSSSSSSSSSSSSSSSSSSSSSSSSSS
%SSSSSSSSSSSSSSSSSSSSSSSSSSSSSSSSSSSSSSSSSSSSSSSSSSSSSSSSSSSSSSSSSSSSSSSSSSSSSSSSSSSS
\subsection{Norm and Trace Functions}	
\begin{dfn}\label{dfn7171CFNE.200t}\hypertarget{S7171CFNE.200t} {\normalfont Let $q=p^k$ be a prime power and let $\F_{q^n}$ be a finite field extension. The absolute trace and norm functions
		\begin{align}\label{eq7171CFNE.200t}
			\tr&:	\F_{q^n}\longrightarrow \F_p,\\[.2cm]
			N&:	\F_{q^n}\longrightarrow \F_p\nonumber
		\end{align}			
		are defined by
		\begin{align}\label{eq7171CFNE.200t2}
			\tr(\alpha)&=\alpha+\alpha^{p}+\alpha^{p^2}+\cdots+\alpha^{p^{kn-1}},\\[.2cm]
			N(\alpha)&=\alpha^{p+p^2+\cdots+p^{kn-1}}\nonumber
		\end{align}
		respectively.
	}
\end{dfn}	

%SSSSSSSSSSSSSSSSSSSSSSSSSSSSSSSSSSSSSSSSSSSSSSSSSSSSSSSSSSSSSSSSSSSSSSSSSSSSSSSSSSSSSS
\subsection{Linearized Polynomials}	
\begin{dfn}\label{dfn7171CFNE.200m}\hypertarget{S7171CFNE.200m} {\normalfont The linearized polynomial $L_f(x)$ associated to a polynomial $f(x)=a_0+a_1x+a_2x^2+\cdots+a_nx^n\in\F_q[x]$ is defined by 
		\begin{equation}\label{eq7171CFNE.200m}
			L_f(x)=a_0x+a_1x^p+a_2x^{p^2}+\cdots+a_nx^{p^{n}}.
		\end{equation}			
	}
\end{dfn}

The linearized polynomial is often denoted by $L_r(\alpha)=r\circ\alpha$ or by $L_r(\alpha)=r(x)(\sigma)(\alpha)$ or by some other notation.\\

The Frobenious map $\sigma :\F_{q^n}\longrightarrow \F_{q}$ on a finite field is usually defined by the power map $\sigma(x)=x^q$, and a $k$-fold iterate is written as $\sigma_k(x)=x^{q^k}$. The finite field is the fixed field of the Frobenious map action on the algebraic closure $\overline{\F_{q}}=\bigcup_{n\to\infty} \F_{q^n}$.
%SSSSSSSSSSSSSSSSSSSSSSSSSSSSSSSSSSSSSSSSSSSSSSSSSSSSSSSSSSSSSSSSSSSSSSSSSSSSSSSSSSSSSS
\subsection{Additive Characters and Additive Orders}
\begin{dfn}\label{dfn2727PNT.200C}\hypertarget{dfn2727PNT.200C} {\normalfont Let $q=p^k$ be a prime power and let $\F_q$ be a finite field. Assume $f(x)\in \F_q[x]$ is irreducible of degree $\deg f=n$. The \textit{additive order} of an element $\alpha\in \F_q[x]/f(x)$ is defined by
		$$\Ord_{q}\alpha =\min\{r(x)\mid x^n-1:r\circ\alpha=0\}.$$		
		An element of maximal order $\Ord_{q}\alpha=x^n-1$ is called a \textit{normal element}.}
\end{dfn}			
\begin{dfn}\label{dfn7171CFNE.200i}\hypertarget{S7171CFNE.200i} {\normalfont An additive character $\psi:\F_{q^n}\longrightarrow \C$ is an exponential function of the form 
		\begin{equation}\label{eq7171CFNE.200i}
			\psi(\alpha)=e^{i2 \pi  \frac{\tr(c\alpha)}{p}},
		\end{equation}			
		where $c\in\F_{q^n}$ is a constant. The trivial additive character is denoted by $\psi_0(\alpha)=1$.		
	}
\end{dfn}	
\begin{dfn}\label{dfn7171CFNE.200j}\hypertarget{S7171CFNE.200j} {\normalfont The additive order $\Ord_q\psi=d(x)$ of an additive character $\psi$ is a divisor $d(x)\mid x^n-1$ is a polynomial of degree $\deg d\mid n$. In particular, the trivial character has additive order $\Ord_q\psi_0=1$ and a primitive character has additive order $\Ord_q\psi_c=x^n-1$.
	}
\end{dfn}

%SSSSSSSSSSSSSSSSSSSSSSSSSSSSSSSSSSSSSSSSSSSSSSSSSSSSSSSSSSSSSSSSSSSSSSSSSSSSSSSSSSSSSS
%SSSSSSSSSSSSSSSSSSSSSSSSSSSSSSSSSSSSSSSSSSSSSSSSSSSSSSSSSSSSSSSSSSSSSSSSSSSSSSSSSSSSSS
\subsection{Multiplicative Characters and Multiplicative Orders}	
\begin{dfn}\label{dfn2727PED.200P}\hypertarget{dfn2727PED.200P} {\normalfont Let $q=p^k$ be a prime power and let $\F_{q^n}$ be a finite field. The \textit{multiplicative order} $\ord_{q} \alpha$ of an element $\alpha\in\F_{q^n}$ is defined by 
		\begin{equation}\label{eq2727PET.150d}
			\ord_{q}\alpha =\min\{d\mid q^n-1:\alpha^d=1\}.
		\end{equation}	
		An element of maximal order $\ord_{q}\alpha=q^n-1$ is called a \textit{primitive element}.}
\end{dfn}
\begin{dfn}\label{dfn7171CFPE.200i}\hypertarget{S7171CFPE.200i} {\normalfont A multiplicative character $\chi:\F_{q^n}\longrightarrow \C$ is an exponential function of the form 
		\begin{equation}\label{eq7171CFPE.200i}
			\chi(\alpha)=e^{i2 \pi  \frac{N(c\alpha)}{q^n-1}},
		\end{equation}			
		where $N(\alpha)$ is the norm function and $c\in\F_{q^n}$ is a constant.. The trivial multiplicative character is denoted by $\chi_0(\alpha)=1$.		
	}
\end{dfn}	
\begin{dfn}\label{dfn7171CFPE.200j}\hypertarget{S7171CFNE.200j} {\normalfont The multiplicative order $\ord_q\chi=d$ of a multiplicative character $\chi$ is a divisor $d\mid q^n-1$. In particular, the trivial multiplicative character has  multiplicative order $\ord_q\chi_0=1$ and a primitive  multiplicative character has  multiplicative order $\ord_q\psi_c=q^n-1$.
	}
\end{dfn}

\begin{lem}\label{lem7171CFNE.200A}\hypertarget{lem7171CFNE.200A}  The set of additive characters forms an additive group $\mathscr{G}_+=\{\psi_c:c\in \F_{q^n}\}$ of cardinality $\#\mathscr{G}_+=q^n$ characters, and any subgroup $\mathscr{H}_+\subset \mathscr{G}_+$ contains $\#\mathscr{H}_+=\Psi_q(d(x))=q^{\deg d}$ characters, where both $d(x)\mid x^n-1$ and $d\mid n$.
\end{lem}	
	
%%%%%%%%%%%%%%%%%%%%%%%%%%%%%%%%%%%%%%%%%%%%%%%%%%%%%%

%SSSSSSSSSSSSSSSSSSSSSSSSSSSSSSSSSSSSSSSSSSSSSSSSSSSSSSSSSSSSSSSSSSSSSSSSSSSSSSSSSSSSSS
%SSSSSSSSSSSSSSSSSSSSSSSSSSSSSSSSSSSSSSSSSSSSSSSSSSSSSSSSSSSSSSSSSSSSSSSSSSSSSSSSSSSSSS

\subsection{Normal Element Test}	

A survey of some normality test algorithms is given in \cite{VG1990}.
\begin{lem} \label{lem2727PNT-I}\hypertarget{lem2727PNT-I}  {\normalfont (Normal element test.)} Let $r(x) \mid x^n-1$ and let $(x^{n}-1)/r(x)=a_0+a_1x+a_2x^2+\cdots +a_dx^d\in\F_q[x]$.	An element $\alpha \in \F_{q^n}$ is normal if and only if 
\begin{equation}
	\left( \frac{x^{n}-1}{r(x)}\right) (\sigma)(\alpha)=\sum_{0\leq k\leq d}\sigma_k(\alpha)-\alpha=\sum_{0\leq k\leq d}a_kx^{kq^k}-\alpha\ne0 
\end{equation}
for all irreducible factors $r(x) \mid x^n-1$.
\end{lem}
Normal bases also have a deterministic test based on linear algebra. This an important test because it is independent of the irreducible factorization of the polynomial $x^n-1\in \F[x]$. 

\subsection{Primitive Element Test}

\begin{lem} \label{lem2727PET-I}\hypertarget{lem2727PET-I}  {\normalfont (Primitive element test)} An element $\alpha \in \F_{q^n}$ is a primitive if and only if 
	\begin{equation}\label{eq2727PET.150g}
		\alpha^{(q^n-1)/p} -1\ne 0
	\end{equation}
	for all prime divisors $p \mid q^n-1$.
\end{lem}	
Unlike, normal elements, there is no primitive root test based on linear algebra. Currently, except for some special cases, all the versions of the primitive root test depend on the prime factorization of $q^n-1$.		
%SSSSSSSSSSSSSSSSSSSSSSSSSSSSSSSSSSSSSSSSSSSSSSSSSSSSSSSSSSSSSSSSSSSSSSSSSSSSSSSSSSS
%SSSSSSSSSSSSSSSSSSSSSSSSSSSSSSSSSSSSSSSSSSSSSSSSSSSSSSSSSSSSSSSSSSSSSSSSSSSSSSSSSSS
%SSSSSSSSSSSSSSSSSSSSSSSSSSSSSSSSSSSSSSSSSSSSSSSSSSSSSSSSSSSSSSSSSSSSSSSSSSSSSSSSSSS
\subsection{Fundamental Identities} \label{S2727CEES-IA}\hypertarget{S2727CEES-IA}
The set of additive characters is denoted by $\mathscr{C}_a(q^n)=\{\psi:\F_{q^n}\longrightarrow\C\}$ and the set of multiplicative characters is denoted by $\mathscr{C}_m(q^n)=\{\chi:\F_{q^n}^{\times}\longrightarrow\C\}$ 

%LLLLLLLLLLLLLLLLLLLLLLLLLLLLLLLLLLLLLLLLLLLLLLLLLLLLLLLLLLLLLLLLLLLLLLLLLLLLLLLLLLLLLLL
%LLLLLLLLLLLLLLLLLLLLLLLLLLLLLLLLLLLLLLLLLLLLLLLLLLLLLLLLLLLLLLLLLLLLLLLLLLLLLLLLLLLLLLL
\begin{lem}   \label{lem2727CEES.600}\hypertarget{lem2727CEES.600} Let $q=p^k$ be a prime power and let $\F_{q^n}$ be a finite field. If $\tau\in \F_{q^n}$ is a primitive root,  then the multiplicative characters $\chi $ and the additive characters $\psi $ of $\F_{q^n}$ have the representations.
	\begin{enumerate} [font=\normalfont, label=(\roman*)]
		\item  
		$\displaystyle 	
		\chi(\alpha)=e^{i2\pi b\frac{N(\alpha)}{q^n-1}},$\tabto{4cm}for some $b\in [1, q^n-1]$, multiplicative character.\\
		
		\item  $\displaystyle 	
		\chi(\alpha)=e^{i2\pi b\frac{\tr(\alpha)}{p}},$\tabto{4cm}for some $b\in [0, q^n-1]$, additive identity. 
	\end{enumerate}
\end{lem} 
\begin{proof}[\textbf{Proof}]The absolute norm and trace functions in characteristic $\tchar  p$ are given by  $N(\alpha)$ and $\tr(\alpha)$, see {\color{red}\cite[Characters Sums]{LN1997}}.
\end{proof} 

%LLLLLLLLLLLLLLLLLLLLLLLLLLLLLLLLLLLLLLLLLLLLLLLLLLLLLLLLLLLLLLLLLLLLLLLLLLLLLLLLLLLLLLL
%LLLLLLLLLLLLLLLLLLLLLLLLLLLLLLLLLLLLLLLLLLLLLLLLLLLLLLLLLLLLLLLLLLLLLLLLLLLLLLLLLLLLLLL
\begin{lem}   \label{lem2727CEES.650A}\hypertarget{lem2727CEES.650A} Let $q=p^k$ be a prime power and let $\F_{q^n}$ be a finite field. If $\psi $ is an additive character and $\chi $ is a multiplicative of $\F_{q^n}$, then the following finite Fourier representations hold.
	\begin{enumerate} [font=\normalfont, label=(\roman*)]
		\item $\displaystyle \psi(\alpha)=\frac{1}{q^n-1}\sum _{\chi \in \mathscr{C}_m(q^n)}\chi(\alpha)G(\overline{\chi},\psi),$\tabto{9cm}Additive identity. \\
		
		\item  
		$\displaystyle 	\chi(\alpha)=\frac{1}{q^n}\sum _{\psi \in \mathscr{C}_a(q^n)}\psi(\alpha)G(\chi,\overline{\psi}),$\tabto{9cm}Multiplicative identity.		
	\end{enumerate}
\end{lem} 
\begin{proof}[\textbf{Proof}]These are the finite Fourier transforms of the functions $\chi$ and $\psi$ respectively, see {\color{red}\cite[Characters Sums]{LN1997}}.
\end{proof}

%LLLLLLLLLLLLLLLLLLLLLLLLLLLLLLLLLLLLLLLLLLLLLLLLLLLLLLLLLLLLLLLLLLLLLLLLLLLLLLLLLLLLLLL
%LLLLLLLLLLLLLLLLLLLLLLLLLLLLLLLLLLLLLLLLLLLLLLLLLLLLLLLLLLLLLLLLLLLLLLLLLLLLLLLLLLLLLLL
\begin{lem}   \label{lem2727CEES.650GS}\hypertarget{lem2727CEES.650GS} Let $q=p^k$ be a prime power and let $\F_{q^n}$ be a finite field. If $\psi $ is an additive character and $\chi $ is a multiplicative of $\F_{q^n}$, then the Gauss sum satisfies the followings.
	\begin{equation}\label{eq2727CEES.650Ik}
		G(\chi,\psi)=\left \{
		\begin{array}{ll}
			q^n-1 & \text{   \normalfont if } \chi=1\text{ and }\psi=1,  \\[.3cm]
			0 & \text{   \normalfont if } \chi\ne1\text{ and }\psi=1, \\[.3cm]
			-1 & \text{   \normalfont if } \chi=1\text{ and }\psi\ne1,
		\end{array} \right .
	\end{equation}
\end{lem} 
\begin{proof}[\textbf{Proof}]These are the finite Fourier transforms of the functions $\chi$ and $\psi$ respectively, see {\color{red}\cite[Characters Sums]{LN1997}}.
\end{proof}

%SSSSSSSSSSSSSSSSSSSSSSSSSSSSSSSSSSSSSSSSSSSSSSSSSSSSSSSSSSSSSSSSSSSSSSSSSSSSSSSSSSS
%SSSSSSSSSSSSSSSSSSSSSSSSSSSSSSSSSSSSSSSSSSSSSSSSSSSSSSSSSSSSSSSSSSSSSSSSSSSSSSSSSSS
%SSSSSSSSSSSSSSSSSSSSSSSSSSSSSSSSSSSSSSSSSSSSSSSSSSSSSSSSSSSSSSSSSSSSSSSSSSSSSSSSSSS
%SSSSSSSSSSSSSSSSSSSSSSSSSSSSSSSSSSSSSSSSSSSSSSSSSSSSSSSSSSSSSSSSSSSSSSSSSSSSSSSSSSS
%SSSSSSSSSSSSSSSSSSSSSSSSSSSSSSSSSSSSSSSSSSSSSSSSSSSSSSSSSSSSSSSSSSSSSSSSSSSSSSSSSSS
%SSSSSSSSSSSSSSSSSSSSSSSSSSSSSSSSSSSSSSSSSSSSSSSSSSSSSSSSSSSSSSSSSSSSSSSSSSSSSSSSSSS
%SSSSSSSSSSSSSSSSSSSSSSSSSSSSSSSSSSSSSSSSSSSSSSSSSSSSSSSSSSSSSSSSSSSSSSSSSSSSSSSSSSS
\section{Estimates for Character Sums} \label{S2727CEES-I}\hypertarget{S2727CEES-I}
%TTTTTTTTTTTTTTTTTTTTTTTTTTTTTTTTTTTTTTTTTTTTTTTTTTTTTTTT
The new characteristic function for primitive elements and normal elements in finite fields require certain exponential sums over certain subgroups. There are several methods for estimating these exponential sums. Some of these methods are employed here. An estimate of the character sum similar to \hyperlink{lem2727CEES.650A}{Lemma} \ref{lem2727CEES.650A} over the set of prime numbers is proved in {\color{red}\cite[Section 2]{ES1957}}. The general version over finite fields is presented here. Let $ \mathcal{R}$ be the finite ring of polynomials $r(x)\in\F_q[x]$ modulo $x^n-1$. \\

%SSSSSSSSSSSSSSSSSSSSSSSSSSSSSSSSSSSSSSSSSSSSSSSSSSSSSSSSSSSSSSSSSSSSSSSSSSSSSSSSSSS
%SSSSSSSSSSSSSSSSSSSSSSSSSSSSSSSSSSSSSSSSSSSSSSSSSSSSSSSSSSSSSSSSSSSSSSSSSSSSSSSSSSS
%SSSSSSSSSSSSSSSSSSSSSSSSSSSSSSSSSSSSSSSSSSSSSSSSSSSSSSSSSSSSSSSSSSSSSSSSSSSSSSSSSSS
\subsection{Estimates for Incomplete Character Sums over Finite Rings} \label{S2727SGFR-EFR}\hypertarget{S2727SGFR-EFR}
%TTTTTTTTTTTTTTTTTTTTTTTTTTTTTTTTTTTTTTTTTTTTTTTTTTTTTTTT
The finite ring of integers is denoted by $ \mathcal{R}=\Z/q^n\Z$ and its group of units is denoted by $\mathcal{R}^{\times}=\left( \Z/q^n\Z\right)^{\times}$. Some basic estimates of the incomplete character sums over the finite ring  $ \mathcal{R}$ are computed in this section. These estimates are essentially the same as those appearing in \cite{ES1957}, and \cite{GS2008}.
%LLLLLLLLLLLLLLLLLLLLLLLLLLLLLLLLLLLLLLLLLLLLLLLLLLLLLLLLLLLLLLLLLLLLLLLLLLLLLLLLLLLL%LLLLLLLLLLLLLLLLLLLLLLLLLLLLLLLLLLLLLLLLLLLLLLLLLLLLLLLLLLLLLLLLLLLLLLLLLLLLLLLLLLLLLLL
%LLLLLLLLLLLLLLLLLLLLLLLLLLLLLLLLLLLLLLLLLLLLLLLLLLLLLLLLLLLLLLLLLLLLLLLLLLLLLLLLLLLLLLL
\begin{lem}   \label{lem2727FRI.675I}\hypertarget{lem2727FRI.675I} Let $q=p^k$ be a prime power and let $\mathcal{U},\mathcal{V}\subset \mathcal{R}^{\times} $ be a pair of subsets of integers of cardinalities $\#\mathcal{U},\#\mathcal{V}\leq q^n$. If $\psi\ne1 $ is an additive character over the finite ring $\mathcal{R}$, then
	\begin{equation} \label{eq2727FRI.650Id}
		\sum _{v\in \mathcal{V}}\sum _{u\in \mathcal{U}} \psi(uv) 
		\ll q^{n/2} \cdot \sqrt{\#U\cdot \#V }\nonumber.
	\end{equation} 
\end{lem} 

\begin{proof}[\textbf{Proof}]Let $\chi_{\scriptscriptstyle  \mathcal{U}} $ and $\chi_{\scriptscriptstyle\mathcal{V}}$ be the indicator functions of the subsets $\mathcal{U} \text{ and } \mathcal{V}\subset \Z/q^n\Z$ respectively. Now rewrite the character sum as
	\begin{equation} \label{eq2727FRI.675Id}
		\sum _{v\in \mathcal{V}}\sum _{u\in \mathcal{U}} \psi(uv) 
		=\sum _{v\in \mathcal{R}}\chi_{\scriptscriptstyle\mathcal{V}}(v) \sum _{u\in \mathcal{R}}\chi_{\scriptscriptstyle\mathcal{U}}(u)  \psi(uv) ,
	\end{equation} 
	and an application of the Schwarz inequality yields
	\begin{equation} \label{eq2727FRI.675If}
\left|	\sum _{v\in \mathcal{V}}\sum _{u\in \mathcal{U}} \psi(uv)\right|^2 
\leq \sum _{v\in\mathcal{R}}\left|\chi_{\scriptscriptstyle\mathcal{V}}(v)\right)|^2  \sum _{v\in \mathcal{R}}\left|\sum _{u\in \mathcal{R}}\chi_{\scriptscriptstyle\mathcal{U}}(u)  \psi(uv)\right|^2  .
	\end{equation} 
	
	The norm of the outer sum is
	\begin{equation} \label{eq2727FRI.675Ig}
		\sum _{v\in \mathcal{R}}\left|\chi_{\scriptscriptstyle\mathcal{V}}(v)\right)|^2 
		=\#V .
	\end{equation}
	And the norm of the inner sum is
	\begin{eqnarray}\label{eq2727FRI.675Ii}
		\sum _{v\in \mathcal{R}}\left|\sum _{u\in \mathcal{R}^{\times}}\chi_{\scriptscriptstyle\mathcal{U}}(u)  \psi(uv)\right|^2 
		&=& \sum _{v\in \mathcal{R}}\sum _{u_1\in \mathcal{R}}\chi_{\scriptscriptstyle\mathcal{U}}(u_1)  \psi(u_1v)   \sum _{u_0\in \mathcal{R}}\overline{\chi_{\scriptscriptstyle\mathcal{U}}(u_0)  \psi(uv)}\\[.3cm]
		&=&\sum _{u_1\in \mathcal{R}}\chi_{\scriptscriptstyle\mathcal{U}}(u_1)\sum _{u_0\in \mathcal{R}}\chi_{\scriptscriptstyle\mathcal{U}}(u_0)\sum _{v\in \mathcal{R}} \psi((u_1-u_0)v)\nonumber\\[.3cm]
		&=&\sum _{u_0\in \mathcal{R}}\chi_{\scriptscriptstyle\mathcal{U}}(u_0)^2\sum _{v\in \mathcal{R}} 1\nonumber\\[.3cm]
		&=&q^n\#	\mathcal{U},\nonumber
	\end{eqnarray}
	since for $\psi\ne1$, the character sum
	\begin{equation}\label{eq2727FRI.675Ij}
		\sum _{v\in \mathcal{R}} \psi((u_1-u_0)v)=\left \{
		\begin{array}{ll}
			q^n & \text{   \normalfont if } u_0=u_1,  \\
			0 & \text{   \normalfont if } u_0\ne u_1, \\
		\end{array} \right .
	\end{equation}
	and the indicator function $\chi_{\mathcal{U}}(u)=\overline{\chi_{\mathcal{U}}(u)}$ is a real value function. Merging \eqref{eq2727FRI.675Ig} and \eqref{eq2727FRI.675Ii} into \eqref{eq2727FRI.675If} returns
	\begin{equation} \label{eq2727FRI.675Iv}
		\left|	\sum _{v\in \mathcal{V}}\sum _{u\in \mathcal{U}} \psi(uv)\right|^2 
		\leq q^n\#	\mathcal{U}\#V  .
	\end{equation} 
	This completes the verification.
\end{proof}

%SSSSSSSSSSSSSSSSSSSSSSSSSSSSSSSSSSSSSSSSSSSSSSSSSSSSSSSSSSSSSSSSSSSSSSSSSSSSSSSSSSS
%SSSSSSSSSSSSSSSSSSSSSSSSSSSSSSSSSSSSSSSSSSSSSSSSSSSSSSSSSSSSSSSSSSSSSSSSSSSSSSSSSSS
%SSSSSSSSSSSSSSSSSSSSSSSSSSSSSSSSSSSSSSSSSSSSSSSSSSSSSSSSSSSSSSSSSSSSSSSSSSSSSSSSSSS
\subsection{Estimates for Incomplete Character Sums over Finite Fields} \label{S2727CEES-IB}\hypertarget{S2727CEES-IB}
%TTTTTTTTTTTTTTTTTTTTTTTTTTTTTTTTTTTTTTTTTTTTTTTTTTTTTTTT
The basic technique used in the proofs of these incomplete character sums are essentially the same as those appearing in \cite{ES1957}, and \cite{GS2008}.
%LLLLLLLLLLLLLLLLLLLLLLLLLLLLLLLLLLLLLLLLLLLLLLLLLLLLLLLLLLLLLLLLLLLLLLLLLLLLLLLLLLLLLLL
%LLLLLLLLLLLLLLLLLLLLLLLLLLLLLLLLLLLLLLLLLLLLLLLLLLLLLLLLLLLLLLLLLLLLLLLLLLLLLLLLLLLLLLL
\begin{lem}   \label{lem2727CEES.675I}\hypertarget{lem2727CEES.675I} Let $q=p^k$ be a prime power and let $\mathcal{U},\mathcal{V}\subset \F_{q^n}$ be a pair of subsets of elements of cardinalities $\#\mathcal{U},\#\mathcal{V}\leq q^n$. If $\psi\ne1 $ is an additive character over the finite field $\F_{q^n}$, then
	\begin{equation} \label{eq2727CEES.650Id}
		\sum _{v\in \mathcal{V}}\sum _{u\in \mathcal{U}} \psi(uv) 
		\ll q^{n/2} \cdot \sqrt{\#U\cdot \#V }\nonumber.
	\end{equation} 
\end{lem} 

\begin{proof}[\textbf{Proof}]Let $\chi_{\scriptscriptstyle  \mathcal{U}} $ and $\chi_{\scriptscriptstyle\mathcal{V}}$ be the indicator functions of the subsets $\mathcal{U} \text{ and } \chi_{\mathcal{V}}\subset \F_{q^n}$ respectively. Now rewrite the character sum as
\begin{equation} \label{eq2727CEES.675Id}
\sum _{v\in \mathcal{V}}\sum _{u\in \mathcal{U}} \psi(uv) 
=\sum _{v\in \F_{q^n}}\chi_{\scriptscriptstyle\mathcal{V}}(v) \sum _{u\in \F_{q^n}}\chi_{\scriptscriptstyle\mathcal{U}}(u)  \psi(uv) ,
	\end{equation} 
and an application of the Schwarz inequality yields
\begin{equation} \label{eq2727CEES.675If}
	\left|	\sum _{v\in \mathcal{V}}\sum _{u\in \mathcal{U}} \psi(uv)\right|^2 
	\leq \sum _{v\in \F_{q^n}}\left|\chi_{\scriptscriptstyle\mathcal{V}}(v)\right)|^2  \sum _{v\in \F_{q^n}}\left|\sum _{u\in \F_{q^n}}\chi_{\scriptscriptstyle\mathcal{U}}(u)  \psi(uv)\right|^2  .
\end{equation} 

The norm of the outer sum is
\begin{equation} \label{eq2727CEES.675Ig}
	\sum _{v\in \F_{q^n}}\left|\chi_{\scriptscriptstyle\mathcal{V}}(v)\right)|^2 
	=\#V .
\end{equation}
And the norm of the inner sum is
\begin{eqnarray}\label{eq2727CEES.675Ii}
 \sum _{v\in \F_{q^n}}\left|\sum _{u\in \F_{q^n}}\chi_{\scriptscriptstyle\mathcal{U}}(u)  \psi(uv)\right|^2 
&=& \sum _{v\in \F_{q^n}}\sum _{u_1\in \F_{q^n}}\chi_{\scriptscriptstyle\mathcal{U}}(u_1)  \psi(u_1v)   \sum _{u_0\in \F_{q^n}}\overline{\chi_{\scriptscriptstyle\mathcal{U}}(u_0)  \psi(uv)}\\[.3cm]
&=&\sum _{u_1\in \F_{q^n}}\chi_{\scriptscriptstyle\mathcal{U}}(u_1)\sum _{u_0\in \F_{q^n}}\chi_{\scriptscriptstyle\mathcal{U}}(u_0)\sum _{v\in \F_{q^n}} \psi((u_1-u_0)v)\nonumber\\[.3cm]
&=&\sum _{u_0\in \F_{q^n}}\chi_{\scriptscriptstyle\mathcal{U}}(u_0)^2\sum _{v\in \F_{q^n}} 1\nonumber\\[.3cm]
	&=&q^n\#	\mathcal{U},\nonumber
\end{eqnarray}
since for $\psi\ne1$, the character sum
\begin{equation}\label{eq2727CEES.675Ij}
	\sum _{v\in \F_{q^n}} \psi((u_1-u_0)v)=\left \{
	\begin{array}{ll}
		q^n & \text{   \normalfont if } u_0=u_1,  \\
		0 & \text{   \normalfont if } u_0\ne u_1, \\
	\end{array} \right .
\end{equation}
and the indicator function $\chi_{\mathcal{U}}(u)=\overline{\chi_{\mathcal{U}}(u)}$ is a real value function. Merging \eqref{eq2727CEES.675Ig} and \eqref{eq2727CEES.675Ii} into \eqref{eq2727CEES.675If} returns
\begin{equation} \label{eq2727CEES.675Iv}
	\left|	\sum _{v\in \mathcal{V}}\sum _{u\in \mathcal{U}} \psi(uv)\right|^2 
	\leq q^n\#	\mathcal{U}\#V  .
\end{equation} 
This completes the verification.
\end{proof}

%LLLLLLLLLLLLLLLLLLLLLLLLLLLLLLLLLLLLLLLLLLLLLLLLLLLLLLLLLLLLLLLLLLLLLLLLLLLLL
%LLLLLLLLLLLLLLLLLLLLLLLLLLLLLLLLLLLLLLLLLLLLLLLLLLLLLLLLLLLLLLLLLLLLLLLLLLLLL
\begin{lem}   \label{lem2727CEES.650U}\hypertarget{lem2727CEES.650U} Let $q$ be a prime power and let $\mathcal{U}\subset \mathcal{R}$ be the group of units in the finite ring $\mathcal{R}=\F_q[x]/(x^n-1)$. If $\psi\ne1 $ be an additive character of the finite field $\F_{q^n}$, then
	\begin{equation} \label{eq2727CEES.650Ud}
		\sum _{u\in \mathcal{U}} \psi(u) 
		\leq q^{n/2} .
	\end{equation} 
\end{lem} 

\begin{proof}[\textbf{Proof}] The group of units in $\mathcal{R}$ is defined by 
	\begin{equation}\label{eq2727CEES.650Uf}
		\mathcal{U}=\{u(x)\in\F_q[x]:\deg u(x)\leq n-1 \text{ and }\gcd (u(x),x^n-1)=1\}.
	\end{equation}
	Furthermore, for each fixed $v\in  \mathcal{V}\subseteq \mathcal{U}$ the map $u\longrightarrow uv$ permutes the set $\mathcal{U}$. This implies that
	\begin{equation} \label{eq2727CEES.650Ui}
		\sum _{v\in \mathcal{V}}\sum _{u\in \mathcal{U}}  \psi(uv)=  \#\mathcal{V}\sum _{u\in \mathcal{U}} \psi(u).
	\end{equation}
	
	Consequently, applying \hyperlink{lem2727CEES.675I}{Lemma} \ref{lem2727CEES.675I} yields the upper bound
	\begin{eqnarray} \label{eq2727CEES.650Uj}
		q^{n/2} \cdot \sqrt{\#U\cdot \#V }&\geq& \left|	\sum _{v\in \mathcal{V}}\sum _{u\in \mathcal{U}} \psi(uv) \right|\\[.3cm]
		&=& \#\mathcal{V}	\left|	\sum _{u\in \mathcal{U}} \psi(u) \right|\nonumber.
	\end{eqnarray} 
	Setting $\mathcal{V}=\mathcal{U}$ completes the proof.
\end{proof}
%LLLLLLLLLLLLLLLLLLLLLLLLLLLLLLLLLLLLLLLLLLLLLLLLLLLLLLLLLLLLLLLLLLLLLLLLLL
%LLLLLLLLLLLLLLLLLLLLLLLLLLLLLLLLLLLLLLLLLLLLLLLLLLLLLLLLLLLLLLLLLLLLLLLLLL
\begin{lem}   \label{lem2727CEES.650V}\hypertarget{lem2727CEES.650V} Let $q$ be a prime power and let $\mathcal{U}\subset \mathcal{R}$ be the group of units in the finite ring $\mathcal{R}=\F_q[x]/(x^n-1)$. If $\chi\ne1 $ is a multiplicative character of the finite field $\F_{q^n}$, then
	\begin{equation} \label{eq2727CEES.650Vd}
		\sum _{v\in \mathcal{V}}\sum _{u\in \mathcal{U}}\chi(u+v) 
		\ll q^{n/2} \cdot \sqrt{\#U\cdot \#V }.
	\end{equation} 
\end{lem} 

\begin{proof}[\textbf{Proof}] Let $\chi \ne1$ be a nontrivial multiplicative character over the finite field $\F_{q^n}$. The canonical complete character Gauss sum has the form
	\begin{equation}\label{eq2727CEES.650If}
		G(\chi,\psi)=	\sum _{\rho \in \F_{q^n}^{\times}}\chi(\rho)\psi(\rho),
	\end{equation} 
	where $\psi(\rho)=e^{i2\pi \tr(\rho)/p}$ and $\psi_0(\rho)=1$. Now use \hyperlink{lem2727CEES.650A}{Lemma} \ref{lem2727CEES.650A} and \eqref{eq2727CEES.650If} to rewrite the sum over the subsets $\mathcal{U}$ and $\mathcal{V}$ as
	\begin{eqnarray}\label{eq2727CEES.650Ig}
		\sum _{v\in \mathcal{V}}\sum _{u\in \mathcal{U}} \chi(u+v)
		&=&\sum _{v\in \mathcal{V}}\sum _{u\in \mathcal{U}} \frac{1}{q^n-1}\sum _{\psi \in \mathscr{C}_m(q^n)}\psi(u+v)G(\overline{\chi},\psi)\\[.3cm]
		&=& \frac{1}{q^n-1}\sum _{v\in \mathcal{V}}\sum _{u\in \mathcal{U}} \psi_0(u+v)G(\overline{\psi_0},\psi)\nonumber\\[.3cm]
		&&\hskip 1 in +  \frac{1}{q^n-1}\sum _{1\ne\psi \in \mathscr{C}_m(q^n)}G(\overline{\chi},\psi)\sum _{v\in \mathcal{V}}\psi(v)\sum _{u\in \mathcal{U}}\psi(u)\nonumber\\[.3cm]
		&=&\frac{1}{q^n-1}\sum _{1\ne\chi \in \mathscr{C}_m(q^n)}G(\overline{\chi},\psi)\sum _{v\in \mathcal{V}}\chi(v)\sum _{u\in \mathcal{U}}\chi(u)\nonumber.
	\end{eqnarray}
The trivial character $\psi_0=1$ does not contribute since $ G(\chi,\psi_0)=0$, see \hyperlink{lem2727CEES.650GS}{Lemma} \ref{lem2727CEES.650GS}. Taking absolute value and using $|G(\chi,\psi)|=q^{n/2}$ yield
	\begin{eqnarray}\label{eq2727CEES.650Ih}
		\left|	\sum _{v\in \mathcal{V}}\sum _{u\in \mathcal{U}} \psi(uv)\right|
		&\leq&\frac{1}{q^n-1}\sum _{1\ne\chi \in \mathscr{C}_m(q^n)}\left|G(\overline{\chi},\psi)\right|\left|\sum _{v\in \mathcal{V}}\chi(v)\right|\left|\sum _{u\in \mathcal{U}}\chi(u)\right|\\[.3cm]
		&\leq&\frac{q^{n/2}}{q^n-1}\sum _{1\ne\chi \in \mathscr{C}_m(q^n)}\left|\sum _{v\in \mathcal{V}}\chi(v)\right|\left|\sum _{u\in \mathcal{U}}\chi(u)\right|	\nonumber.
	\end{eqnarray}
	Next observe that
	\begin{eqnarray}\label{eq2727CEES.650Ii}
		\sum _{1\ne\chi \in \mathscr{C}_m(q^n)}\left|\sum _{u\in \mathcal{U}}\chi(u)\right|^2
		&\leq&\sum _{\chi \in \mathscr{C}_m(q^n)}\sum _{u_0\in \mathcal{U}}\chi(u_0)\sum _{u_1\in \mathcal{U}}\overline{\chi(u_1)}\\[.3cm]
		&=&\sum _{1\ne\chi \in \mathscr{C}_m(q^n)}\sum _{u_0\in \mathcal{U}}\chi(u_0)\sum _{u_1\in \mathcal{U}}\chi(u_1^{-1})\nonumber\\[.3cm]
		&=&\sum _{u_0\in \mathcal{U}}\sum _{u_1\in \mathcal{U}}\sum _{1\ne\chi \in \mathscr{C}_m(q^n)}\chi(u_0u_1^{-1})\nonumber\\[.3cm]
		&=&(q^n-1)\#	\mathcal{U}\nonumber
	\end{eqnarray}
	since for $\chi\ne1$, the character sum
	\begin{equation}\label{eq2727CEES.650Ij}
		\sum _{1\ne\chi \in \mathscr{C}_m(q^n)}\chi(u_0u_1^{-1})=\left \{
		\begin{array}{ll}
			\#	\mathcal{U} & \text{   \normalfont if } u_0u_1^{-1}=1,  \\
			0 & \text{   \normalfont if } u_0u_1^{-1}\ne1. \\
		\end{array} \right .
	\end{equation}
	An application of the Cauchy-Schwarz inequality to \eqref{eq2727CEES.650Ih} and substituting \eqref{eq2727CEES.650Ii} yield
	
	\begin{eqnarray}\label{eq2727CEES.650Il}
		\left|	\sum _{v\in \mathcal{V}}\sum _{u\in \mathcal{U}} \psi(uv)\right|^2
		&\leq&\frac{q^n}{(q^n-1)^2}\sum _{1\ne\chi \in \mathscr{C}_m(q^n)}\left|\sum _{v\in \mathcal{V}}\chi(v)\right|^2 \sum _{1\ne\chi \in \mathscr{C}_m(q^n)}\left|\sum _{u\in \mathcal{U}}\chi(u)\right|^2\\[.3cm]
		&\leq&\frac{q^n}{(q^n-1)^2}\cdot (q^n-1)\#	\mathcal{V}\cdot (q^n-1)\#	\mathcal{U}	\nonumber.
	\end{eqnarray}
	Simplifying the last inequality completes the proof. 
\end{proof}

%SSSSSSSSSSSSSSSSSSSSSSSSSSSSSSSSSSSSSSSSSSSSSSSSSSSSSSSSSSSSSSSSSSSSSSSSSSSSSSSSSSS
%SSSSSSSSSSSSSSSSSSSSSSSSSSSSSSSSSSSSSSSSSSSSSSSSSSSSSSSSSSSSSSSSSSSSSSSSSSSSSSSSSSS
%SSSSSSSSSSSSSSSSSSSSSSSSSSSSSSSSSSSSSSSSSSSSSSSSSSSSSSSSSSSSSSSSSSSSSSSSSSSSSSSSSSS
%SSSSSSSSSSSSSSSSSSSSSSSSSSSSSSSSSSSSSSSSSSSSSSSSSSSSSSSSSSSSSSSSSSSSSSSSSSSSSSSSSSS
%SSSSSSSSSSSSSSSSSSSSSSSSSSSSSSSSSSSSSSSSSSSSSSSSSSSSSSSSSSSSSSSSSSSSSSSSSSSSSSSSSSS
%SSSSSSSSSSSSSSSSSSSSSSSSSSSSSSSSSSSSSSSSSSSSSSSSSSSSSSSSSSSSSSSSSSSSSSSSSSSSSSSSSSS
\section{Average Orders of Mobius Functions}\label{S2222MN}\hypertarget{S2222MN}
The support $\text{supp}(\mu)$ of the Mobius function $\mu: \N\longrightarrow \{-1,0,1\}$ is the subset of squarefree integers. The \textit{Mobius randomness principle} anticipates significant cancellation on the summatory function $\sum_{n\leq x}\mu(n)f(n)$, where $f:\R\longrightarrow \C$ is a bounded function.
\begin{thm} \label{thm2222.500}\hypertarget{lem2222.500} If $\mu$ is the Mobius function, then, for any large number $x\geq1$, the following statements are true.
	\begin{enumerate} [font=\normalfont, label=(\roman*)]
		\item $\displaystyle \sum_{n <x} \mu(n)=O \left (xe^{-c\sqrt{\log x}}\right )$, \tabto{8cm} unconditionally,
		\item $\displaystyle \sum_{n<x}\frac{\mu(n)}{n}=O\left( e^{-c\sqrt{\log x}} \right ), $ \tabto{8cm} unconditionally,
	\end{enumerate}where $c>0$ is an absolute constant.
\end{thm}
\begin{proof}  See {\color{red} \cite[Equation 27.11.12]{DLMF}}, {\color{red} \cite[p.\ 6]{DL2012}}, {\color{red} \cite[p.\ 182]{MV2007}}, et alii.   
\end{proof}
The nontrivial cancellation in the finite sum $\sum_{n\leq x}\mu(n)f(n)$, where the function $f(n)=1/n$, is derived by partial summation from the prime number theorem $\sum_{n\leq x}\mu(n)=o(x)$, the precise statement is in \hyperlink{lem2222.500}{Theorem} \ref{thm2222.500}. Another simple case for the fractional part function $f(z)=\{z\}$ is derived from the largest integer $[z]$ function and the identity \eqref{eq2222.600f}. Many other special cases have been proved in the literature.
\begin{lem} \label{lem2222.600}\hypertarget{lem2222.600} If $x\geq 1$ is a large number, $\mu$ is the Mobius function and $f(z)=\{z\}$ is the fractional part function, then
	\begin{equation}\label{eq2222.600d}
		\sum_{n\leq x }\mu(n)\bigg\{\frac{x}{n}\bigg\} =-1 +O \left (xe^{-c\sqrt{\log x}}\right ) ,
	\end{equation}
	where $c>0$ is an absolute constant. 
\end{lem}
\begin{proof} Start with the identity
	\begin{align}\label{eq2222.600f}
		1&=\sum_{n\leq x}\mu(n)\bigg[\frac{x}{n}\bigg]\\[.3cm]
		&=x\sum_{n\leq x}\frac{\mu(n)}{n}-\sum_{n\leq x}\mu(n)\bigg\{\frac{x}{n}\bigg\}\nonumber,
	\end{align}
	see { \color{red}\cite[Theorem 3.12]{AT1976}}, the second line uses the largest integer function $[z]=z-\{z\}$. Applying \hyperlink{lem2222.500}{Theorem} \ref{thm2222.500} to the first finite sum on the left side and rearranging it complete the verification.
\end{proof}

%SSSSSSSSSSSSSSSSSSSSSSSSSSSSSSSSSSSSSSSSSSSSSSSSSSSSSSSSSSSSSSSSSSSSSSSSSSSSSSSSSSS
%SSSSSSSSSSSSSSSSSSSSSSSSSSSSSSSSSSSSSSSSSSSSSSSSSSSSSSSSSSSSSSSSSSSSSSSSSSSSSSSSSSS
%SSSSSSSSSSSSSSSSSSSSSSSSSSSSSSSSSSSSSSSSSSSSSSSSSSSSSSSSSSSSSSSSSSSSSSSSSSSSSSSSSSS
%SSSSSSSSSSSSSSSSSSSSSSSSSSSSSSSSSSSSSSSSSSSSSSSSSSSSSSSSSSSSSSSSSSSSSSSSSSSSSSSSSSS
%SSSSSSSSSSSSSSSSSSSSSSSSSSSSSSSSSSSSSSSSSSSSSSSSSSSSSSSSSSSSSSSSSSSSSSSSSSSSSSSSSSS
%SSSSSSSSSSSSSSSSSSSSSSSSSSSSSSSSSSSSSSSSSSSSSSSSSSSSSSSSSSSSSSSSSSSSSSSSSSSSSSSSSSS
%SSSSSSSSSSSSSSSSSSSSSSSSSSSSSSSSSSSSSSSSSSSSSSSSSSSSSSSSSSSSSSSSSSSSSSSSSSSSSSSSSSS
\subsection{Estimates of Exponential Sums over Subgroups} %\label{S7770-EES}\hypertarget{S7770-EES}

\begin{lem} \label{lem7770.300}\hypertarget{lem7770.300}  Let \(p\geq 2\) be a large prime and let $\tau$ be a primitive root  mod \(p\). Assume the element $\alpha\ne0$ is not a primitive root. Then, 
	\begin{equation} \label{eq7770.l00c}
		\sum _{1\leq t_1\leq q^n-1}\sum _{\substack{1\leq s\leq q^n-1\\\gcd (s,q^n-1)=1}} e^{\frac{-i 2\pi(s-\log_{\tau}\alpha)t_1}{q^n}}=O\bigg(q^ne^{-c\sqrt{\log q^n}}\bigg)\nonumber,
	\end{equation} 
	where $c>0$ is an absolute constant.
\end{lem}
\begin{proof}[\textbf{Proof}] Employing the characteristic function
	\begin{align} \label{eq7770.l00d}
		\sum_{\substack{d\mid m\\d\mid n}}\mu(d)=
		\left \{\begin{array}{ll}
			1 & \text{   \normalfont if } \gcd(m,n)=1,  \\
			0 & \text{   \normalfont if } \gcd(m,n)\ne1,\\
		\end{array} \right.
	\end{align}
	
	of relatively prime integers to remove the double index in the first line of \eqref{eq7770.l00f} and switching the order of summation in the second line yield
	\begin{eqnarray} \label{eq7770.l00f}
		S(x)&=&\sum _{1\leq t\leq q^n-1}\sum _{\substack{1\leq s\leq q^n-1\\\gcd (s,q^n-1)=1}} e^{\frac{-i 2\pi(s-\log_{\tau}\alpha)t}{q^n}} \nonumber\\[.3cm]  
		&= &\sum _{1\leq t\leq q^n-1,}\sum _{1\leq s\leq q^n-1} e^{\frac{-i 2\pi(s-\log_{\tau}\alpha)t}{q^n}}\sum_{\substack{d\mid s\\d\mid q^n-1}}\mu(d)\nonumber\\[.3cm]
		&= &\sum_{d\leq  q^n-1}\mu(d)\sum _{1\leq m< (q^n-1)/d,}\sum _{1\leq t< q^n-1} e^{\frac{i2\pi t\left (dm-\log_{\tau}\alpha\right)}{q^n}}.
	\end{eqnarray} 
	
	By hypothesis $s-\log_{\tau}\alpha=dm-\log_{\tau}\alpha\ne0$ for any $s \geq 1$ such that $\gcd(s,q^n-1)=1$. Thus, evaluating the inner sum indexed by $t$ followed by the evaluation of the resulting inner sum indexed by $m$ yield
	\begin{eqnarray} \label{eq7770.l00h}
		S(x)&=&\sum_{d\leq  q^n-1}\mu(d)\sum _{1\leq m< (q^n-1)/d}(-1) \\[.3cm]  
		&= &-\sum_{d\leq  q^n-1}\mu(d)\left [\frac{q^n-1}{d}\right ]\nonumber.
	\end{eqnarray} 
	Replace the the largest integer function $[z]=z-\{z\}$ to obtain
	\begin{eqnarray} \label{eq7770.l00i}
		S(x)&=&-\sum_{d\leq  q^n-1}\mu(d)\Bigg(\frac{q^n-1}{d}-\Bigg\{\frac{q^n-1}{d}\Bigg\}\Bigg) \\[.3cm]  
		&= &-(q^n-1)\sum_{d\leq  q^n-1}\frac{\mu(d)}{d}+ \sum_{d\leq  q^n-1}\mu(d)\Bigg\{\frac{q^n-1}{d}\Bigg\}\nonumber.
	\end{eqnarray}
	
	Applying \hyperlink{thm2222.500}{Theorem} \ref{thm2222.500} to the first finite sum and \hyperlink{lem2222.600}{Lemma} \ref{lem2222.600} to the second finite sum return
	\begin{eqnarray} \label{eq7770.l00j}
		S(x)&=&O\Bigg((q^n-1)e^{-c_1\sqrt{\log q^n-1}}\Bigg) +\Bigg(-1+O\bigg((q^n-1)e^{-c_2\sqrt{\log (q^n-1)}}\bigg)\Bigg)\nonumber\\[.3cm]  
		&= &O\bigg(q^ne^{-c_2\sqrt{\log q^n}}\bigg),
	\end{eqnarray} 
	where $c_1,c_2>0$ are constants. 
\end{proof}

%SSSSSSSSSSSSSSSSSSSSSSSSSSSSSSSSSSSSSSSSSSSSSSSSSSSSSSSSSSSSSSSSSSSSSSSSSSSSSSSSSSS
%SSSSSSSSSSSSSSSSSSSSSSSSSSSSSSSSSSSSSSSSSSSSSSSSSSSSSSSSSSSSSSSSSSSSSSSSSSSSSSSSSSS
%SSSSSSSSSSSSSSSSSSSSSSSSSSSSSSSSSSSSSSSSSSSSSSSSSSSSSSSSSSSSSSSSSSSSSSSSSSSSSSSSSSS
%SSSSSSSSSSSSSSSSSSSSSSSSSSSSSSSSSSSSSSSSSSSSSSSSSSSSSSSSSSSSSSSSSSSSSSSSSSSSSSSSSSS
%SSSSSSSSSSSSSSSSSSSSSSSSSSSSSSSSSSSSSSSSSSSSSSSSSSSSSSSSSSSSSSSSSSSSSSSSSSSSSSSSSSS
%SSSSSSSSSSSSSSSSSSSSSSSSSSSSSSSSSSSSSSSSSSSSSSSSSSSSSSSSSSSSSSSSSSSSSSSSSSSSSSSSSSS
%SSSSSSSSSSSSSSSSSSSSSSSSSSSSSSSSSSSSSSSSSSSSSSSSSSSSSSSSSSSSSSSSSSSSSSSSSSSSSSSSSSS
\section{Totient Functions over the Integers} \label{C5005}\hypertarget{C5005}
Various identities and properties of the totient functions are required in analytic number theory, for example, in the analysis of the primitive root, and the analysis of elliptic groups of points of elliptic curves. Some of the essential analytic methods are introduced here. 

%SSSSSSSSSSSSSSSSSSSSSSSSSSSSSSSSSSSSSSSSSSSSSSSSSSS
%SSSSSSSSSSSSSSSSSSSSSSSSSSSSSSSSSSSSSSSSSSSSSSSSSSS
%SSSSSSSSSSSSSSSSSSSSSSSSSSSSSSSSSSSSSSSSSSSSSSSSSSS
%SSSSSSSSSSSSSSSSSSSSSSSSSSSSSSSSSSSSSSSSSSSSSSSSSSS
\subsection{Basic Definitions and Properties} \label{S5005D}\hypertarget{S5005D}
Let $p_1, p_2, p_3, \ldots, p_k$ be the sequence of prime numbers in increasing order, and let $n=p_1^{e_1} p_2^{e_2} \cdots p_t^{e_t}$ be the prime decomposition of an arbitrary integer.  
\begin{dfn}\label{dfn5005.100d}\hypertarget{dfn5005.100d}
	The Euler totient function over the finite ring $\mathbb{Z}/n\mathbb{Z}$ is defined by  $$\varphi (n)=\sum_{\substack{k<n\\\gcd(k,n)=1}}1.$$ It counts the number of relatively prime integers up to $n\geq1$.
\end{dfn}

For each $n\in \mathbb{N}$, the totient functions have many identities, for example, there are additive identity, multiplicative identity, and other identities. Some of the identities and properties of the totient functions are recorded in this section.

\begin{lem}  \label{lem5005.100e}\hypertarget{lem5005.100e} The totient function has the following representations.
	\begin{enumerate} [font=\normalfont, label=(\roman*)]
		\item  
		$\displaystyle 	
		\varphi (n)=n\sum _{d \mid  n} \frac{\mu (d)}{d},$\tabto{8cm} additive identity.
		\item  
		$\displaystyle 	
		\varphi (n)=n\prod_{p \mid n}\left (1-\frac{1}{p}\right ),$\tabto{8cm} multiplicative identity.
		\item  
		$\displaystyle 	\varphi (n)=\sum _{1\leq m< n,} \sum _{\substack{d \mid m \\ d \mid n}} \mu (d)       ,$\tabto{8cm} counting identity.
	\end{enumerate}
\end{lem}

\begin{proof}[\textbf{Proof}] (ii) The product formula 
	\begin{equation}
		\frac{	\varphi(n)}{n}=\prod_{p\mid n}\left( 1-\frac{1}{p}\right),
	\end{equation} where $p\geq2$ is a prime divisor of $n$ is proved in {\color{red}\cite[Theorem 2.4]{AT1976}}. (iii) Employ the characteristic function of relatively prime integers
	\begin{equation} \label{eq997.103d}
		\sum _{d \mid  n} \mu (d) =
		\begin{cases}
			1 & \gcd (d,n)=1, \\
			0 & \gcd (d,n)\neq 1, \\
		\end{cases}
	\end{equation}
	to derive this identity.
\end{proof}
\begin{lem} \label{lem5005.100h}\hypertarget{lem5005.100h} Given a pair of integers $m,n \in \N$, the totient function satisfies the followings relations.
	\begin{enumerate} [font=\normalfont, label=(\roman*)]
		\item  
		$\displaystyle \varphi(mn)=\varphi(m)\varphi(n)$, \tabto{9cm} \text {\normalfont if}  $\gcd(m,n)=1.$
		\item 
		$\displaystyle \varphi(mn)=\frac{d}{\varphi(d)}\varphi(m)\varphi(n)$, \tabto{9cm} \text {\normalfont where}$ \gcd(m,n)=d.$
		\item 
		$\displaystyle \varphi(m)\varphi(n)=\sum_{d \mid \gcd(m,n)} \varphi(mn/d)\mu(d),$ \tabto{9cm} \text {\normalfont for any }$ m, n\geq 1.$
		\item 
		$\displaystyle \varphi(m) \mid \varphi(n)$, \tabto{9cm} \text{\normalfont if} $m \mid  n.$
		\item 
		$\displaystyle \varphi(n) \equiv 0 \bmod 2$, \tabto{9cm} \text{\normalfont if} $n \geq 3.$
	\end{enumerate}
\end{lem}

The more general power sum of relatively prime integers can be computed using the same approach as in the proof of identity (iii). 

\begin{lem}  \label{lem5005.100i}\hypertarget{lem5005.100i} The inverse totient function has the following representations.
	\begin{enumerate} [font=\normalfont, label=(\roman*)]
		\item  
		$\displaystyle 	
		\frac{1}{\varphi (n)}=\frac{1}{n}\sum _{d \mid  n} \frac{\mu^2 (d)}{\varphi (d)},$\tabto{8cm} additive inverse identity.
		\item  
		$\displaystyle 	
		\frac{1}{\varphi (n)}=\frac{1}{n}\prod _{p \mid  n} \left (1-\frac{1}{p}\right )^{-1},$\tabto{8cm} multiplicative inverse identity.
	\end{enumerate}
\end{lem}
The Mobius inverse pair, which has the form
\begin{equation}\label{eq5005.100d}
	n=\sum_{d\mid n}	\varphi (d) \quad \text{ and }\quad	\varphi (n)=n\sum_{d\mid n} \frac{\mu (d)}{d},
\end{equation}
is used in many calculations. 
There are many generalizations of the totient function. One of the simplest generalizations is the generalized Mobius inverse pair 
\begin{equation}\label{eq5005.100j}
	n^s=\sum_{d\mid n}	\varphi_s (d) \quad \text{ and }\quad	\varphi_s (n)=n^s\sum_{d\mid n} \frac{\mu (d)}{d^{s}},
\end{equation}
where $s\in \C$ is a complex number. These are obtained via the generalized Mobius inverse formula.

%SSSSSSSSSSSSSSSSSSSSSSSSSSSSSSSSSSSSSSSSSSSSSSSSSSS
%SSSSSSSSSSSSSSSSSSSSSSSSSSSSSSSSSSSSSSSSSSSSSSSSSSS
%SSSSSSSSSSSSSSSSSSSSSSSSSSSSSSSSSSSSSSSSSSSSSSSSSSS
%SSSSSSSSSSSSSSSSSSSSSSSSSSSSSSSSSSSSSSSSSSSSSSSSSSS
\subsection{Extreme Values of the Totient Function} \label{S9977TFI}
Some estimates for the extreme values of the Euler totient function are stated in this subsection. Currently the best unconditional upper bound of this arithmetical function is 
\begin{equation} \label{eq9977TFI.175}
	\frac{n}{\varphi(n)} < 
	e^{\gamma} \log \log n +\frac{5}{2\log \log n}
\end{equation}  
for any integer $n \in \N$ with one exception for $n =2\cdot 3 \cdots 23$, see {\color{red}\cite[Theorem 15]{RS1962}}. The maximal values of the Euler function occurs at the prime arguments. Id est, $\varphi(p)=p-1<p$. There are other subsets of integers that have nearly maximal values. In fact, asymptotically, these integers and the primes number have the same order of magnitudes.

\begin{lem} \label{lem9977TFI.300}\hypertarget{lem9977TFI.300}{Lemma} Let $ x \geq 1$ be a large number, and let $n=1+\prod_{p \leq \log x}p$. Then
	\begin{enumerate} [font=\normalfont, label=(\roman*)]
		\item $\varphi(n)=n+O\left (n/\log \log n \right)$,\tabto{7cm}infinitely often as $n\to\infty.$
		\item $\varphi(n+1)=n/2+O\left (n/\log n \right)$,\tabto{7cm}infinitely often as $n\to\infty.$
	\end{enumerate}		 
\end{lem}

\begin{proof}[\textbf{Proof}] (i) Observe that $\log n \geq \sum_{p \leq \log x} \log p$, so that $p \leq \log x \leq 2\log n$. Hence, a prime divisor $q \mid n=1+\prod_{p \leq \log x}p$ implies that $q>\log n$. Consequently, there is the upper bound
	\begin{eqnarray}\label{eq997.11}
		\varphi(n)&= &n \prod_{p\mid n}\left ( 1-\frac{1}{p}\right)\\
		&\leq & n\left ( 1-\frac{1}{\log n}\right)\nonumber\\[.2cm]
		&=& n +O\left (\frac{n}{\log n}\right) \nonumber.
	\end{eqnarray}
The frequency of occurrences is obvious since $n=1+\prod_{p \leq \log x}p$ has no restrictions. In the other direction, there is the lower bound
	\begin{eqnarray}\label{eq997.13}
		\varphi(n)&=&n \prod_{p \mid n}\left ( 1-\frac{1}{p}\right)\\
		&\geq &n \prod_{\log n< p \leq 2\log n}\left ( 1-\frac{1}{p}\right)\nonumber\\[.2cm]
		&=&n+O\left ( \frac{n}{\log \log n} \right )\nonumber.
	\end{eqnarray}
	Both relations \eqref{eq997.11} and \eqref{eq997.13} confirm the claim.
	(ii) The prime divisors of $n+1$ are $q=2$ and some prime $q> \log n$, so the claim follows from
	\begin{equation} \label{eq997.15}
		\varphi(n+1)=(n+1) \prod_{p|(n+1)}\left ( 1-\frac{1}{p}\right)\leq \frac{n}{2}\left (1-\frac{1}{\log n} \right )=\frac{n}{2}+O\left ( \frac{n}{\log n} \right ).
	\end{equation} 
\end{proof}

\begin{thm} \label{thm9192.21}  Let \(p\geq 2\) be a large prime. Then, the followings extreme values hold.
	\begin{enumerate} [font=\normalfont, label=(\roman*)]
		\item$ \displaystyle 
		\frac{\varphi(n)}{n}\leq 1-\frac{1}{n}$,  \tabto{6cm} if $n\geq 2$ is an integer.
		\item$ \displaystyle 
		\frac{\varphi(n)}{n}\geq \frac{e^{-\gamma}}{4\log \log n}$,  \tabto{6cm} if $n\geq 2$ is a highly composite integer.
		\item$ \displaystyle 
		\frac{\varphi(n)}{n}\approx \frac{e^{-\gamma}}{\log \log \log n}$,    \tabto{6cm} if $n\geq 2$ is an average integer.
	\end{enumerate}
\end{thm}

The totient function have a wide range of values, as confirmed by \hyperlink{lem9977TFI.300}{Lemma} \ref{lem9977TFI.300}, and this accounts for the wide range and large gaps in the sequence of totient gaps  
\begin{equation} \label{eq997.17}
	\varphi(2)-\varphi(1), \; \;	\varphi(3)-\varphi(2),	\; \;\varphi(4)-\varphi(3),\; \; \ldots, \; \;	\varphi(n+1)-\varphi(n), \; \; \ldots .
\end{equation}
The gap can be as small as $\varphi(n+1)-\varphi(n)=0$, and it can be as large as $\varphi(n+1)-\varphi(n)=n/2+O\left (n/\log n \right)$. For example, $\varphi(4)-\varphi(3)=0$, and $\varphi(2\cdot3\cdot5+1)-\varphi(2\cdot3\cdot5+2)=14$.

%SSSSSSSSSSSSSSSSSSSSSSSSSSSSSSSSSSSSSSSSSSSSSSSSSSS
%SSSSSSSSSSSSSSSSSSSSSSSSSSSSSSSSSSSSSSSSSSSSSSSSSSS
%SSSSSSSSSSSSSSSSSSSSSSSSSSSSSSSSSSSSSSSSSSSSSSSSSSS
%SSSSSSSSSSSSSSSSSSSSSSSSSSSSSSSSSSSSSSSSSSSSSSSSSSS
\subsection{Lower And Upper Bounds for Almost All Integers} \label{s7}
The totient function has the explicit lower bound
\begin{equation}\label{eq997.30}
	\frac{\varphi(n)}{n}> \frac{1}{e^{\gamma} \log \log n+5/(2 \log \log n)}
\end{equation} 
for any integer $n \geq 1$, see {\color{red}\cite[Theorem 7]{RS1962}}. The subset of integers that satisfy the extreme values of this inequality is a very thin subset of integers. In contrast, almost every integers has a large lower and upper bounds. The lower and upper bounds for almost every integer are significantly smaller by an iterated factor of log as demonstrated below.
\begin{thm} \label{thm1.4} For almost all integers $n \geq 1$, the ratio $n/\varphi(n)$ has the followings bounds.
	\begin{enumerate} [font=\normalfont, label=(\roman*)]
		\item $ \displaystyle 
		\log \log \log n \ll \frac{n}{\varphi(n)} .$
		\item 
		$ \displaystyle \frac{n}{\varphi(n)} \ll \log \log \log n.$
	\end{enumerate}
\end{thm}

\begin{proof}[\textbf{Proof}] (ii) By the Erdos-Kac theorem, (confer \cite{EK1940}, {\color{red}\cite[Theorem 431]{HW1979}}, et cetera), almost every integer $n \geq 1$ has $\omega(n) \ll \log \log n$ prime divisors. In addition, the interval $[1, (\log \log n)^2]$ contains $\pi((\log \log n)^2) \gg \log \log n$ primes. Thus, 
	\begin{eqnarray}
		\frac{n}{\varphi(n)} &=& \prod_{p \mid n }\left( 1- \frac{1}{p} \right) ^{-1}  \\
		&\ll&  \prod_{p \ll (\log \log n)^2}\left( 1- \frac{1}{p} \right) ^{-1} \nonumber\\[.3cm]
		&\ll& \log \log \log n \nonumber.
	\end{eqnarray}
	The reverse inequality is similar.
\end{proof}

\begin{thm} \label{thm1.5}  Let $ x \geq 1$ be a large number. Then, there are constants $c_0>0$ and $c_1>0$ for which the Euler totient function has the followings bounds.
	\begin{enumerate} [font=\normalfont, label=(\roman*)]
		\item  $\displaystyle \frac{\varphi(n)} {n} \geq \frac{c_0}{\log \log \log n}$ \tabto{6cm}
		for almost all large integers $n \geq 1$.
		\item $\displaystyle \frac{\varphi(n)} {n} \leq \frac{c_1}{\log \log \log n}$ \tabto{6cm}
		for almost all large integers $n \geq 1$.
	\end{enumerate}
\end{thm}
The proof of this theorem uses the same technique as Theorem \ref{thm1.4}.

%SSSSSSSSSSSSSSSSSSSSSSSSSSSSSSSSSSSSSSSSSSSSSSSSSSSSSSSSSSSSSSSS
%SSSSSSSSSSSSSSSSSSSSSSSSSSSSSSSSSSSSSSSSSSSSSSSSSSSSSSSSSSSSSSSS
%SSSSSSSSSSSSSSSSSSSSSSSSSSSSSSSSSSSSSSSSSSSSSSSSSSSSSSSSSSSSSSSS
%SSSSSSSSSSSSSSSSSSSSSSSSSSSSSSSSSSSSSSSSSSSSSSSSSSSSSSSSSSSSSSSS
%SSSSSSSSSSSSSSSSSSSSSSSSSSSSSSSSSSSSSSSSSSSSSSSSSSSSSSSSSSSSSSSS
%SSSSSSSSSSSSSSSSSSSSSSSSSSSSSSSSSSSSSSSSSSSSSSSSSSSSSSSSSSSSSSSS
\subsection{Lower Bound for All Integers}
This section provides the detailed proofs of the lower bounds of the totient function $\varphi(n)/n=\prod_{r\mid n}\left( 1-1/r\right)$, where $r\geq2$ is a prime divisor of $n$, see {\color{red}\cite[Theorem 2.4]{AT1976}}. These estimates are similar to the explicit lower bound
\begin{equation}\label{eq9955P.400Ta}
	\frac{\varphi(n)}{n}> \frac{1}{e^{\gamma} \log \log n+5/(2 \log \log n)}
\end{equation} 
for any integer $n \geq 1$, see {\color{red}\cite[Theorem 7]{RS1962}}, but one of the proofs is much simpler.

%TTTTTTTTTTTTTTTTTTTTTTTTTTTTTTTTTTTTTTTTTTTTTTTTTTTTTTTT
\begin{lem}  \label{lem9955P.400T}\hypertarget{lem9955P.400T} If $p$ is a large prime, then
	$$\frac{\varphi(p-1)}{p}\gg\frac{1}{\log \log p}.$$
\end{lem}
\begin{proof}[\textbf{Proof}] For any prime $p$, the ratio $\varphi(p-1)/p$ can be rewritten as a product over the prime
	\begin{equation}\label{eq9955P.400Tc}
		\frac{\varphi(p-1)}{p}=		\frac{p-1}{p}\cdot \frac{\varphi(p-1)}{p-1}=\frac{p-1}{p}\prod_{r\mid p-1}\left( 1-\frac{1}{r}\right).
	\end{equation}	
	where $r\geq2$ ranges over the prime divisor of $p-1$. This step follows from the identity $\varphi(n)/n=\prod_{r\mid n}\left( 1-1/r\right)$, where $r\geq2$ ranges 
	over the prime divisors of $n$. Since the number $p-1$ has fewer than $ 2\log p$ prime divisors, see {\color{red}\cite[Theorem 2.10]{MV2007}}, let $ x=2\log p$. Then, an application of the lower bound of the product given in {\color{red}\cite[Theorem 6.12]{DP2016}} yields 
	\begin{eqnarray}\label{eq9955P.400Tf}
		\frac{\varphi(p-1)}{p}&\geq&\frac{p-1}{p}\prod_{r\leq 2\log p}\left( 1-\frac{1}{r}\right)\\[.3cm]
		&>&\frac{p-1}{p}\cdot \frac{e^{-\gamma}}{\log( 2\log p)}\left(1-\frac{0.2}{(\log (2\log p))^2} \right)\nonumber\\[.3cm]
		& \gg&\frac{1}{\log \log p}>0 \nonumber,
	\end{eqnarray}	
	where $\gamma>0$ is Euler constant.	
\end{proof}

The extreme values of the prime divisors counting function $\omega(p-1)$ occur on the set of primorial-type primes $\mathcal{P}=\{p=2^{v_2}\cdot3^{v_3}\cdots p_k^{v_k} +1\}$, where $v_p\geq1$ is the $p$-adic valuation. A \textit{primorial} prime has the form $p=2\cdot3\cdots p_k +1$, where $p_k$ denotes the $k$th prime in increasing order. The primorial primes are very rare, see \cite{CG2002} for some data. In fact, it is unknown whether or not the set $\mathcal{P}$ is finite or infinite. \\

An alternative result for the lower bound of the ratio $\varphi(n)/n$ appears in {\color{red}\cite[Theorem 2.9]{MV2007}}. Another technique is employed below to derive an explicit lower bound for all integers $n\geq10$. 

\begin{lem}  \label{lem9955P.400TL}\hypertarget{lem9955P.400TL} If $n\geq5$ is a large integer, then
	$$\frac{\varphi(n)}{n}\geq\frac{3}{e^{\gamma}\pi^2}\frac{1}{\log \log n},$$
	where $\gamma>0$ is Euler constant.
\end{lem}
\begin{proof}[\textbf{Proof}] Start with the sigma-phi identity
	\begin{equation}\label{eq9955P.400Tm}
		\frac{\varphi(n)}{n}	\frac{\sigma(n)}{n}=\prod_{p^{v}\mid \mid n	}\left( 1-\frac{1}{p^{v+1}}\right) ,
	\end{equation}	
	where $p^{v}\mid \mid n$ is the maximal prime power divisor of $n$.  Since the sum of divisors function has the unconditional upper bound
	\begin{equation}\label{eq9955P.400To}
		\sigma(n)=\sum_{d\mid n}d\leq 2e^{\gamma}n\log\log n
	\end{equation} 
	for all integers $n\geq 5$, it follows that the phi function has the lower bound
	\begin{eqnarray}\label{eq9955P.400Tp}
		\frac{\varphi(n)}{n}&=&	\frac{n}{\sigma(n)}\prod_{p^{v}\mid \mid n	}\left( 1-\frac{1}{p^{v+1}}\right) \\
		&\geq&	\frac{n}{2e^{\gamma}n\log\log n}\prod_{p\geq 2	}\left( 1-\frac{1}{p^{2}}\right)\nonumber \\[.2cm]
		&\geq&	\frac{1}{2e^{\gamma}\log\log n}\cdot \frac{6}{\pi^2}\nonumber 
	\end{eqnarray}		
	as claimed.	
\end{proof}
The constant is approximately
\begin{equation}\label{eq9955P.400Tq}
	\frac{3}{e^{\gamma}\pi^2}=	0.17066321832663749538636772818453491728130\ldots.
\end{equation}
The upper bound of the sum of divisors function $\sigma(n)$, displayed in \eqref{eq9955P.400To}, is a classical result, see {\color{red}\cite[Theorem 323]{HW1979}}. The maximization at the subset colossally abundant integers is proved in \cite{AE1944}, \cite{RS1997} et cetera. Similarly, the phi function $\varphi(n)$ is minimized at the subset of colossally abundant integers.

%SSSSSSSSSSSSSSSSSSSSSSSSSSSSSSSSSSSSSSSSSSSSSSSSSSSSSSSSSSSSSSSSSSSSSSSSSSSSSSSSSSS
%SSSSSSSSSSSSSSSSSSSSSSSSSSSSSSSSSSSSSSSSSSSSSSSSSSSSSSSSSSSSSSSSSSSSSSSSSSSSSSSSSSS
%SSSSSSSSSSSSSSSSSSSSSSSSSSSSSSSSSSSSSSSSSSSSSSSSSSSSSSSSSSSSSSSSSSSSSSSSSSSSSSSSSSS
%SSSSSSSSSSSSSSSSSSSSSSSSSSSSSSSSSSSSSSSSSSSSSSSSSSSSSSSSSSSSSSSSSSSSSSSSSSSSSSSSSSS
%SSSSSSSSSSSSSSSSSSSSSSSSSSSSSSSSSSSSSSSSSSSSSSSSSSSSSSSSSSSSSSSSSSSSSSSSSSSSSSSSSSS
%SSSSSSSSSSSSSSSSSSSSSSSSSSSSSSSSSSSSSSSSSSSSSSSSSSSSSSSSSSSSSSSSSSSSSSSSSSSSSSSSSSS
%SSSSSSSSSSSSSSSSSSSSSSSSSSSSSSSSSSSSSSSSSSSSSSSSSSSSSSSSSSSSSSSSSSSSSSSSSSSSSSSSSSS
\section{Totient Functions over Dedekind Domains}\label{S7171TFDD-A}\hypertarget{S7171TFDD-A}
This section introduced one of the essential part in the generalization of the totient function over integral domains. 
Let $\mathcal{D}$ be a Dedekind domain. An element $\mathfrak{n}\in\mathcal{D}$ has a decomposition as a product of prime ideals $\mathfrak{n}=\prod_{\mathfrak{p}\mid \mathfrak{n}} \mathfrak{p}^{v}$, where $\mathfrak{p}\in \mathcal{D}$ ranges over the prime ideals, and $v\geq0$.\\

For a pair of ideals $\mathfrak{m}$ and $\mathfrak{n}$ in a Dedekind domain $\mathcal{D}$ the greatest common divisor is denoted by $\gcd(\mathfrak{m},\mathfrak{n})$. Similarly, the lowest common multiple is denoted by $\lcm(\mathfrak{m},\mathfrak{n})$, see {\color{red}\cite[p.\ 8]{NW2004}} or similar reference. The norm is a real valued function $N: \mathcal{D} \longrightarrow \R$ denoted by $N(\mathcal{I})$.\\

\begin{dfn}\label{dfn7171TFDD.200c}\hypertarget{dfn7171TFDD.200c} {\normalfont For any element $\mathfrak{n}\in \mathcal{D}$, the Mobius function is defined by
	\begin{equation} \label{eq7171TFDD.200c}
		\mu_\mathcal{D}(\mathfrak{n})=
		\begin{cases}
			0 & \text{ if some prime ideal }\mathfrak{p}^2\mid \mathfrak{n}, \\
			1 & \text{ if no prime ideal }\mathfrak{p}^2\mid \mathfrak{n}, \\
		\end{cases}
	\end{equation}
}
\end{dfn}

The totient function over many domains have similar properties. 
\begin{dfn}\label{D-totient.100} {\normalfont Let $\mathcal{R}$ be a Dedekind domain, and let $\mathcal{I}$ be an ideal. The Euler totient function over the Dedekind domain $\mathcal{R}$ is defined by $$\Phi (\mathcal{I})=\sum_{\substack{\mathfrak{a}\in \mathcal{R}/\mathcal{I} \\\gcd(\mathfrak{a},\mathcal{I})=1}}1.$$ It counts the number of invertible elements $\mathfrak{a}\ne0$ in the factor ring $\mathcal{R}/\mathcal{I}$.
}
\end{dfn}

The additive identity, multiplicative identity, and other identities have analogous versions in this domain.

\begin{lem}  \label{lem7171TFDD.200f}\hypertarget{lem7171TFDD.200f} The totient function has the following representations.
	\begin{enumerate} [font=\normalfont, label=(\roman*)]
		\item  
		$\displaystyle 	
		\Phi (\mathcal{I})=N(\mathcal{I})\sum _{\mathfrak{d} \mid  \mathcal{I}} \frac{\mu (\mathfrak{d})}{N(\mathfrak{d})},$\tabto{9cm} additive identity.
		\item  
		$\displaystyle 	
		\Phi (\mathcal{I})=N(\mathcal{I})\prod_{\mathfrak{p} \mid  \mathcal{I}}\left (1-\frac{1}{N(\mathfrak{p})}\right ),$\tabto{9cm} multiplicative identity.
		\item  
		$\displaystyle 	\Phi (\mathfrak{n})=\sum _{1\leq N(\mathfrak{m})<N(\mathfrak{n}) ,} \sum _{\substack{\mathfrak{d} \mid \mathfrak{m} \\ \mathfrak{d} \mid \mathfrak{n}}} \mu (\mathfrak{d})       ,$\tabto{9cm} counting identity.
		\item  
		$\displaystyle 	\Phi (\mathfrak{n})= \sum _{ \mathfrak{d} \mid \mathfrak{n}} \Phi  (\mathfrak{d})       ,$\tabto{9cm} self-linearity relation identity.
	\end{enumerate}
\end{lem}

\begin{proof}[\textbf{Proof}] (iii) The standard approach is to employ the characteristic function of relatively prime ideals
	\begin{equation} \label{eq7171TFDD2103}
		\sum _{\mathfrak{d} \mid  \mathfrak{n}} \mu (\mathfrak{d}) =
		\begin{cases}
			1 & \gcd (\mathfrak{d},\mathfrak{n})=1, \\
			0 & \gcd (\mathfrak{d},\mathfrak{n})\neq 1, \\
		\end{cases}
	\end{equation}
	to derive this identity.
\end{proof}

\begin{lem}  \label{lem7171TFDD.200r}\hypertarget{lem7171TFDD.200r} The inverse totient function has the following representations.
	\begin{enumerate} [font=\normalfont, label=(\roman*)]
		\item  
		$\displaystyle 	
		\frac{1}{\Phi (n)}=\frac{1}{N(\mathfrak{n})}\sum _{\mathfrak{d} \mid  \mathfrak{n}} \frac{\mu^2 (\mathfrak{d})}{\Phi (\mathfrak{d})},$\tabto{8cm} additive inverse identity.
		\item  
		$\displaystyle 	
		\frac{1}{\Phi (\mathfrak{n})}=\frac{1}{N(\mathfrak{n})}\prod _{\mathfrak{p} \mid  \mathfrak{n}} \left (1-\frac{1}{N(\mathfrak{p})}\right )^{-1},$\tabto{8cm} multiplicative inverse identity.
	\end{enumerate}
\end{lem}
\begin{lem} \label{lem1771TFDD.200t}\hypertarget{lem1771TFDD.200t} Given a pair of elements $\mathfrak{m},\mathfrak{n} \in \mathcal{D}$, the totient function satisfies the followings relations.
	\begin{enumerate} [font=\normalfont, label=(\roman*)]
		\item  
		$\displaystyle \Phi(\mathfrak{m}\mathfrak{n})=\Phi(\mathfrak{m})\Phi(\mathfrak{n})$, \tabto{10cm} \text {\normalfont if}  $\gcd(\mathfrak{m},\mathfrak{n})=1.$
		\item 
		$\displaystyle \Phi(\mathfrak{m}\mathfrak{n})=\frac{N(\mathfrak{d})}{\Phi(\mathfrak{d})}\Phi(\mathfrak{m})\Phi(\mathfrak{n})$, \tabto{10cm} \text {\normalfont where} $ \gcd(\mathfrak{m},\mathfrak{n})=\mathfrak{d}.$
	\end{enumerate}
\end{lem}
%EEEEEEEEEEEEEEEEEEEEEEEEEEEEEEEEEEEEEEEEEEEEEEEEEEEEEEEEEEEEEEEEEEEEEEE
%EEEEEEEEEEEEEEEEEEEEEEEEEEEEEEEEEEEEEEEEEEEEEEEEEEEEEEEEEEEEEEEEEEEEEEE
\begin{exa}\label{exa7171TFDD.200j1}\hypertarget{exa7171TFDD.200j}{\normalfont Let $q=p^m$ be a prime power and $\mathcal{R}_q=\F_q[x]/\langle f(x)\rangle$. If $f(x)\in \F_q[x]$ is irreducible of degree $\deg f=n$, then the totient function over the finite ring $\mathcal{R}_q$ is defined as follows. 
		\begin{align}\label{eq1771TFDD.200exei1}
			\Phi_q(f)&=\#\{r(x)\in \F_q[x]:\gcd(f,r)=c\}\\[.2cm]
			&=N(f)\prod_{r(x) \mid  f(x)}\left (1-\frac{1}{N(r)}\right )\nonumber \\[.2cm]
			&=N(f)-1,\nonumber 
		\end{align}		
where the polynomial $r$ ranges over the irreducible factors of $f$, $N(f)=q^{\deg f}$ is the norm of $f$ and $c\in \F_q^{\times}$ is a constant, see \hyperlink{lef7171TFDD.200f}{Lemma} \ref{lem7171TFDD.200f}. Observe that $\mathcal{R}_q^{\times}$ is the multiplicative group of units and $\#\mathcal{R}_q^{\times}=N(f)-1=q^n-1$.\\
	}		
\end{exa}
%\vskip .15 in 
Earlier applications of the example above appears in {\color{red}\cite[Theorem 11]{OO1934}}, \cite{CL1952A}, \cite{DH1968}, {\color{red}\cite[p.\;219]{LS1987}}, et alia. \\

%EEEEEEEEEEEEEEEEEEEEEEEEEEEEEEEEEEEEEEEEEEEEEEEEEEEEEEEEEEEEEEEEEEEEEEE
%EEEEEEEEEEEEEEEEEEEEEEEEEEEEEEEEEEEEEEEEEEEEEEEEEEEEEEEEEEEEEEEEEEEEEEE
%EEEEEEEEEEEEEEEEEEEEEEEEEEEEEEEEEEEEEEEEEEEEEEEEEEEEEEEEEEEEEEEEEEEEEEE
%EEEEEEEEEEEEEEEEEEEEEEEEEEEEEEEEEEEEEEEEEEEEEEEEEEEEEEEEEEEEEEEEEEEEEEE
\begin{exa}\label{exa7171TFDD.200j2}\hypertarget{exa7171TFDD.200j}{\normalfont Let $q=p^m$ be a prime power and $\mathcal{R}_q=\F_q[x]/\langle f(x)\rangle$. If $g(x)\in \F_q[x]$, then the totient function over the finite ring $\mathcal{R}_q$ is defined as follows. 
		\begin{align}\label{eq1771TFDD.200exei2}
			\Phi_q(g)&=\#\{r(x)\in \F_q[x]:\gcd(g,r)=c\}\\[.2cm]
			&=N(g)\prod_{r(x) \mid  g(x)}\left (1-\frac{1}{N(r)}\right ),\nonumber 
		\end{align}		
where the polynomial $r(x)$ ranges over the irreducible factors of $g(x)$, $N(g)=q^{\deg g}$ is the norm of $g(x)$ and $c\in \F_q^{\times}$ is a constant,	see \hyperlink{lem7171TFDD.200f}{Lemma} \ref{lem7171TFDD.200f}.
	}		
\end{exa}
%EEEEEEEEEEEEEEEEEEEEEEEEEEEEEEEEEEEEEEEEEEEEEEEEEEEEEEEEEEEEEEEEEEEEEEE
%EEEEEEEEEEEEEEEEEEEEEEEEEEEEEEEEEEEEEEEEEEEEEEEEEEEEEEEEEEEEEEEEEEEEEEE

%LLLLLLLLLLLLLLLLLLLLLLLLLLLLLLLLLLLLLLLLLLLLLLLLLLLLLLLLLLLLLLLLLLLLLLLLLLLLLLLL
%LLLLLLLLLLLLLLLLLLLLLLLLLLLLLLLLLLLLLLLLLLLLLLLLLLLLLLLLLLLLLLLLLLLLLLLLLLLLLLLL
\begin{lem} \label{lem1771TFDD.300FF}\hypertarget{lem1771TFDD.300FF}{Lemma} If $q$ is an odd prime power and $n\geq2$, then the number of normal elements contained in the finite field $\F_{q^n}$ is 
	 $$\Phi(x^n-1)=q^n\prod_{d(x)\mid x^n-1}\left(1-\frac{1}{q^{\deg d(x)}} \right), $$
where $d(x)$ varies over the irreducible factors of $x^n-1$ .
\end{lem}

In the smallest such finite fields $\F_{q^2}$, this counting function has the value
\begin{equation}\label{eq1771TFDD.200exei5}
	\Phi_q(x^2-1)=q^2\left (1-\frac{1}{q}\right )^2 =(q-1)^2
\end{equation}		
since  $x-1,x+1\mid x^2-1$. For example, $\F_{2^2}=\{0,1,\alpha, \alpha^2:\alpha^2+\alpha+1=0\}$ has a single normal element, namely $\alpha+1$. This is precisely as predicted $\Phi_q(x^2-1)=(q-1)^2=1$ normal element. \\

An integer version of this result for the umber of normal elements in finite fields suitable for computer calculation is also known. 

%LLLLLLLLLLLLLLLLLLLLLLLLLLLLLLLLLLLLLLLLLLLLLLLLLLLLLLLLLLLLLLLLLLLLLLLLLLLLLLLL
%LLLLLLLLLLLLLLLLLLLLLLLLLLLLLLLLLLLLLLLLLLLLLLLLLLLLLLLLLLLLLLLLLLLLLLLLLLLLLLLL
\begin{lem} \label{lem1771TFDD.300FF-I}\hypertarget{lem1771TFDD.300FF-I} Let $p = \text{char }\F_q$ and let $n = mp^v$ where $p$ does not divide $m$. For $d$ relatively prime to $q$ let $\ord_d(q)$ denote the multiplicative order of $q \bmod d$, and let $\varphi$ be the Euler totient function. Then, the finite field $\F_{q^n}$ contains
	
	 $$\Phi(x^n-1)=q^n\prod_{d\mid m}\left(1-\frac{1}{q^{\ord_d q}} \right)^{\varphi(d)/\ord_d q} $$ normal elements.
\end{lem}
A completely detailed proof appears in {\color{red} \cite[Theorem 1]{HT2018}}. Apparently, the normal base theorem for finite fields was conjecture by  Eisenstein and Schonemann circa 1850. The complete proof was given by Hensel in \cite{HK1888}. \\

%\vskip .1 in 
This topic is important in the theory of primitive roots over certain rings and integral domains. There is an extensive literature on the generalization of the primitive root conjecture to number fields and function fields, consult \cite{HJ1983}, \cite{HJ1986}, \cite{RN1999}, \cite{ZV2006}, \cite{PS1995}, \cite{RH2002}, \cite{SP2021}, and the references within. \\
%CCCCCCCCCCCCCCCCCCCCCCCCCCCCCCCCCCCCCCCCCCCCCCCCCCCCCCCCCCCCCCCCCCCCCCCCCCCCCCC
%CCCCCCCCCCCCCCCCCCCCCCCCCCCCCCCCCCCCCCCCCCCCCCCCCCCCCCCCCCCCCCCCCCCCCCCCCCCCCCC
Assuming $n\mid q-1$, the probability of normal elements in finite field has the exact form
\begin{align}\label{eq1771TFDD.300FF0}
\frac{\Phi(x^n-1)}{q^n-1}
&=\frac{q^n}{q^n-1}\left(1-\frac{1}{q} \right)^{n
}\\[.3cm]
&=\frac{q^n}{q^n-1}e^{n\log \left(1-\frac{1}{q} \right)}\nonumber\\[.3cm]
&=1-\frac{n}{ q}+O\left (\frac{n(n-1}{ q^2}\right )\nonumber.
\end{align} 
A different asymptotic for the lower bound suitable for all $n\in [2,q-1]$ and in terms of $\log q^n$ is provided here. 
\begin{cor} \label{cor1771TFDD.300FF-A}\hypertarget{cor1771TFDD.300FF-A} Let $q=p^k\geq 2^k$ be a prime power and let $n\geq2$. Then $$\frac{\Phi(x^n-1)}{q^n-1}\gg \frac{1}{\log q^{n}} $$
	uniformly for all $q\geq 2$, $k\geq 1$ and $n\geq 2.$
\end{cor}
\begin{proof}[\textbf{Proof}] There are two important cases indexed by $n\mid q-1$ and $n\nmid q-1$. The case $n\mid q-1$ specifies the lowest bound possible. Thus, to derive a lower bound in terms of $\log q^n$, it is sufficient to consider this case. Toward this goal, suppose that $n\mid q-1$. Furthermore, this case implies that $n<q$, the finite field contains all the $n$th roots of unity, and the polynomial $x^n-1\in \F_q[x]$ splits into linear factors. Next, by \hyperlink{lem1771TFDD.300FF}{Lemma} \ref{lem1771TFDD.300FF} there is a lower bound of the form
\begin{align}\label{eq1771TFDD.300FF1}
\frac{\Phi(x^n-1)}{q^n-1}&=\frac{q^n}{q^n-1}\prod_{d(x)\mid x^n-1}\left(1-\frac{1}{q^{\deg d(x)}} \right)\\[.3cm]
&=\frac{q^n}{q^n-1}\left(1-\frac{1}{q} \right)^{n}\nonumber\\[.3cm]
&\gg\frac{1}{\log q^n}\nonumber.
\end{align}
Set $x=q>0$ and let $n\in [2,x-1]$ be an integer. Then, the conclusion follows from the nonnegativity of the function
\begin{equation}
f(n,x)=\left(1-\frac{1}{x} \right)^{n}-\frac{1}{5\log x^{n}}>0
\end{equation}
over the domain $[2,x-1]\times[2,\infty]$. 
\end{proof}	
%CCCCCCCCCCCCCCCCCCCCCCCCCCCCCCCCCCCCCCCCCCCCCCCCCCCCCCCCCCCCCCCCCCCCCCCCCCCCCCC
%CCCCCCCCCCCCCCCCCCCCCCCCCCCCCCCCCCCCCCCCCCCCCCCCCCCCCCCCCCCCCCCCCCCCCCCCCCCCCCC
\begin{cor} \label{cor1771TFDD.300FF-C}\hypertarget{cor1771TFDD.300FF-C} Let $q=p^k\geq 2^k$ be a prime power and let $n\geq2$. Then $$\frac{\Phi(x^n-1)}{q^n-1}\gg \frac{1}{(\log\log q^{n})} $$
	uniformly for all $q\geq 8$, $k\geq 1$ and $n\geq 2.$
\end{cor}
\begin{proof}[\textbf{Proof}] The proof is similar to the previous one.
\end{proof}
%SSSSSSSSSSSSSSSSSSSSSSSSSSSSSSSSSSSSSSSSSSSSSSSSSSSSSSSSSSSSSSSSSSSSSSSSSSSSSSSSSSS
%SSSSSSSSSSSSSSSSSSSSSSSSSSSSSSSSSSSSSSSSSSSSSSSSSSSSSSSSSSSSSSSSSSSSSSSSSSSSSSSSSSS
%SSSSSSSSSSSSSSSSSSSSSSSSSSSSSSSSSSSSSSSSSSSSSSSSSSSSSSSSSSSSSSSSSSSSSSSSSSSSSSSSSSS
%SSSSSSSSSSSSSSSSSSSSSSSSSSSSSSSSSSSSSSSSSSSSSSSSSSSSSSSSSSSSSSSSSSSSSSSSSSSSSSSSSSS
%SSSSSSSSSSSSSSSSSSSSSSSSSSSSSSSSSSSSSSSSSSSSSSSSSSSSSSSSSSSSSSSSSSSSSSSSSSSSSSSSSSS
%SSSSSSSSSSSSSSSSSSSSSSSSSSSSSSSSSSSSSSSSSSSSSSSSSSSSSSSSSSSSSSSSSSSSSSSSSSSSSSSSSSS
%SSSSSSSSSSSSSSSSSSSSSSSSSSSSSSSSSSSSSSSSSSSSSSSSSSSSSSSSSSSSSSSSSSSSSSSSSSSSSSSSSSS
\section{Estimates of Omega Functions} \label{S4545EOFW-W}\hypertarget{S4545EOFW-W}
%TTTTTTTTTTTTTTTTTTTTTTTTTTTTTTTTTTTTTTTTTTTTTTTTTTTTTTTT
The little omega function $\omega:\N\longrightarrow \N$, defined by $\omega(n)=\#\{p\mid n:p \text{ is prime }\}$, is a well studied function in number theory. It has the explicit upper bound $\omega(n)\leq 2\log n/\log\log n$, see {\color{red}\cite[Proposition 7.10]{DL2012}} {\color{red}\cite[Theorem 2.10]{MV2007}} and \cite{EP1985}. It emerges in the finite sum
\begin{equation} \label{eq4545EOFW400j}
	\sum _{d\mid n}|\mu(n)| 
	=2^{\omega(n)}		
	\ll n^{\varepsilon},
\end{equation}
where $\varepsilon>$ is an arbitrarily small number, which is frequently used in finite field analysis. The corresponding estimates for the big omega function are also important in finite field analysis.

\begin{dfn}\label{dfn4545EOFW.300d}\hypertarget{dfn4545EOFW.300d} {\normalfont Let $q$ be a prime power and let $\F_q[x]$ be the ring of polynomials over the finite field $\F_q$. The big omega function $\Omega_q:\N\longrightarrow \N$ is defined by $$\Omega_q(f(x))=\#\{r(x)\mid f(x):r(x)\text{ is irreducible }\}.$$
} 
\end{dfn}
For any polynomial $f(x)\in \F_q[x]$ the trivial extreme values $1\leq \Omega_q(f(x))\leq n$ are clear, but sharper estimates are possible too. An estimate for the frequently used polynomial $x^n-1$ can be derived from its factorization in term of cyclotomic polynomials
\begin{equation} \label{eq4545EOFW.300d}
x^n-1=\prod_{d\mid n}\phi_d(x),
\end{equation} 
where $\phi_d(x)$ is the $d$th cyclotomic polynomial, and the regular structure of these polynomials. In a separable finite field of characteristic $q\mid n$ the nontrivial $q$th roots unity do not exist. Thus   
\begin{equation} \label{eq4545EOFW.300e}
	x^n-1=	(x^q)^{n/q}-1=\prod_{d\mid n/q}\phi_d(x^q).
\end{equation} 
\begin{lem}\label{lem4545EOFW.300c}\hypertarget{lem4545EOFW.300c} Let $n\geq 1$ and the prime power $q$ be relatively prime integers, and let $\ord_n q$ be the multiplicative order of $q$ modulo $n$. Then
	\begin{enumerate}[font=\normalfont, label=(\roman*)]
		\item $\displaystyle \phi_n(x)=f_1(x)f_2(x)\cdot f_t(x)$, where $\deg f_i=\ord_n q$ and $t=\varphi(n)/\ord_n q$.
	\end{enumerate}
\end{lem}
\begin{proof}[\textbf{Proof}] (ii) In light of the cyclotomic polynomial decomposition of the  polynomial $x^n-1$ in \eqref{eq4545EOFW.300d}, it is clear that
	\begin{eqnarray}\label{eq4545EOFW.300h}
		\Omega_q(x^n-1)&=& \Omega_q\left(\prod_{d\mid n}\phi_d(x) \right) \\[.2cm]
		&=&\sum_{d\mid n}\varphi(d)/\ord_d q\nonumber,
	\end{eqnarray} 
	since each polynomial $\phi_d(x)$ has $\varphi(d)/\ord_d q$ irreducible factor in $\F_q[x].$
\end{proof}

\begin{lem}\label{lem4545EOFW.300d}\hypertarget{lem4545EOFW.300d} Let $n\geq 1$ and the prime power $q$ be relatively prime integers, and let $\ord_d q$ be the multiplicative order of $q$ modulo $d$. Then
\begin{enumerate}[font=\normalfont, label=(\roman*)]
\item $\displaystyle \Omega_q(x^n-1)=2$\tabto{7cm}if $n$ is prime and $\ord_n q=n-1$.
\item $\displaystyle \Omega_q(x^n-1)=\sum_{d\mid n}\varphi(d)/\ord_d q,$
\tabto{7cm}if $n$ is any integer and $\ord_n q\geq1$.
\item $\displaystyle \Omega_q(x^n-1)\leq\varphi(n),$
\tabto{7cm}if $n$ is any integer and $\ord_n q\geq1$.
\end{enumerate}
\end{lem}
\begin{proof}[\textbf{Proof}] (ii) In light of the cyclotomic polynomial decomposition of the  polynomial $x^n-1$ in \eqref{eq4545EOFW.300d} and \hyperlink{lem4545EOFW.300c}{Lemma} \ref{lem4545EOFW.300c}, it is clear that for each $d\mid n$, the polynomial $\phi_d(x)$ has $t=\varphi(d)/\ord_d q$ irreducible factors $d_i(x)\in \F_q[x]$ of degree $\deg d_i=\ord_d q$, where $i\in [1,t]$. Hence,
\begin{eqnarray}\label{eq4545EOFW.300i}
\Omega_q(x^n-1)&=& \Omega_q\left(\prod_{d\mid n}\phi_d(x) \right) \\[.2cm]
&=&\sum_{d\mid n}\varphi(d)/\ord_d q\nonumber.
\end{eqnarray} 
(iii) This follows from 
\begin{equation} \label{eq4545EOFW.300k}
\sum_{d\mid n}\frac{\varphi(d)}{\ord_d q}\leq\sum_{d\mid n}\varphi(d)=\varphi(n),
\end{equation}
since $\ord_d q\geq1$ for any pair of integers $d$ and $q$.
\end{proof}
%%%%%%%%%%%%%%%%%%%%%%%%%%%%%%%%%%%%%%%%%%%%%%%%%%%%%%%%%%%
%%%%%%%%%%%%%%%%%%%%%%%%%%%%%%%%%%%%%%%%%%%%%%%%%%%%%%%%%%%%%

%SSSSSSSSSSSSSSSSSSSSSSSSSSSSSSSSSSSSSSSSSSSSSSSSSSSSSSSSSSSSSSSSSSSSSSSSSSSSSSSSSSS
%SSSSSSSSSSSSSSSSSSSSSSSSSSSSSSSSSSSSSSSSSSSSSSSSSSSSSSSSSSSSSSSSSSSSSSSSSSSSSSSSSSS
%SSSSSSSSSSSSSSSSSSSSSSSSSSSSSSSSSSSSSSSSSSSSSSSSSSSSSSSSSSSSSSSSSSSSSSSSSSSSSSSSSSS
%SSSSSSSSSSSSSSSSSSSSSSSSSSSSSSSSSSSSSSSSSSSSSSSSSSSSSSSSSSSSSSSSSSSSSSSSSSSSSSSSSSS
%SSSSSSSSSSSSSSSSSSSSSSSSSSSSSSSSSSSSSSSSSSSSSSSSSSSSSSSSSSSSSSSSSSSSSSSSSSSSSSSSSSS
%SSSSSSSSSSSSSSSSSSSSSSSSSSSSSSSSSSSSSSSSSSSSSSSSSSSSSSSSSSSSSSSSSSSSSSSSSSSSSSSSSSS
%SSSSSSSSSSSSSSSSSSSSSSSSSSSSSSSSSSSSSSSSSSSSSSSSSSSSSSSSSSSSSSSSSSSSSSSSSSSSSSSSSSS
\section{Characteristic Functions for Primitive Elements}
\label{S7171CFPE-P} \hypertarget{S7171CFPE-P}
A primitive element generates the $\Z$-module

\begin{equation}\label{eq7171CFNE.100a}
	\F_{q^n}^{\times}\cong	\Z/(q^n-1)\Z.
\end{equation}	
The group of units $\left( \Z/(q^n-1)\Z\right)^{\times}$ has precisely $\varphi(q^n-1)$ units, and each unit is a primitive element. There several possible techniques that can be used to construct characteristic functions of primitive elements in the group of units $\left( \Z/(q^n-1)\Z\right)^{\times}$. The standard characteristic function for primitive elements in finite field $\F_{q^n}$ dependents on the factorization of the integer $q^n-1$. Whereas a new divisor-free characteristic function for primitive elements introduced here is not dependent on the factorization of the integer $q^n-1$. The basic analytic principles of these indicator functions are presented below. 
%SSSSSSSSSSSSSSSSSSSSSSSSSSSSSSSSSSSSSSSSSSSSSSSSSSSSSSSSSSSSSSSS
%SSSSSSSSSSSSSSSSSSSSSSSSSSSSSSSSSSSSSSSSSSSSSSSSSSSSSSSSSSSSSSSS
%SSSSSSSSSSSSSSSSSSSSSSSSSSSSSSSSSSSSSSSSSSSSSSSSSSSSSSSSSSSSSSSS
%SSSSSSSSSSSSSSSSSSSSSSSSSSSSSSSSSSSSSSSSSSSSSSSSSSSSSSSSSSSSSSSS
%SSSSSSSSSSSSSSSSSSSSSSSSSSSSSSSSSSSSSSSSSSSSSSSSSSSSSSSSSSSSSSSS
%SSSSSSSSSSSSSSSSSSSSSSSSSSSSSSSSSSSSSSSSSSSSSSSSSSSSSSSSSSSSSSSS
%SSSSSSSSSSSSSSSSSSSSSSSSSSSSSSSSSSSSSSSSSSSSSSSSSSSSSSSSSSSSSSSS
%SSSSSSSSSSSSSSSSSSSSSSSSSSSSSSSSSSSSSSSSSSSSSSSSSSSSSSSSSSSSSSSS
%SSSSSSSSSSSSSSSSSSSSSSSSSSSSSSSSSSSSSSSSSSSSSSSSSSSSSSSSSSSSSSSS
%SSSSSSSSSSSSSSSSSSSSSSSSSSSSSSSSSSSSSSSSSSSSSSSSSSSSSSSSSSSSSSSS
%SSSSSSSSSSSSSSSSSSSSSSSSSSSSSSSSSSSSSSSSSSSSSSSSSSSSSSSSSSSSSSSS
\subsection{Divisor Dependent Characteristic Functions for Primitive Elements}
%\label{S7171CFPE-B} \hypertarget{S7171CFPE-B}
%\subsection{Divisors Dependent Characteristic Function}
A representation of the characteristic function dependent on the orders of the cyclic groups is given below. This representation is sensitive to the primes decompositions $q^n-1=p_1^{e_1}p_2^{e_2}\cdots p_t^{e_t}$, with $p_i$ prime and $e_i\geq1$, of the order of the cyclic group $q^n-1=\#\F_{q^n}^{\times}$. 

\begin{lem} \label{lem2727CFFR.100-F}\hypertarget{lem2727CFFR.100-F} If $q=p^k$ is a prime power and $\alpha\in \F_{q^n}$ is an invertible element in the cyclic group $\F_{q^n}^{\times}$, then
	\begin{equation}\label{eq2727CFFR.100d}
		\Psi (\alpha)=\frac{\varphi (q^n-1)}{q^n-1}\sum _{d \mid q^n-1} \frac{\mu (d)}{\varphi (d)}\sum _{\ord(\chi ) = d} \chi (\alpha)=
		\left \{\begin{array}{ll}
			1 & \text{ if } \ord_q (\alpha)=q^n-1,  \\
			0 & \text{ if } \ord_q (\alpha)\neq q^n-1. \\
		\end{array} \right .
	\end{equation}
\end{lem}

\begin{proof}[\textbf{Proof}] Assume that $u=\tau^{rm}$ is a $r$th power residue in $\F_{q^n}$, where $r\mid q^n-1$ and $\gcd(m,q^n-1)=1$. Then, the inner sum
	\begin{equation} 
		\sum _{ \ord(\chi) = r} \chi (u)= \sum _{ \ord(\chi) = r} \chi (\tau^{rm})=\sum _{ \ord(\chi) = r} \chi (\tau^{m})^r=\varphi(r)=r-1,
	\end{equation}
	where $\chi(v)^r=1$ for any $v\in \F_{q^n}^{\times}$. Replacing this information into the product
	\begin{eqnarray} 
		\frac{\phi(q^n-1)}{q^n-1} \sum_{d \mid q^n-1}\frac{\mu(d)}{\varphi(d)} \sum_{\ord(\chi)=d}\chi(u)
		&=&\frac{\phi(q^n-1)}{q^n-1} \prod_{r \mid q^n-1} \left (1- \frac{\sum_{\ord(\chi)=r}\chi(u)}{r-1} \right )  \nonumber \\
		&=&\frac{\phi(q^n-1)}{q^n-1} \prod_{r \mid q^n-1} \left (1- \frac{r-1}{r-1} \right )=0,
	\end{eqnarray}
where $r$ varies over the primes, shows that both sides of the equation vanish if the element $u \in \F_{q^n}^{\times}$ has order $\ord_q(u) =r \mid q^n-1$ and $r < q^n-1$. Now, assume that $u=\tau^{m}$ is not $r$th power residue in $\F_{q^n}$ for any $r \mid q^n-1$, where $\gcd(m,q^n-1)=1$. Then, the inner sum
	\begin{equation} 
		\sum _{ \ord(\psi) = r} \chi (u)= \sum _{ \ord(\psi) = r} \chi (\tau^{m})=-1.
	\end{equation}
	Replacing this information into the product
	\begin{eqnarray} 
		\frac{\phi(q^n-1)}{q^n-1} \sum_{d \mid q^n-1}\frac{\mu(d)}{\varphi(d)} \sum_{\ord(\chi)=d}\chi(u)
		&=&\frac{\phi(q^n-1)}{q^n-1} \prod_{r \mid q^n-1} \left (1- \frac{\sum_{\ord(\chi)=r}\chi(u)}{r-1} \right )  \nonumber \\
		&=&\frac{\phi(q^n-1)}{q^n-1} \prod_{r \mid q^n-1} \left (1- \frac{-1}{r-1} \right )=1 .
	\end{eqnarray}
	These verify that both sides of the equation vanishes if and only if the element $u \in \F_{q^n}^{\times}$ has order $\ord_q(u) =r \mid q^n-1$ and $q < q^n-1$.
\end{proof}

There are a few other variant proofs of this result, these are widely available in the literature, the proof given in {\color{red}\cite[p.\;221 ]{LS1987}} has an error. Almost every result in the theory of primitive roots in finite fields is based on this characteristic function, but sometimes written in different forms. \\

The authors in \cite{DH1937}, \cite{WR2001} attribute the simplest case for prime finite fields $\F_{p}$ of this formula to Vinogradov, \cite{VI1927}, and other authors attribute this formula to Landau, \cite{LE1927}. The proof and other details on the characteristic function are given in {\color{red}\cite[p. 863]{ES1957}}, {\color{red}\cite[p.\ 258]{LN1997}}, {\color{red}\cite[p.\ 18]{MP2007}}.

%SSSSSSSSSSSSSSSSSSSSSSSSSSSSSSSSSSSSSSSSSSSSSSSSSSSSSSSSSSSSSSSS
%SSSSSSSSSSSSSSSSSSSSSSSSSSSSSSSSSSSSSSSSSSSSSSSSSSSSSSSSSSSSSSSS
%SSSSSSSSSSSSSSSSSSSSSSSSSSSSSSSSSSSSSSSSSSSSSSSSSSSSSSSSSSSSSSSS
%SSSSSSSSSSSSSSSSSSSSSSSSSSSSSSSSSSSSSSSSSSSSSSSSSSSSSSSSSSSSSSSS
\subsection{Divisorfree Characteristic Functions for Primitive Elements}
\label{S7171CFPE-B} \hypertarget{S7171CFPE-B}
Let \(\tau\) be a primitive root in $\F_{q^n}$, let $\log_{\tau}\alpha$ be the discrete logarithm with respect to $\tau$, and $s\in \mathcal{S}=\{s<q^n:\gcd(s,q^n-1)=1\}$. The discrete logarithm is defined by the map
\begin{align}\label{eq7171CFNE.150PDLM}
	\F_{q^n}^{\times}&\longrightarrow \left(  \Z(q^n-1)/\Z\right)^{\times}\\[.3cm]
	\alpha &\longrightarrow \log_{\tau}\alpha.
\end{align}
A new \textit{divisors-free} representation of the characteristic function of primitive element is introduced in this section. It detects the order $\ord_{q^n} \alpha$
of the element $\alpha \in \F_{q^n}$ by means of the solutions of the equation \begin{equation}\label{eq7171CFNE.150PDF0}
	s-\log_{\tau}\alpha=0.
\end{equation}  

\begin{lem} \label{lem7171CFPE.150PDF} \hypertarget{lem7171CFNE.150PDF}  Let $q=p^k$ be a prime power, let \(\tau\) be a primitive root in $\F_{q^n}$ and let \(\psi \neq 1\) be a nonprincipal additive character of order $\ord  \psi =q^n$. If $\alpha \in \F_{q^n}$ is a nonzero element, then
	\begin{eqnarray}\label{eq7171CFPE.150PDFd}
		\Psi (\alpha)&=&\sum _{\substack{1\leq s\leq q^n-1\\\gcd (s,q^n-1)=1}} \frac{1}{q^n}\sum _{0\leq t\leq q^n-1} \psi \left ((s-\log_{\tau}\alpha)t\right)\\[.2cm]
		&=&\left \{
		\begin{array}{ll}
			1 & \text{   \normalfont if } \ord_{q^n} (\alpha)=q^n-1,  \\[.2cm]
			0 & \text{   \normalfont if } \ord_{q^n} (\alpha)\ne q^n-1. \\
		\end{array} \right .\nonumber
	\end{eqnarray}
\end{lem}	

\begin{proof} Set the additive character $\psi(t) =e^{i 2\pi at/q^n}\in \C$, where $a\ne0$ is an integer. As the index $s\in \mathcal{S}=\{s<q^n:\gcd(s,q^n-1)=1\}$ ranges over the integers relatively prime to $q^n-1$, the element $\tau ^s\in \F_{q^n}^{\times}$ ranges over the primitive roots in $\F_{q^n}$. Accordingly, the equation \eqref{eq7171CFNE.150PDF0} has a unique solution $s\in \mathcal{S}$ if and only if the fixed element $\alpha \in \F_{q^n}$ is a primitive root. This implies that the inner sum in \eqref{eq7171CFPE.150PDFd} collapses to 
	\begin{equation}\label{eq7171CFNE.150PDF2}
		\sum _{0\leq t\leq q^n-1} e^{\frac{i 2\pi(s-\log_{\tau}\alpha)t}{q^n}}=q^n.
	\end{equation}
	This in turns reduces \eqref{eq7171CFPE.150PDFd} to
	\begin{equation}\label{eq7171CFPE.150PDF7}
		\sum _{\substack{1\leq s\leq q^n-1\\\gcd (s,q^n-1)=1}} \frac{1}{q^n}\sum _{0\leq t\leq q^n-1} e^{\frac{i 2\pi(s-\log_{\tau}\alpha)t}{q^n}}=1.
	\end{equation} 	
	
	Otherwise, if the element $\alpha \in \F_{q^n}$ is not a primitive root, then the equation \eqref{eq7171CFNE.150PDF0} has no solution $s\in \mathcal{S}$. Thus, the inner sum in \eqref{eq7171CFPE.150PDFd} collapses to 
	\begin{equation}\label{eq7171CFPE.150PDF3}
		\sum _{0\leq t\leq q^n-1} e^{\frac{i 2\pi(s-\log_{\tau}\alpha)t}{q^n}}=0,
	\end{equation}
	this follows from the geometric series formula $\sum_{0\leq d\leq  x-1} r^d =(r^x-1)/(r-1)$, where $r=e^{i 2\pi a/x}\ne1$ and $x=q^n$. This in turns forces \eqref{eq7171CFPE.150PDFd} to vanish.
\end{proof}

Other versions of the divisorfree characteristic function for primitive elements are possible.
%SSSSSSSSSSSSSSSSSSSSSSSSSSSSSSSSSSSSSSSSSSSSSSSSSSSSSSSSSSSSSSSSSSSSSSSSSSSSSSSS
%SSSSSSSSSSSSSSSSSSSSSSSSSSSSSSSSSSSSSSSSSSSSSSSSSSSSSSSSSSSSSSSSSSSSSSSSSSSSSSSS
%SSSSSSSSSSSSSSSSSSSSSSSSSSSSSSSSSSSSSSSSSSSSSSSSSSSSSSSSSSSSSSSSSSSSSSSSSSSSSSSS
%SSSSSSSSSSSSSSSSSSSSSSSSSSSSSSSSSSSSSSSSSSSSSSSSSSSSSSSSSSSSSSSSSSSSSSSSSSSSSSSS
%SSSSSSSSSSSSSSSSSSSSSSSSSSSSSSSSSSSSSSSSSSSSSSSSSSSSSSSSSSSSSSSSSSSSSSSSSSSSSSSS
%SSSSSSSSSSSSSSSSSSSSSSSSSSSSSSSSSSSSSSSSSSSSSSSSSSSSSSSSSSSSSSSSSSSSSSSSSSSSSSSS
%SSSSSSSSSSSSSSSSSSSSSSSSSSSSSSSSSSSSSSSSSSSSSSSSSSSSSSSSSSSSSSSSSSSSSSSSSSSSSSSS
\section{Characteristic Functions for Normal Elements} \label{S7171CFNE-N}\hypertarget{S7171CFNE-N}
A normal element generates the $\F_q$-module

\begin{equation}\label{eq7171CFNE.150a}
	\F_{q^n}\cong\F_{q}[x]/(x^n-1)\F_{q}.
\end{equation}	
The group of units $\left( \F_{q}[x]/(x^n-1)\F_{q}\right)^{\times}$ has precisely $\Phi(x^n-1)$ units, and each unit is a normal element. There several possible techniques that can be used to construct characteristic functions of normal elements in the group of units $\left( \F_{q}[x]/(x^n-1)\F_{q}\right)^{\times}$. The standard characteristic function for normal elements dependents on the factorization of the polynomial $x^n-1$. Whereas a new divisor-free characteristic function for normal elements introduced here is not dependent on the factorization of the polynomial $x^n-1$. The basic analytic principles of these indicator functions are presented below. 
\subsection{Divisors Dependent Characteristic Functions for Normal Elements} \label{S7171CFNE-C}\hypertarget{S7171CFNE-C}
%\subsection{Divisors Dependent Characteristic Function}
A representation of the characteristic function of normal elements $\eta\in \F_q[x]/f(x)$ in finite rings is outline here. This representation is sensitive to the irreducible decompositions $x^n-1=p_1(x)^{e_1}p_2(x)^{e_2}\cdots p_k(x)^{e_k}$, with $p_i(x)\in\F_q[x]$ irreducible and $e_i\geq1$. 

\begin{lem} \label{lem7171CFNE.150b}\hypertarget{lem7171CFNE.150b}
	If \(\alpha\in \F_{q^n}\) be a nonzero element, then
	\begin{align}\label{eq7171CFNE.150d}
		\Psi_q (\alpha)&=\frac{\Phi_q (x^n-1)}{q^n}\sum _{d(x) \mid x^n-1} \frac{\mu_q (d(x))}{\Phi_q (d(x))}\sum _{\Ord(\psi ) = d(x)} \psi (\alpha)\\[.2cm]
		&=
		\left \{\begin{array}{ll}
			1 & \text{  \normalfont if } \Ord_q (\alpha)=x^n-1,  \\[.2cm]
			0 & \text{  \normalfont if } \Ord_q (\alpha)\neq x^n-1. \\
		\end{array} \right .\nonumber 
	\end{align}
\end{lem}

\begin{proof}[\textbf{Proof}]Since the arithmetic functions $\mu_q (f(x))$ and $\Phi_q (f(x))$ are multiplicative, consider the product version
	\begin{equation}\label{eq7171CFNE.150f}
	\Psi_q (\alpha)=\frac{\Phi_q (x^n-1)}{q^n}\prod _{r(x) \mid x^n-1}\left (1- \frac{1}{q^{\deg r}-1}\sum _{\Ord(\psi ) = r(x)} \psi (\alpha)\right ),
\end{equation}	
where $r(x)\mid x^n-1$ varies over the irreducible factors of degree $\deg r=d$. Here, the exponential sum over all the additive character of order $\Ord(\psi ) = r(x)$ reduces to
\begin{equation}\label{eq7171CFNE.150i}
\sum _{\Ord(\psi ) = r(x)} \psi (\alpha)
	=
	\left \{\begin{array}{ll}
		q^{\deg r}-1 & \text{ if } \Ord_q (\alpha)\equiv 0\bmod r(x),  \\[.2cm]
		-1& \text{ if } \Ord_q (\alpha)\not\equiv 0\bmod r(x), \\
	\end{array} \right.  
\end{equation}	
see \hyperlink{exe7575CFNE.082}{Exercise} \ref{exe7575CFNE.082}. Now, for any nonnormal element $\alpha \in\F_{q^n}$ of additive order $\Ord_q(\alpha)=d(x)\mid x^n-1$ such that $d(x)\ne x^n-1$, there is an additive character $\psi_(\alpha)$ of order $\Ord_q(\psi_q)=r(x)\mid d(x)$, where $r(x)\in \F_q[x]$ is irreducible. This leads to
	\begin{align}\label{eq7171CFNE.150j}
\Psi_q (\alpha)&=\frac{\Phi_q (x^n-1)}{q^n}\prod _{r(x) \mid x^n-1}\left (1- \frac{1}{q^{\deg r}-1}\sum _{\Ord(\psi ) = r(x)} \psi (\alpha)\right )\\[.2cm]
&=	\left \{\begin{array}{ll}
		1 & \text{  \normalfont if } \Ord_q (\alpha)=x^n-1,  \\[.2cm]
		0 & \text{  \normalfont if } \Ord_q (\alpha)\neq x^n-1 \\
	\end{array} \right .\nonumber 
\end{align}
as claimed, see \hyperlink{exe7171CFNE.012}{Exercise} \ref{exe7171CFNE.012}.
\end{proof}

The earliest development of this indicator function seems to be {\color{red}\cite[Theorem 11]{OO1934}} and {\color{red}\cite[Lemma 4]{CL1952A}}.
%SSSSSSSSSSSSSSSSSSSSSSSSSSSSSSSSSSSSSSSSSSSSSSSSSSSSSSSSSSSSSSSS
%SSSSSSSSSSSSSSSSSSSSSSSSSSSSSSSSSSSSSSSSSSSSSSSSSSSSSSSSSSSSSSSS
%SSSSSSSSSSSSSSSSSSSSSSSSSSSSSSSSSSSSSSSSSSSSSSSSSSSSSSSSSSSSSSSS
%SSSSSSSSSSSSSSSSSSSSSSSSSSSSSSSSSSSSSSSSSSSSSSSSSSSSSSSSSSSSSSSS
%SSSSSSSSSSSSSSSSSSSSSSSSSSSSSSSSSSSSSSSSSSSSSSSSSSSSSSSSSSSSSSSS
%SSSSSSSSSSSSSSSSSSSSSSSSSSSSSSSSSSSSSSSSSSSSSSSSSSSSSSSSSSSSSSSS
%SSSSSSSSSSSSSSSSSSSSSSSSSSSSSSSSSSSSSSSSSSSSSSSSSSSSSSSSSSSSSSSS
\subsection{Divisorfree Characteristic Functions for Normal Elements} %\label{S7171CFNE-N}\hypertarget{S7171CFNE-N}
%\subsection{Divisors Dependent Characteristic Function}
A new divisor-free representation of the characteristic function of normal elements $\eta\in \F_q[x]/f(x)$, where $f(x)\in\F_{q}[x]$ is a polynomial of degree $\deg f=n$, in finite rings is outline here. This representation is not sensitive to the irreducible decompositions $x^n-1=p_1(x)^{e_1}p_2(x)^{e_2}\cdots p_k(x)^{e_k}$, with $p_i(x)\in\F_q[x]$ irreducible and $e_i\geq1$. The result is expressed in terms of the discrete logarithm, defined in \eqref{eq7171CFNE.150PDLM} in \hyperlink{S7171CFPE-P}{Section } \ref{S7171CFPE-P} and a nontrivial additive character $\psi(t)=e^{i2\pi at/q^n}$ with $a\ne0$, see \hyperlink{S7171CFNE-DNE}{Section } \ref{S7171CFNE-DNE}.

\begin{lem} \label{lem7171CFNE.150NDF}\hypertarget{lem7171CFNE.150NDF} Let $q=p^k$ be a prime power and let $\eta \in \F_{q^n}$ be a normal element. If \(\alpha\in \F_{q^n}\) is a nonzero element, then
	\begin{align}\label{eq7171CFNE.150C}
		\Psi_q (\alpha)&=\sum _{\substack{\deg s(x)\leq n-1\\\gcd (s(x),x^n-1)=1}} \frac{1}{q^n}\sum _{0\leq t\leq q^n-1} e^{ \frac{i2\pi(\log_{\tau}s(x)\circ\eta-\log_{\tau}\alpha)t}{q^n}}\\[.2cm]
		&=
		\left \{\begin{array}{ll}
			1& \text{  \normalfont if } \Ord_q (\alpha)=x^n-1,  \\[.2cm]
			0 & \text{  \normalfont if } \Ord_q (\alpha)\neq x^n-1. \\
		\end{array} \right .\nonumber 
	\end{align}
	
\end{lem}

\begin{proof} Fix a normal element $\eta \in \F_{q^n}$ and let $\alpha \in \F_{q^n}$ be an arbitrary element. As  $s(x)=a_0+a_1x+a_2x^2+\cdots+a_{n-1}x^{n-1}\in\F_q[x]$ varies over the set polynomials of degree $\deg s<n$, the linear functional 
	\begin{equation}\label{eq7171CFNE.150D}
		s(x)\longrightarrow s(x)\circ\eta=a_0x+a_1x^q+a_2x^{q^2}+\cdots+a_{n-1}x^{q^{n-1}}
	\end{equation} generates the finite field $\F_{q^n}\cong\F_q[x]/f(x)$, see \hyperlink{S7171CFNE.200m}{Definition} \ref{dfn7171CFNE.200m}. Specifically, the element $s(x)\circ\eta$ is a normal element if and only if $\gcd(s(x),x^n-1)=1$, see \hyperlink{dfn2727PNT.200C}{Definition} \ref{dfn2727PNT.200C}. Accordingly, it follows that the equation 
	\begin{equation}\label{eq7171CFNE.150F}
		\log_{\tau}s(x)\circ\eta-\log_{\tau}\alpha=0
	\end{equation}
	has a unique solution $s(x)\in\F_q[x]$ if and only if $\alpha \in\F_{q^n}$ is a normal element. This implies that the inner sum in \eqref{eq7171CFNE.150C} collapses to 
	\begin{equation}\label{eq7171CFNE.150I}
		\sum _{0\leq t\leq q^n-1} e^{ \frac{i2\pi(\log_{\tau}s(x)\circ\eta-\log_{\tau}\alpha)t}{q^n}}=\sum _{0\leq t\leq q^n-1}1=q^n.
	\end{equation}
	This in turns reduces \eqref{eq7171CFNE.150C} to
	\begin{equation}\label{eq7171CFPE.150PDF9}
		\sum _{\substack{\deg s(x)\leq n-1\\\gcd (s(x),x^n-1)=1}} \frac{1}{q^n}\sum _{0\leq t\leq q^2-1} e^{ \frac{i2\pi(\log_{\tau}s(x)\circ\eta-\log_{\tau}\alpha)t}{q^n}}=1.
	\end{equation}

	Otherwise, the equation \eqref{eq7171CFNE.150F} has no solution $s(x)\in \mathcal{S}[x]=\{s(x)\in\F_q[x]:\gcd(s(x),x^n-1)=1 \text{ and }\deg s<n\}$.  Thus,  the inner sum in \eqref{eq7171CFNE.150C} collapses to 
	\begin{equation}\label{eq7171CFPE.150PDF5}
		\sum _{0\leq t\leq q^n-1} e^{ \frac{i2\pi(\log_{\tau}s(x)\circ\eta-\log_{\tau}\alpha)t}{q^n}}=0,
	\end{equation}
	this follows from the geometric series formula $\sum_{0\leq n\leq  x-1} r^n =(r^x-1)/(r-1)$, where $r=e^{i 2\pi a/p}\ne1$ and $x=q^n$. This in turns forces \eqref{eq7171CFNE.150C} to vanish.
	vanishes.
\end{proof}

Other versions of the divisorfree characteristic function for normal elements are possible.

%SSSSSSSSSSSSSSSSSSSSSSSSSSSSSSSSSSSSSSSSSSSSSSSSSSSSSSSSSSSSSSSSSSSSSSSSSSSSSSSS
%SSSSSSSSSSSSSSSSSSSSSSSSSSSSSSSSSSSSSSSSSSSSSSSSSSSSSSSSSSSSSSSSSSSSSSSSSSSSSSSS
%SSSSSSSSSSSSSSSSSSSSSSSSSSSSSSSSSSSSSSSSSSSSSSSSSSSSSSSSSSSSSSSSSSSSSSSSSSSSSSSS
%SSSSSSSSSSSSSSSSSSSSSSSSSSSSSSSSSSSSSSSSSSSSSSSSSSSSSSSSSSSSSSSSSSSSSSSSSSSSSSSS
%SSSSSSSSSSSSSSSSSSSSSSSSSSSSSSSSSSSSSSSSSSSSSSSSSSSSSSSSSSSSSSSSSSSSSSSSSSSSSSSS
%SSSSSSSSSSSSSSSSSSSSSSSSSSSSSSSSSSSSSSSSSSSSSSSSSSSSSSSSSSSSSSSSSSSSSSSSSSSSSSSS
%SSSSSSSSSSSSSSSSSSSSSSSSSSSSSSSSSSSSSSSSSSSSSSSSSSSSSSSSSSSSSSSSSSSSSSSSSSSSSSSS
\section{Intermediate Estimates and Evaluations} \label{S47171SNEFR-C}\hypertarget{S47171SNEFR-C}
The proof of \hyperlink{thm4343PNEFF.050}{Theorem} \ref{thm4343PNEFF.050} is broken up into four subsums. The estimates and evaluations of these four subsums are provided in  
\hyperlink{lem7171SNEFR.450-CaseI} {Lemma} \ref{lem7171SNEFR.450-CaseI} to \hyperlink{lem7171SNEFR.450-CaseIV} {Lemma} \ref{lem7171SNEFR.450-CaseIV} .
%LLLLLLLLLLLLLLLLLLLLLLLLLLLLLLLLLLLLLLLLLLLLLLLLLLLLLLLLLLLLLLLLLLLLLLLLLLLLLLLLLLLL
%LLLLLLLLLLLLLLLLLLLLLLLLLLLLLLLLLLLLLLLLLLLLLLLLLLLLLLLLLLLLLLLLLLLLLLLLLLLLLLLLLLLL
\begin{lem} \label{lem7171SNEFR.450-CaseI}\hypertarget{lem7171SNEFR.450-CaseI} Let $\tau\in \F_{q^n}$ be a fixed primitive normal element. If the element $\alpha$ is not a primitive normal element and $t_1=0$ and $t_2=0$, then
	\begin{eqnarray}\label{eq7171SNEFR.450-CaseI2}
		N_{00}(\mathcal{A})
		&=&\sum_{\alpha\in \mathcal{A}}\sum _{\substack{1\leq s\leq q^n-1\\\gcd (s,q^n-1)=1}} \frac{1}{q^n}\sum _{0\leq t_1\leq q^n-1} e^{\frac{i 2\pi(s-\log_{\tau}\alpha)t_1}{q^n}}\\[.3cm]
		&&\hskip 1.5 in \times\sum _{\substack{\deg s(x)\leq n-1\\\gcd (s(x),x^n-1)=1}} \frac{1}{q^n}\sum _{0\leq t_2\leq q^n-1} e^{ \frac{i2\pi(\log_{\tau}s(x)\circ\eta-\log_{\tau}\alpha)t_2}{q^n}}\nonumber\\[.3cm]
		&=&\frac{\varphi (q^n-1)}{q^{n}}\cdot \frac{\Phi (q^n-1)}{q^{n}}\cdot \# \mathcal{A}\nonumber	 .
	\end{eqnarray}	
\end{lem}

\begin{proof} Substitute the parameters $t_1=0$ and $t_2=0$ in \eqref{eq7171SNEFR.450-CaseI2} and assume that the element $\alpha\in \F_{q^n}$ is not a primitive normal element. These steps yield
	\begin{eqnarray}\label{7171SNEFR.450-CaseI4}
		N_{00}(\mathcal{A})
		&=&	\sum_{\alpha\in \mathcal{A}}\sum _{\substack{1\leq s\leq q^n-1\\\gcd (s,q^n-1)=1}} \frac{1}{q^n} \;\times\;\sum _{\substack{\deg s(x)\leq n-1\\\gcd (s(x),x^n-1)=1}} \frac{1}{q^n}\\[.3cm]
		&=&\frac{\varphi (q^n-1)}{q^{n}}\cdot \frac{\Phi (q^n-1)}{q^{n}}\cdot \# \mathcal{A}	\nonumber .
	\end{eqnarray}	
\end{proof}

%LLLLLLLLLLLLLLLLLLLLLLLLLLLLLLLLLLLLLLLLLLLLLLLLLLLLLLLLLLLLLLLLLLLLLLLLLLLLLLLLLLLL
%LLLLLLLLLLLLLLLLLLLLLLLLLLLLLLLLLLLLLLLLLLLLLLLLLLLLLLLLLLLLLLLLLLLLLLLLLLLLLLLLLLLL
\begin{lem} \label{lem7171SNEFR.450-CaseII}\hypertarget{lem7171SNEFR.450-CaseII} Let $\tau\in \F_{q^n}$ be a fixed primitive normal element. If the element $\alpha$ is not a primitive normal element and $t_1\in [0,p-1]$ and $t_2=0$, then
	\begin{eqnarray}\label{eq7171SNEFR.450-CaseII2}
		N_{01}(\mathcal{A})
		&=&\sum_{\alpha\in \mathcal{A}}\sum _{\substack{1\leq s\leq q^n-1\\\gcd (s,q^n-1)=1}} \frac{1}{q^n}\sum _{0\leq t_1\leq q^n-1} e^{\frac{i 2\pi(s-\log_{\tau}\alpha)t_1}{q^n}}\\[.3cm]
		&&\hskip 1.5 in \times\sum _{\substack{\deg s(x)\leq n-1\\\gcd (s(x),x^n-1)=1}} \frac{1}{q^n}\sum _{0\leq t_2\leq q^n-1} e^{ \frac{i2\pi(\log_{\tau}s(x)\circ\eta-\log_{\tau}\alpha)t_2}{q^n}}\nonumber\\[.3cm]
		&=&0\nonumber.
	\end{eqnarray}	
\end{lem}
\begin{proof} Substitute the parameters $t_1\in [0,p-1]$ and $t_2=0$ in \eqref{eq7171SNEFR.450-CaseII2} and assume that the element $\alpha\in \F_{q^n}$ is not a primitive normal element. These steps yield
	
	\begin{align}\label{eq7171PNEFR.550j1}
		N_{01}(\mathcal{A})
		&=\sum_{\alpha\in \mathcal{A}}\sum _{\substack{1\leq s\leq q^n-1\\\gcd (s,q^n-1)=1}} \frac{1}{q^n}\sum _{0\leq t_1\leq q^n-1} e^{\frac{i 2\pi(s-\log_{\tau}\alpha)t_1}{q^n}}\; \times\;\sum _{\substack{\deg s(x)\leq n-1\\\gcd (s(x),x^n-1)=1}} \frac{1}{q^n}\nonumber\\[.3cm]
		&=0,
	\end{align}	
	since the condition that $\alpha\in \F_{q^n}$ is a non primitive normal element implies that $\tr\left (\tau^r-\alpha\right )\ne0$ for all $r\in [1,q^n-1]$ such that $\gcd (r,q^n-1)=1$.
\end{proof}
%LLLLLLLLLLLLLLLLLLLLLLLLLLLLLLLLLLLLLLLLLLLLLLLLLLLLLLLLLLLLLLLLLLLLLLLLLLLLLLLLLLLL
%LLLLLLLLLLLLLLLLLLLLLLLLLLLLLLLLLLLLLLLLLLLLLLLLLLLLLLLLLLLLLLLLLLLLLLLLLLLLLLLLLLLL
\begin{lem} \label{lem7171SNEFR.450-CaseIII}\hypertarget{lem7171SNEFR.450-CaseIII} Let $\tau\in \F_{q^n}$ be a fixed primitive normal element. If the element $\alpha$ is not a primitive normal element and $t_1\in [0,p-1]$ and $t_2=0$, then
	\begin{eqnarray}\label{eq7171SNEFR.450-CaseIII2}
		N_{10}(\mathcal{A})
		&=&\sum_{\alpha\in \mathcal{A}}\sum _{\substack{1\leq s\leq q^n-1\\\gcd (s,q^n-1)=1}} \frac{1}{q^n}\sum _{0\leq t_1\leq q^n-1} e^{\frac{i 2\pi(s-\log_{\tau}\alpha)t_1}{q^n}}\\[.3cm]
		&&\hskip 1.5 in \times\sum _{\substack{\deg s(x)\leq n-1\\\gcd (s(x),x^n-1)=1}} \frac{1}{q^n}\sum _{0\leq t_2\leq q^n-1} e^{i2\pi \frac{\tr(s(x)\circ\tau-\alpha)t_2}{q^n}}\nonumber\\[.3cm]
		&=&0\nonumber,
	\end{eqnarray}	
\end{lem}
\begin{proof} Substitute the parameters $t_1=0$ and $t_2\in [0,p-1]$ in \eqref{eq7171PNEFR.550j1} and assume that the element $\alpha\in \F_{q^n}$ is not a primitive normal element. These steps yield
	
	\begin{align}\label{eq7171PNEFR.550k1}
		N_{10}(\mathcal{A})
		&=\sum_{\alpha\in \mathcal{A}}\sum _{\substack{1\leq s\leq q^n-1\\\gcd (s,q^n-1)=1}} \frac{1}{q^n}\times\sum _{\substack{\deg s(x)\leq n-1\\\gcd (s(x),x^n-1)=1}} \frac{1}{q^n}\sum _{0\leq t_2\leq q^n-1} e^{ \frac{i2\pi(\log_{\tau}s(x)\circ\eta-\log_{\tau}\alpha)t_2}{q^n}}\nonumber\\[.3cm]
		&=\sum_{\alpha\in \mathcal{A}}\sum _{\substack{1\leq r\leq q^n-1\\\gcd (r,q^n-1)=1}} \frac{1}{q^{n}}\sum _{\substack{\deg s(x)\leq n-1\\\gcd (s(x),x^n-1)=1}} \frac{1}{q^{n}}\sum _{0\leq t_2\leq q^n-1} e^{ \frac{i2\pi(\log_{\tau}s(x)\circ\eta-\log_{\tau}\alpha)t_2}{q^n}}\nonumber\\[.3cm]
		&=0,
	\end{align}	
	since the condition that $\alpha\in \F_{q^n}$ is a non primitive normal element implies that the trace $\tr(s(x)\circ\tau-\alpha)\ne0$ for all $\deg s(x)\leq n-1$ such that $\gcd (s(x),x^n-1)=1$.
\end{proof}

%LLLLLLLLLLLLLLLLLLLLLLLLLLLLLLLLLLLLLLLLLLLLLLLLLLLLLLLLLLLLLLLLLLLLLLLLLLLLLLLLLLLL
%LLLLLLLLLLLLLLLLLLLLLLLLLLLLLLLLLLLLLLLLLLLLLLLLLLLLLLLLLLLLLLLLLLLLLLLLLLLLLLLLLLLL
\begin{lem} \label{lem7171SNEFR.450-CaseIV}\hypertarget{lem7171SNEFR.450-CaseIV} Let $\tau\in \F_{q^n}$ be a fixed primitive normal element. If the element $\alpha$ is not a primitive normal element and $t_1\ne0$ and $t_2\ne0$, then
	\begin{eqnarray}\label{eq7171SNEFR.450-CaseIV2}
		N_{11}(\mathcal{A})
		&=&\sum_{\alpha\in \mathcal{A}}\sum _{\substack{1\leq s\leq q^n-1\\\gcd (s,q^n-1)=1}} \frac{1}{q^n}\sum _{0\leq t_1\leq q^n-1} e^{\frac{i 2\pi(s-\log_{\tau}\alpha)t_1}{q^n}}\\[.3cm]
		&&\hskip 1.5 in \times\sum _{\substack{\deg s(x)\leq n-1\\\gcd (s(x),x^n-1)=1}} \frac{1}{q^n}\sum _{0\leq t_2\leq q^n-1} e^{ \frac{i2\pi(\log_{\tau}s(x)\circ\eta-\log_{\tau}\alpha)t_2}{q^n}}\nonumber\\[.3cm]
		&=&O\left( 	\frac{\Phi (x^n-1)}{q^{n}}\cdot\# \mathcal{A}\cdot e^{-c\sqrt{\log q^n}}\right) \nonumber,
	\end{eqnarray}	
	where $c>0$ is a constant. 
\end{lem}

\begin{proof}[\textbf{Proof}] Substitute the parameters $t_1\ne0$ and $t_2\ne0$ and assume that the element $\alpha\in \F_{q^n}$ is not a primitive normal element. Then, summing over $t_2$ and simplifying yield
	
	\begin{eqnarray}\label{eqeq7171SNEFR.450-CaseIV12}
		N_{11}(\mathcal{A})
		&=&\sum_{\alpha\in \mathcal{A}}\sum _{\substack{1\leq s\leq q^n-1\\\gcd (s,q^n-1)=1}} \frac{1}{q^n}\sum _{1\leq t_1\leq q^n-1} e^{\frac{i 2\pi(s-\log_{\tau}\alpha)t_1}{q^n}}\\[.3cm]
		&&\hskip 1.5 in \times\sum _{\substack{\deg s(x)\leq n-1\\\gcd (s(x),x^n-1)=1}} \frac{1}{q^n}\sum _{1\leq t_2\leq q^n-1} e^{ \frac{i2\pi(\log_{\tau}s(x)\circ\eta-\log_{\tau}\alpha)t_2}{q^n}}\nonumber\\[.3cm]
		&=&\frac{1}{q^{2n}} \sum_{\alpha\in \mathcal{A},}\sum _{1\leq t_1\leq q^n-1}\sum _{\substack{1\leq s\leq q^n-1\\\gcd (s,q^n-1)=1}}  e^{\frac{i 2\pi(s-\log_{\tau}\alpha)t_1}{q^n}}\times\sum _{\substack{\deg s(x)\leq n-1\\\gcd (s(x),x^n-1)=1}}  (-1)\nonumber\\[.3cm]
		&=&-\frac{\Phi (x^n-1)}{q^{2n}} \sum_{\alpha\in \mathcal{A},}\sum _{1\leq t_1\leq q^n-1}\sum _{\substack{1\leq s\leq q^n-1\\\gcd (s,q^n-1)=1}}  e^{\frac{i 2\pi(s-\log_{\tau}\alpha)t_1}{q^n}}\nonumber.
	\end{eqnarray}
	Taking absolute value and applying \hyperlink{lem7770.300}{Lemma} \ref{lem7770.300} to the inner double finite sum yield
	\begin{eqnarray} \label{eq7171SNEFR.450-CaseIV10}
		|N_{11}(\mathcal{A})|
		&=&\bigg |-\frac{\Phi (x^n-1)}{q^{2n}} \sum_{\alpha\in \mathcal{A},}\sum _{1\leq t_1\leq q^n-1}\sum _{\substack{1\leq s\leq q^n-1\\\gcd (s,q^n-1)=1}}  e^{\frac{i 2\pi(s-\log_{\tau}\alpha)t_1}{q^n}}\bigg |\\[.3cm]
		&\leq & \frac{\Phi (x^n-1)}{q^{2n}} \sum_{\alpha\in \mathcal{A}}\bigg |\sum _{1\leq t_1\leq q^n-1}\sum _{\substack{1\leq s\leq q^n-1\\\gcd (s,q^n-1)=1}}  e^{\frac{i 2\pi(s-\log_{\tau}\alpha)t_1}{q^n}}\bigg |\nonumber\\[.3cm]
		&\ll & \frac{\Phi (x^n-1)}{q^{2n}} \sum_{\alpha\in \mathcal{A}}q^ne^{-c\sqrt{\log q^n}} \nonumber
		\\[.3cm]
		&\ll & \frac{\Phi (x^n-1)}{q^{n}}\cdot\# \mathcal{A}\cdot e^{-c\sqrt{\log q^n}} \nonumber,
	\end{eqnarray}
	where $c>0$ is a constant.
\end{proof}
%SSSSSSSSSSSSSSSSSSSSSSSSSSSSSSSSSSSSSSSSSSSSSSSSSSSSSSSSSSSSSSSSSSSSSSSSSSSSSSS
%SSSSSSSSSSSSSSSSSSSSSSSSSSSSSSSSSSSSSSSSSSSSSSSSSSSSSSSSSSSSSSSSSSSSSSSSSSSSSSS
%SSSSSSSSSSSSSSSSSSSSSSSSSSSSSSSSSSSSSSSSSSSSSSSSSSSSSSSSSSSSSSSSSSSSSSSSSSSSSSS
%SSSSSSSSSSSSSSSSSSSSSSSSSSSSSSSSSSSSSSSSSSSSSSSSSSSSSSSSSSSSSSSSSSSSSSSSSSSSSSS
%SSSSSSSSSSSSSSSSSSSSSSSSSSSSSSSSSSSSSSSSSSSSSSSSSSSSSSSSSSSSSSSSSSSSSSSSSSSSSSS
%SSSSSSSSSSSSSSSSSSSSSSSSSSSSSSSSSSSSSSSSSSSSSSSSSSSSSSSSSSSSSSSSSSSSSSSSSSSSSSS
%SSSSSSSSSSSSSSSSSSSSSSSSSSSSSSSSSSSSSSSSSSSSSSSSSSSSSSSSSSSSSSSSSSSSSSSSSSSSSSS
\section{Proof of the Main Theorem} \label{S4343PNEFF-T}\hypertarget{S4343PNEFF-T} 
The case $n=1$ is trivial since the map $f(x)=ax$ is one-to-one for every $a\ne0$. The case $n\geq2$ realizes the first nontrivial primitive element normal elements in quadratic or larger extension of a finite field $\F_q$ have both primitive elements and normal elements. Various partial results were achieved in \cite{CL1952} and \cite{DH1968}. The first complete result on the existence of primitive element normal elements in finite fields was proved in \cite{LS1987}. The first result on the existence of primitive normal elements in small subset is proved here. The basic proof presented here has the same structure of the earlier proofs, but uses two new characteristic functions for primitive elements and normal elements in $\F_{q^n}$ developed in \hyperlink{S7171CFPE-P}{Section} \ref{S7171CFPE-P} and \hyperlink{S7171CFNE-N}{Section} \ref{S7171CFNE-N}, respectively. This technique removes the dependence on the divisors of $q^n-1$ and the divisors of $x^n-1$, which plays a major role in the proofs given in \cite{CL1952}, \cite{DH1968}, \cite{LS1987}, et alia. \\

%SSSSSSSSSSSSSSSSSSSSSSSSSSSSSSSSSSSSSSSSSSSSSSSSSSSSSSSSSSSSSSSSSSSSSSSSSSSSSSS
%SSSSSSSSSSSSSSSSSSSSSSSSSSSSSSSSSSSSSSSSSSSSSSSSSSSSSSSSSSSSSSSSSSSSSSSSSSSSSSS
\subsection{Version I}
The characteristic functions for primitive elements and normal elements are defined by

\begin{equation}\label{eq4343PNEFF.500c2}
	\Psi (\alpha)=\left \{
	\begin{array}{ll}
		1 & \text{   \normalfont if } \ord_{q^n} (\alpha)=q^n-1,  \\[.3cm]
		0 & \text{   \normalfont if } \ord_{q^n} (\alpha)\ne q^n-1, \\
	\end{array} \right .
\end{equation}
and 
\begin{equation}\label{eq4343PNEFF.500c4}
	\Psi_q (\alpha)=	\left \{\begin{array}{ll}
		1 & \text{  \normalfont if } \Ord_q (\alpha)=x^n-1,  \\[.3cm]
		0 & \text{  \normalfont if } \Ord_q (\alpha)\neq x^n-1, \\
	\end{array} \right .
\end{equation}
respectively. 
Let $\mathcal{A}\subset \F_{q^n}$ be a nonstructured subset of elements and let 
\begin{equation}
	N_q(\mathcal{A})=\#\{\alpha \in \mathcal{A}:\ord_{q^n} (\alpha)=q^n-1 \text{ and }\Ord_q (\alpha)\neq x^n-1\}
\end{equation}
be the counting function for the number of primitive normal elements $\alpha\in \mathcal{A}$. \\

Various exponentially large subsets $\mathcal{S}\subset \F_{q^n}$ do not contain primitive normal elements, for example a proper subfield $\F_{q^d}\subset \F_{q^n}$, where $d\mid n$ and $d<n$. In view of this fact, some restrictions on the subset $\mathcal{A}\subset \F_{q^n}$ are required. There are many possible forms of nonstructured subsets, see \hyperlink{dfn7171CFNE.100S} {Definition} \ref{dfn7171CFNE.100S}. Some of these possibilities are \hyperlink{exa7171SBS.100Q1}{Example} \ref{exa7171SBS.100Q1}, \hyperlink{exa7171SBS.100Q2}{Example} \ref{exa7171SBS.100Q2}, \hyperlink{exa7171SBS.100Q3}{Example} \ref{exa7171SBS.100Q3}, $\mathcal{A}=\mathcal{A}_d(H)$ and the subset
\begin{align}\label{eq4343PNEFF.500k}
	\mathcal{A}_{d,r}&=\{\beta =a_0+a_1x+a_2x^2+\cdots+a_{d}x^{d}:\\
	&\hskip 1.05 in \gcd(a_0,a_1,\ldots,a_d)=1 
	\text{ and }a_0,\ldots, a_{r-1}\leq q^{\varepsilon/r}\}\nonumber,
\end{align}
where $r\geq 1$ and the $d-r$-tuple $(a_r,a_{b+1},\ldots,a_d)$ are fixed.

\begin{proof}[\textbf{\color{blue}Proof of \hyperlink{thm4343PNEFF.050}{Theorem} {\normalfont \ref{thm4343PNEFF.050}}}] Let $q$ be a prime power and let $n\geq2$. The hypothesis that the subset $\mathcal{A}\subset \F_{q^n}$ of cardinality $\# \mathcal{A}\ll (\log q^n)(\log\log q^n)^{1+\varepsilon}$ contain no primitive normal elements has the analytic expression
	
	\begin{equation}\label{eq4343PNEFF.500c6}
		N_q(\mathcal{A})
		=\sum_{\alpha\in \mathcal{A}}	\Psi (\alpha)	\Psi_q (\alpha)=0.
	\end{equation}
	Further, by definition the subset $\mathcal{A}$ is not a proper subfield nor a subspace of $\F_{q^n}$, see \eqref{eq4343PNEFF.500k}. Thus, the hypothesis \eqref{eq4343PNEFF.500c6} is not trivially true. Replacing the formulas of these weighted indicator functions, given in \hyperlink{lem7171CFPE.150PDF}{Lemma} \ref{lem7171CFPE.150PDF} and  \hyperlink{lem7171CFNE.150NDF}{Lemma} \ref{lem7171CFNE.150NDF}, respectively, yields
	\begin{align}\label{eq7171CNEFR.500e}
		N_q(\mathcal{A})
		&=\sum_{\alpha\in \mathcal{A}}	\Psi (\alpha)	\Psi_q (\alpha)\\[.2cm]
		&=\sum_{\alpha\in \mathcal{A}} \sum _{\substack{1\leq s\leq q^n-1\\\gcd (s,q^n-1)=1}} \frac{1}{q^n}\sum _{0\leq t_1\leq q^n-1} e^{\frac{i 2\pi(s-\log_{\tau}\alpha)t}{q^n}}\nonumber\\[.3cm]	
		&\hskip 1.75 in \times \sum _{\substack{\deg s(x)\leq n-1\\\gcd (s(x),x^n-1)=1}} \frac{1}{q^n}\sum _{0\leq t_2\leq q^n-1} e^{ \frac{i2\pi(\log_{\tau}s(x)\circ\eta-\log_{\tau}\alpha)t}{q^n}}\nonumber\\[.3cm]	
		&=N_{00}(\mathcal{A})\;+\;N_{01}(\mathcal{A})\;+\;N_{10}(\mathcal{A})\;+\;N_{11}(\mathcal{A})\nonumber.
	\end{align}
	The four subsums $N_{ij}(\mathcal{A})$ are determined by four proper subsets 
	\begin{equation}\label{eq7171CNEFR.500e5}
		\mathcal{T}_{ij}=\{(t_1,t_2)\}\subset[0,p-1]\times [0,p-1].
	\end{equation}
	These subsums are evaluated and estimated independently in \hyperlink{lem7171SNEFR.450-CaseI} {Lemma} \ref{lem7171SNEFR.450-CaseI} to \hyperlink{lem7171SNEFR.450-CaseIV} {Lemma} \ref{lem7171SNEFR.450-CaseIV}. Substituting these evaluations and estimates yield
	\begin{eqnarray}\label{eq7171PNEFR.550k11}
		N_{q}(\mathcal{A})
		&=&N_{00}(\mathcal{A})\;+\;N_{01}(\mathcal{A})\;+\;N_{10}(\mathcal{A})\;+\;N_{11}(\mathcal{A})\\[.3cm]
		&=&\frac{\varphi (q^n-1)}{q^{n}}\cdot\frac{\Phi (x^n-1)}{q^{n}}\cdot \# \mathcal{A}+0+0+O\left( \frac{\Phi (x^n-1)}{q^{n}}\cdot\# \mathcal{A}\cdot e^{-c\sqrt{\log q^n}}\right)\nonumber\\[.3cm]
		&=&\frac{\varphi (q^n-1)}{q^n}\frac{\Phi (x^n-1)}{q^n}\cdot\# \mathcal{A}\left (1+O\left ( \frac{q^n}{\varphi (x^n-1)}\cdot e^{-c\sqrt{\log q^n}}\right )\right )\nonumber.
	\end{eqnarray}	
	Replacing the lower bounds  
	\begin{equation}\label{eq7171PNEFR.550k15}
		\frac{1}{\log \log q^n}\ll \frac{\varphi(q^n-1)}{q^n}\ll 1
	\end{equation}
	see \hyperlink{lem9955P.400TL}{Lemma} \ref{lem9955P.400TL}, and 
	\begin{equation}\label{eq7171PNEFR.550k17}
		\frac{1}{\log q^n}\ll \frac{\Phi(q^n-1)}{q^n}\ll 1,
	\end{equation}
	see \hyperlink{cor1771TFDD.300FF-A}{Corollary} \ref{cor1771TFDD.300FF-A}, into \eqref{eq7171PNEFR.550k11} and any small nonstructured subset $\mathcal{A}\subset \F_{q^n}$ of cardinality $\# \mathcal{A}\gg (\log q^n)(\log\log q^n)^{1+\varepsilon}$, yield	\begin{eqnarray}\label{eq7171PNEFR.550k19}
		N_{q}(\mathcal{A})
		&=&\frac{\varphi (q^n-1)}{q^n}\cdot\frac{\Phi (x^n-1)}{q^n}\cdot\# \mathcal{A}\left (1+O\left ( \frac{q^n}{\varphi (x^n-1)}\cdot e^{-c_0\sqrt{\log q^n}}\right )\right )\nonumber\\[.3cm]
		&\gg&\frac{1}{(\log q^n)(\log\log q^{n})}\cdot\# \mathcal{A}
		\left (1+O\left ( \frac{\log\log q^n}{ e^{c_0\sqrt{\log q^n}}}\right )\right )\nonumber\\[.3cm]
		&\gg&(\log \log q^n)^{\varepsilon}
		\left (1+O\left ( \frac{\log\log q^n}{ e^{c_0\sqrt{\log q^n}}}\right )\right )\nonumber\\[.3cm]
		&\gg& 	1,
	\end{eqnarray}
	where $c_0>0$ is a constant. Clearly, as $q^n\to \infty$, in any characteristic $p\geq 2$, \eqref{eq7171PNEFR.550k19} contradicts the hypothesis \eqref{eq4343PNEFF.500c6}. Therefore, the small nonstructured subset $\mathcal{A}$ contains a primitive normal element.
\end{proof}

In synopsis, this result is effective and practical, for a fixed prime power $q=p^k\geq2^k$. In fact, any small nonstructured subset of cardinality
\begin{equation}\label{eq7171PNEFR.550k22}
 \# \mathcal{A}\gg (\log q^n)(\log\log q^n)^{1+\varepsilon},
\end{equation}
where $\varepsilon> 0$, as $n\to \infty$, contains a primitive normal element. This result is consistent with the high density of primitive normal elements encapsulated in the asymptotic formula

\begin{align}\label{eq7171PNEFR.550k24}
	N_{q}(\F_{q^n})
&=\frac{\varphi (q^n-1)}{q^n}\cdot\Phi (x^n-1)+O\left ( q^{n/2+\varepsilon}\right ),
\end{align}
proved in {\color{red}\cite[Theorem 2]{CL1952A}}. In fact, it can be viewed as a general version of \eqref{eq7171PNEFR.550k24} to any nonstructured subset or admissible subset $\mathcal{A}\subseteq \F_{q^n}$ that satisfies the inequality \eqref{eq7171PNEFR.550k22}.  

%SSSSSSSSSSSSSSSSSSSSSSSSSSSSSSSSSSSSSSSSSSSSSSSSSSSSSSSSSSSSSSSSSSSSSSSSSSSSSSS
%SSSSSSSSSSSSSSSSSSSSSSSSSSSSSSSSSSSSSSSSSSSSSSSSSSSSSSSSSSSSSSSSSSSSSSSSSSSSSSS
\subsection{Version II}
A sharper version of the main theorem, restricted to prime powers $q\geq8$, is included in this subsection.
%TTTTTTTTTTTTTTTTTTTTTTTTTTTTTTTTTTTTTTTTTTTTTTTTTTTTTTTTTTTTTTTTTT
%TTTTTTTTTTTTTTTTTTTTTTTTTTTTTTTTTTTTTTTTTTTTTTTTTTTTTTTTTTTTTTTTTT

\begin{thm} \label{thm4343PNEFF.050B}\hypertarget{thm4343PNEFF.050B}  Let $q=p^k$ be a prime power and $\varepsilon>0$ be a small number. Let 
$\mathcal{A}\subset \F_{q^n}$ be a nonstructured subset of cardinality 
\begin{equation}\label{eq4343PNEFF.050j10}
	\#\mathcal{A}\gg (\log\log q^n)^{2+\varepsilon}.
\end{equation}	
Then, as $q^n\to \infty$, the subset $\mathcal{A}$ contains primitive normal elements, uniformly for all $q\geq8$, $k\geq1$ and $n\geq2$. In particular, for large $q^n> q^2$, the weighted counting function for the number of primitive normal elements has the asymptotic lower bound   
\begin{equation} \label{eq4343PNEFF.050i13}
N_q(\mathcal{A})\gg(\log\log q^n)^{\varepsilon}.
	\end{equation}	
\end{thm}

\begin{proof}[\textbf{Proof}]The verification is the as same as the previous one, but uses \hyperlink{cor1771TFDD.300FF-C}{Corollary} \ref{cor1771TFDD.300FF-C} in \eqref{eq7171PNEFR.550k11}. 
\end{proof} 
%SSSSSSSSSSSSSSSSSSSSSSSSSSSSSSSSSSSSSSSSSSSSSSSSSSSSSSSSSSSSSSSS
%SSSSSSSSSSSSSSSSSSSSSSSSSSSSSSSSSSSSSSSSSSSSSSSSSSSSSSSSSSSSSSSS
%SSSSSSSSSSSSSSSSSSSSSSSSSSSSSSSSSSSSSSSSSSSSSSSSSSSSSSSSSSSSSSSS
%SSSSSSSSSSSSSSSSSSSSSSSSSSSSSSSSSSSSSSSSSSSSSSSSSSSSSSSSSSSSSSSS
%SSSSSSSSSSSSSSSSSSSSSSSSSSSSSSSSSSSSSSSSSSSSSSSSSSSSSSSSSSSSSSSS
%SSSSSSSSSSSSSSSSSSSSSSSSSSSSSSSSSSSSSSSSSSSSSSSSSSSSSSSSSSSSSSSS
%SSSSSSSSSSSSSSSSSSSSSSSSSSSSSSSSSSSSSSSSSSSSSSSSSSSSSSSSSSSSSSSS
\section{Research Problem --- Density of Primitive Normal Elements} \label{S2727SPRFR-DP}\hypertarget{S2727SPRFR-DP}
As the case for primitive roots, it seems that the asymptotic formula  \eqref{eq7171PNEFR.550k24} is actually the average number of primitive normal elements, and there is a correction constant associated to the individual prime characteristic $p$. The correction constant accounts for the interdependence of the property of being a primitive element and the property of being a normal element.\\

Define the sets $\mathcal{N}=\{\alpha\in\F_{q^n}: \alpha \text{ is a normal element }\}$ and
$\mathcal{P}=\{\alpha\in\F_{q^n}: \alpha \text{ is a primitive element}\}$. The density $\delta_q(\overline{\F_{q}})$ of primitive normal elements in the algebraic closure $\overline{\F_{q}}=\bigcup_{n\to\infty} \F_{q^n}$ is defined by 
\begin{equation} \label{eq2727SPRFR.900i1}
	\delta(\overline{\F_{q}})	=\lim_{n\to\infty}\frac{\#\{\alpha\in\F_{q^n}: \alpha \text{ is a primitive normal element }\} }{q^n}.
\end{equation}	
The asymptotic formula for the corresponding counting function should have the form
\begin{align}\label{eq2727SPRFR.900i2}
\#	\mathcal{N}\cap \mathcal{P}&=\{\alpha\in\F_{q^n}: \alpha \text{ is a primitive normal element }\} \\[.3cm]
	 &= \frac{\delta(\overline{\F_{q}})}{q^n} \cdot \varphi(q^n-1)\Phi(x^n-1)+O\left( q^{n/2+\varepsilon}\right) \nonumber,
\end{align}
where $\varepsilon>0$ is a small number. \\

Similar result is derived in {\color{red}\cite[Theorem 2]{CL1952A}}. In particular,
\begin{align}\label{eq2727SPRFR.900i3}
	\#	\mathcal{N}\cap \mathcal{P}&=\{\alpha\in\F_{q^n}: \alpha \text{ is a primitive normal element }\} \\[.3cm]
	&= \frac{1}{q^n} \cdot \varphi(q^n-1)\Phi(x^n-1)+O\left( q^{n/2+\varepsilon}\right) \nonumber.
\end{align}
It appear that \eqref{eq2727SPRFR.900i3} is the average number of primitive normal elements in a finite field $\F_{q^n}$ of random characteristic $\tchar p$ as $n\to\infty$. But, the density of primitive normal elements in a finite field $\F_{q^n}$ of specific $\tchar p$ as $n\to\infty$ has a correction factor as in \eqref{eq2727SPRFR.900i2}. \\

This seems to be similar to the case for the density $\delta(a)=A_1=0.3739558\ldots$, see \cite{SP1969}, of primes 
\begin{align}\label{eq2727SPRFR.900i4}
\pi_a(x)&=\{p\leq x: a \text{ is a primitive root} \bmod p\} \\[.3cm]
	&= \delta(a)\frac{x}{\log x}+O\left( \frac{x}{(\log x)^{2}}\right) \nonumber,
\end{align}
in the set of primes $\tP=\{2,3,5, \ldots\} $ with a random primitive root $a\ne\pm1,b^2$, and the density $(84/85)A_1\leq \delta(a)\leq (6/5)A_1$, see \cite{HC1967} and \cite{PS1995}, of primes
\begin{align}\label{eq2727SPRFR.900i5}
	\pi_a(x)&=\{p\leq x: a \text{ is a primitive root} \bmod p\} \\[.3cm]
	&= \delta(a)\frac{x}{\log x}+O\left( \frac{x}{(\log x)^{2}}\right) \nonumber,
\end{align}
in the set of primes $\tP=\{2,3,5, \ldots\} $ with a specific primitive root $a\ne\pm1,b^2$. A survey of this topic appears in \cite{SP2003}.

%PPPPPPPPPPPPPPPPPPPPPPPPPPPPPPPPPPPPPPPPPPPPPPPPPPPPPPPPPPPPPPPPPPPPPPPPPPPPPPPP
%PPPPPPPPPPPPPPPPPPPPPPPPPPPPPPPPPPPPPPPPPPPPPPPPPPPPPPPPPPPPPPPPPPPPPPPPPPPPPPPP
%PPPPPPPPPPPPPPPPPPPPPPPPPPPPPPPPPPPPPPPPPPPPPPPPPPPPPPPPPPPPPPPPPPPPPPPPPPPPPPPP
%PPPPPPPPPPPPPPPPPPPPPPPPPPPPPPPPPPPPPPPPPPPPPPPPPPPPPPPPPPPPPPPPPPPPPPPPPPPPPPPP
%PPPPPPPPPPPPPPPPPPPPPPPPPPPPPPPPPPPPPPPPPPPPPPPPPPPPPPPPPPPPPPPPPPPPPPPPPPPPPPPP
%PPPPPPPPPPPPPPPPPPPPPPPPPPPPPPPPPPPPPPPPPPPPPPPPPPPPPPPPPPPPPPPPPPPPPPPPPPPPPPPP
%PPPPPPPPPPPPPPPPPPPPPPPPPPPPPPPPPPPPPPPPPPPPPPPPPPPPPPPPPPPPPPPPPPPPPPPPPPPPPPPP
\section{Problems } \label{S7171CFNE-PPP}

%PPPPPPPPPPPPPPPPPPPPPPPPPPPPPPPPPPPPPPPPPPPPPPPPPPPPPPP
%PPPPPPPPPPPPPPPPPPPPPPPPPPPPPPPPPPPPPPPPPPPPPPPPPPPPPPP
%\section{Problems}\label{Prob9955K}
%\subsection{Totient Functions Problems}

%EEEEEEEEEEEEEEEEEEEEEEEEEEEEEEEEEEEEEEEEEEEEEEEEEEEEEEEEEEEEEEEEEEEEEEEEEEEEEEEEEEEEEEEE
\vskip .15in
\subsection{Primitive elements --- Totient function}
\begin{exe} \label{exe7171CFNE.010}\hypertarget{exe7171CFNE.010} {\normalfont  Let $q=p^k\geq 2$ be a prime power and let $r$ be the prime divisors of $q^n-1$. Show that
		$$\frac{\varphi (q^n-1)}{q^n}\prod _{r \mid q^n-1}\left (1+ \frac{1}{r-1}\right )=1 .$$
	} 
\end{exe}
%EEEEEEEEEEEEEEEEEEEEEEEEEEEEEEEEEEEEEEEEEEEEEEEEEEEEEEEEEEEEEEEEEEEEEEEEEEEEEEEEEEEEEEEE
\vskip .15in
\begin{exe} \label{exe7171CFNE.012}\hypertarget{exe7171CFNE.012} {\normalfont  Let $n\geq 1$ be an integer and let $d$ be the divisors of $q^n-1$. Show that 
		$$\sum_{d\mid q^n-1}\varphi(q^d-1)=q^n-1 .$$
	} 
\end{exe}

%EEEEEEEEEEEEEEEEEEEEEEEEEEEEEEEEEEEEEEEEEEEEEEEEEEEEEEEEEEEEEEEEEEEEEEEEEEEEEEEEEEEEEEEE
\vskip .15in

\subsection{Normal elements --- Totient function}
\begin{exe} \label{exe7171CFNE.014}\hypertarget{exe7171CFNE.012} {\normalfont  Let $n\geq 1$ be an integer and let $r(x)$ be the irreducible divisors of $x^n-1\in\F_q[x]$. Show that
		$$\frac{\Phi_q (x^n-1)}{q^n}\prod _{r(x) \mid x^n-1}\left (1+ \frac{1}{q^{\deg r}-1}\right )=1 .$$
	} 
\end{exe}

%EEEEEEEEEEEEEEEEEEEEEEEEEEEEEEEEEEEEEEEEEEEEEEEEEEEEEEEEEEEEEEEEEEEEEEEEEEEEEEEEEEEEEEEE
\vskip .15in

\begin{exe} \label{exe7171CFNE.014}\hypertarget{exe7171CFNE.012} {\normalfont  Let $n\geq 1$ be an integer and let $q$ be a prime power and let $x^n-1\in\F_q[x]$. If $n\nmid q-1$, show that
		$$\frac{\Phi(x^n-1)}{q^n-1}=\frac{q^n}{q^n-1}\prod_{d\mid n}\left(1-\frac{1}{q^{\ord_d q}} \right)^{\varphi(d)/\ord_d q}>C(n,q)$$
is 	bounded below by a constant $C(n,q)>0$.
} 
\end{exe}

%\subsection{Totient functions}
%EEEEEEEEEEEEEEEEEEEEEEEEEEEEEEEEEEEEEEEEEEEEEEEEEEEEEEEEEEEEEEEEEEEEEEEEEEEEEEEEEEEEEEEE
\vskip .15in
\subsection{Primitive Normal elements --- Totient functions}
Computational and Relationships\\
%EEEEEEEEEEEEEEEEEEEEEEEEEEEEEEEEEEEEEEEEEEEEEEEEEEEEEEEEEEEEEEEEEEEEEEEEEEEEEEEEEEEEEEEE
\vskip .15 in
\begin{exe} \label{exe7171CFNE.112-2}%\hypertarget{exe7171CFNE.112-2} 
	{\normalfont  Let $q\geq2$ be a prime power and let $n=2$. Show that \\
		
(i) $\displaystyle \Phi(x^2-1)=(q-1)^2,$\tabto{10cm}for all $q$.\\
		
		(ii) $\displaystyle \varphi(q^2-1)=(q^2-1)\prod_{\text{prime }r\mid q^2-1}(1-r^{-1}),$\tabto{10cm}for all $q$.\\
		
		(iii) Find a condition for which
		
		$$\displaystyle \varphi(q^2-1)=\Phi(x^2-1).$$
	} 
\end{exe}

%EEEEEEEEEEEEEEEEEEEEEEEEEEEEEEEEEEEEEEEEEEEEEEEEEEEEEEEEEEEEEEEEEEEEEEEEEEEEEEEEEEEEEEEE
\vskip .15 in
\begin{exe} \label{exe7171CFNE.112-2f}%\hypertarget{exe7171CFNE.112-2} 
	{\normalfont  Let $q=p^k$ be an odd prime power and let $n=2$. Show that the quadratic finite fields $\F_{q^2}$ contains more normal elements than primitive elements. More precisely,
		
		$$\displaystyle \varphi(q^2-1)\leq \Phi(x^2-1)$$
for all $p\geq5$.
	} 
\end{exe}

%EEEEEEEEEEEEEEEEEEEEEEEEEEEEEEEEEEEEEEEEEEEEEEEEEEEEEEEEEEEEEEEEEEEEEEEEEEEEEEEEEEEEEEEE
\vskip .15 in
\begin{exe} \label{exe7171CFNE.112-2k}%\hypertarget{exe7171CFNE.112-2} 
	{\normalfont  Let $q=p^k$ be an odd prime power and let $n=2$. Show that the average density of primitive normal elements in quadratic finite fields $\F_{q^2}$ has the asymptotic,

$$PN_2(q)=	\#	\mathcal{N}\cap \mathcal{P}= \frac{1}{q^2} \cdot \varphi(q^2-1)\Phi(x^2-1)+O\left( q^{1+\varepsilon}\right) .$$
The general version appears in \eqref{eq2727SPRFR.900i3} and proved in  {\color{red}\cite[Theorem 2]{CL1952A}}. In particular, the claim 

$PN_2(q) = \varphi(q^2-1)$ for every prime power $q$, in Proposition 13.1.1, Topics in Galois fields, Algorithms Comput. Math., 29
Springer, Cham, 2020, is incorrect.
	} 
\end{exe}

%EEEEEEEEEEEEEEEEEEEEEEEEEEEEEEEEEEEEEEEEEEEEEEEEEEEEEEEEEEEEEEEEEEEEEEEEEEEEEEEEEEEEEEEE
\vskip .15in
\begin{exe} \label{exe7171CFNE.112-3s}%\hypertarget{exe7171CFNE.112-3s} 
	{\normalfont  Let $q\geq2$ be odd and let $n=3$. Show that \\
		(i) $\displaystyle \Phi(x^3-1)=q^2(q-1),$\tabto{10cm}if $3\mid q$.\\
		
		(ii) $\displaystyle \varphi(q^3-1)=(q^3-1)\prod_{\text{prime }r\mid q^3-1}(1-r^{-1}),$\tabto{10cm}if $3\mid q$.\\
		
		(iii) Find a condition on the prime power $q\equiv 0\bmod 3$ for which
		
		$$\displaystyle \varphi(q^3-1)<\Phi(x^3-1).$$
		
		(iv) Find a condition on the prime power $q\equiv 0\bmod 3$ for which
		
		$$\displaystyle \varphi(q^3-1)=\Phi(x^3-1).$$	
		
		(v) Find a condition on the prime power $q\equiv 0\bmod 3$ for which
		
		$$\displaystyle \varphi(q^3-1)>\Phi(x^3-1).$$	
	} 
\end{exe}

%EEEEEEEEEEEEEEEEEEEEEEEEEEEEEEEEEEEEEEEEEEEEEEEEEEEEEEEEEEEEEEEEEEEEEEEEEEEEEEEEEEEEEEEE
\vskip .15in
\begin{exe} \label{exe7171CFNE.112-3a}%\hypertarget{exe7171CFNE.112-3a} 
{\normalfont  Let $q\geq2$ be odd and let $n=3$. Show that \\
(i) $\displaystyle \Phi(x^3-1)=(q-1)^3,$\tabto{10cm}if $3\mid q-1$.\\
		
	(ii) $\displaystyle \varphi(q^3-1)=(q^3-1)\prod_{\text{prime }r\mid q^3-1}(1-r^{-1}),$\tabto{10cm}if $3\mid q-1$.\\
	
(iii) Find a condition on the prime power $q\equiv 1\bmod 3$ for which
$$\displaystyle \varphi(q^3-1)<\Phi(x^3-1).$$

(iv) Find a condition on the prime power $q\equiv 1\bmod 3$ for which

$$\displaystyle \varphi(q^3-1)=\Phi(x^3-1).$$	

(v) Find a condition on the prime power $q\equiv 1\bmod 3$ for which

$$\displaystyle \varphi(q^3-1)>\Phi(x^3-1).$$	
	} 
\end{exe}

%EEEEEEEEEEEEEEEEEEEEEEEEEEEEEEEEEEEEEEEEEEEEEEEEEEEEEEEEEEEEEEEEEEEEEEEEEEEEEEEEEEEEEEEE
\vskip .15in
\begin{exe} \label{exe7171CFNE.112-3b}%\hypertarget{exe7171CFNE.112-3a} 
	{\normalfont  Let $q\geq2$ be odd and let $n=3$. Show that \\
		(i) $\displaystyle \Phi(x^3-1)=(q-1)(q^2-1),$\tabto{10cm}if $3\nmid q-1$.\\
		
		(ii) $\displaystyle \varphi(q^3-1)=(q^3-1)\prod_{\text{prime }r\mid q^3-1}(1-r^{-1}),$\tabto{10cm}if $3\nmid q-1$.\\
		
		(iii) Find a condition on the prime power $q\equiv 2\bmod 3$ for which		
		$$\displaystyle \varphi(q^3-1)<\Phi(x^3-1).$$
		
		(iv) Find a condition on the prime power $q\equiv 2\bmod 3$ for which
		
		$$\displaystyle \varphi(q^3-1)=\Phi(x^3-1).$$	
		
		(v) Find a condition on the prime power $q\equiv 2\bmod 3$ for which
		
		$$\displaystyle \varphi(q^3-1)>\Phi(x^3-1).$$	
	} 
\end{exe}

%EEEEEEEEEEEEEEEEEEEEEEEEEEEEEEEEEEEEEEEEEEEEEEEEEEEEEEEEEEEEEEEEEEEEEEEEEEEEEEEEEEEEEEEE
\vskip .15in

\begin{exe} \label{exe7171CFNE.016}\hypertarget{exe7171CFNE.012} {\normalfont  Let $n\geq 1$ be an integer and let $d(x)$ be the divisors of $x^n-1\in\F_q[x]$. Show that $$\sum_{d(x)\mid x^n-1}\Phi_q(d(x))=\Phi_q(x^n-1) .$$
	} 
\end{exe}
%EEEEEEEEEEEEEEEEEEEEEEEEEEEEEEEEEEEEEEEEEEEEEEEEEEEEEEEEEEEEEEEEEEEEEEEEEEEEEEEEEEEEEEEE
\vskip .15in

\begin{exe} \label{exe7171CFNE.018}\hypertarget{exe7171CFNE.012} {\normalfont  Let $n\geq 1$ be an integer and let $d(x)$ be the divisors of $x^n-1\in\F_q[x]$. Show that $$\sum_{\alpha\in \F_{q^n}}\sum_{d(x)\mid x^n-1}\#\{\Ord \alpha=d(x)\} =\Phi_q(x^n-1) .$$
	} 
\end{exe}
%EEEEEEEEEEEEEEEEEEEEEEEEEEEEEEEEEEEEEEEEEEEEEEEEEEEEEEEEEEEEEEEEEEEEEEEEEEEEEEEEEEEEEEEE
\vskip .15in
\subsection{Arithmetic functions identities}
\begin{exe} \label{exe7171CFNE.019}\hypertarget{exe7171CFNE.012} {\normalfont  Let $q$ be a prime power and let $n\geq 1$ be an integer. Determine whether or not the definition of the sum of divisors function $\sigma_q(x^n-1)=\sum_{d(x)\mid x^n-1}q^{\deg d(x)}$ over the ring $\F_q[x]$ leads to an analogous sigma-phi identity $\displaystyle \frac{\sigma(n)}{n} \frac{\varphi(n)}{n}=\prod_{p^v\mid \mid n}(1-p^{-(v+1)}) $ for the ratio $$\frac{\sigma_q(x^n-1)}{q^n-1} \frac{\Phi(x^n-1)}{q^n-1}\overset{?}{=}\prod_{r(x)^v\mid \mid x^n-1}\left( 1-\frac{1}{q^{\deg r(x)}}\right) ,$$
			where $r(x)\mid x^n-1$ ranges over the irreducible factors. 
	} 
\end{exe}

\subsection{Exponential sums}
%EEEEEEEEEEEEEEEEEEEEEEEEEEEEEEEEEEEEEEEEEEEEEEEEEEEEEEEEEEEEEEEEEEEEEEEEEEEEEEEEEEEEEEEE
\vskip .15in
\begin{exe} \label{exe7575CFNE.082}\hypertarget{exe7575CFNE.082} {\normalfont  Let $n\geq 1$ be an integer and let $r(x)$ be the irreducible divisors of $x^n-1\in\F_q[x]$. If $\psi(\alpha)$ is an additive character of order $\Ord_q \psi=r(x)$, verify that
		$$\sum _{\Ord(\psi ) = r(x)} \psi (\alpha)
		=
		\left \{\begin{array}{ll}
			q^{\deg r}-1 & \text{ if } \Ord_q (\alpha)\equiv 0\bmod r(x),  \\
			-1& \text{ if } \Ord_q (\alpha)\not\equiv 0\bmod r(x). \\
		\end{array} \right .\nonumber $$
	} 
\end{exe}
%EEEEEEEEEEEEEEEEEEEEEEEEEEEEEEEEEEEEEEEEEEEEEEEEEEEEEEEEEEEEEEEEEEEEEEEEEEEEEEEEEEEEEEEE
\vskip .15in
\subsection{Periodicity of characteristic functions}

\begin{exe} \label{exe8787CFNE.082}\hypertarget{exe8787CFNE.082} {\normalfont  Let $\Psi$ be the characteristic function of primitive elements $\alpha\in\F_{q^n}$. Verify that
		$$\Psi (\alpha)
		=\Psi (\alpha+q^n-1). $$
	} 
\end{exe}

\subsection{Convergence of Simple Sums } \label{exe9955S}

%EEEEEEEEEEEEEEEEEEEEEEEEEEEEEEEEEEEEEEEEEEEEEEEEEEEEEEEEEEEEEEE
\begin{exe} \label{exe9955.021} {\normalfont  Let $x\geq 1$ be a large number. Let $\mu(n)$ and $\varphi(n)$ be the Mobius and totient functions respectively. Show that 
		$$
		\sum_{n\leq x}\frac{1}{n\varphi(n)}\geq a_0>0, \qquad \text{ and } \qquad \sum_{n\leq x}\frac{\mu(n)}{n\varphi(n)}\geq a_1>0,
		$$
		where $a_0, a_1>0$ are constants. Hint: use the identity $n/\varphi(n)=\sum_{d\mid n}\mu(d)/d$.
	} 
\end{exe}
%EEEEEEEEEEEEEEEEEEEEEEEEEEEEEEEEEEEEEEEEEEEEEEEEEEEEEEEEEEEEEEE
\vskip .15 in
%EEEEEEEEEEEEEEEEEEEEEEEEEEEEEEEEEEEEEEEEEEEEEEEEEEEEEEEEEEEEEEE
\begin{exe} \label{exe9955.031} {\normalfont  Let $x\geq 1$ be a large number. Let $\mu(n)$ and $\varphi(n)$ be the Mobius and totient functions respectively. Show that 
		$$
		\sum_{n\geq x}\frac{1}{n\varphi(n)}=O\left (\frac{1}{x}\right ), \qquad \text{ and } \qquad \sum_{n\leq x}\frac{\mu(n)}{n\varphi(n)}=O\left (\frac{1}{x}\right ).
		$$
		
	} 
\end{exe}
%EEEEEEEEEEEEEEEEEEEEEEEEEEEEEEEEEEEEEEEEEEEEEEEEEEEEEEEEEEEEEEE
\vskip .15 in
%EEEEEEEEEEEEEEEEEEEEEEEEEEEEEEEEEEEEEEEEEEEEEEEEEEEEEEEEEEEEEEE

\begin{exe} \label{exe9955.030} {\normalfont  Let $x\geq 1$ be a large number. Let $\lambda(n)$ and $\varphi(n)$ be the Mobius and totient functions respectively. Show that 
		$$
		\sum_{n\leq x}\frac{1}{n\varphi(n)}\geq b_0>0, \qquad \text{ and } \qquad \sum_{n\leq x}\frac{\lambda(n)}{n\varphi(n)}\geq b_1>0,
		$$
		where $b_0, b_1>0$ are constants. Hint: use the identity $n/\varphi(n)=\sum_{d\mid n}\mu(d)/d$.
	} 
\end{exe}
\subsection{Equivalence Relations and Metric Spaces } \label{exe9955MS}

%EEEEEEEEEEEEEEEEEEEEEEEEEEEEEEEEEEEEEEEEEEEEEEEEEEEEEEEEEEEEEEE
\begin{exe} \label{exe9955MS.011} {\normalfont  Let $\F_q[x]$ be the ring of polynomials, and $d(r,s)=H(s-r)$ be the Hamming metric function. Prove \hyperlink{prop7171CFNE.200C}{Proposition} \ref{prop7171CFNE.200C}.
	} 
\end{exe}
%EEEEEEEEEEEEEEEEEEEEEEEEEEEEEEEEEEEEEEEEEEEEEEEEEEEEEEEEEEEEEEE
\vskip .15 in
%EEEEEEEEEEEEEEEEEEEEEEEEEEEEEEEEEEEEEEEEEEEEEEEEEEEEEEEEEEEEEEE
\begin{exe} \label{exe9955MS.012} {\normalfont  Let $\F_q[x]$ be the ring of polynomials, and $d(r,s)=H(s-r)$ be the Hamming metric function. Show that set inclusion is a nonsymmetric equivalent relation with respect to this metric. Specifically, let $A,B,C\subset \F_q[x]$. If $A\subset B$ and $B\subset C$ then $A\subset C$. But $A\subseteq B$ does not imply $B\subseteq A$.
	} 
\end{exe}
%EEEEEEEEEEEEEEEEEEEEEEEEEEEEEEEEEEEEEEEEEEEEEEEEEEEEEEEEEEEEEEE
\vskip .15 in

%EEEEEEEEEEEEEEEEEEEEEEEEEEEEEEEEEEEEEEEEEEEEEEEEEEEEEEEEEEEEEEE
\begin{exe} \label{exe9955MS.021} {\normalfont  Let $\F_q[x]$ be the ring of polynomials, and $d(r,s)=h(s-r)$ be the height metric function. Prove \hyperlink{prop7171CFNE.200D}{Proposition} \ref{prop7171CFNE.200D}.
	} 
\end{exe}
%EEEEEEEEEEEEEEEEEEEEEEEEEEEEEEEEEEEEEEEEEEEEEEEEEEEEEEEEEEEEEEE
\vskip .15 in
%EEEEEEEEEEEEEEEEEEEEEEEEEEEEEEEEEEEEEEEEEEEEEEEEEEEEEEEEEEEEEEE
\begin{exe} \label{exe9955MS.022} {\normalfont  Let $\F_q[x]$ be the ring of polynomials, and $d(r,s)=h(s-r)$ be the height metric function. Show that set inclusion is a nonsymmetric equivalent relation with respect to this metric. Specifically, let $A,B,C\subset \F_q[x]$. If $A\subset B$ and $B\subset C$ then $A\subset C$. But $A\subseteq B$ does not imply $B\subseteq A$.
	} 
\end{exe}
%EEEEEEEEEEEEEEEEEEEEEEEEEEEEEEEEEEEEEEEEEEEEEEEEEEEEEEEEEEEEEEE
\vskip .15 in
%BBBBBBBBBBBBBBBBBBBBBBBBBBBBBBBBBBBBBBBBBBBBBBBBBBBBBBBBBBBBBBBB
%BBBBBBBBBBBBBBBBBBBBBBBBBBBBBBBBBBBBBBBBBBBBBBBBBBBBBBBBBBBBBBBB
%BBBBBBBBBBBBBBBBBBBBBBBBBBBBBBBBBBBBBBBBBBBBBBBBBBBBBBBBBBBBBBBB
%BBBBBBBBBBBBBBBBBBBBBBBBBBBBBBBBBBBBBBBBBBBBBBBBBBBBBBBBBBBBBBBB
%%%%%%%%%%%%%%%%%%%%%%%%%%%%%% Bibliography%%%%%%%%%%%%%%%%%%%%%%
%\newpage
{\small

\end{document}